\documentclass[a4paper,11pt]{article}
\usepackage{amsmath, amsfonts, amscd, amssymb, amsthm, enumerate, xypic}

\newcommand{\Mat}{\operatorname{M}}
\newcommand{\Mats}{\operatorname{S}}
\newcommand{\Mata}{\operatorname{A}}

\newcommand{\GL}{\operatorname{GL}}
\newcommand{\Ker}{\operatorname{Ker}}

\newcommand{\Vect}{\operatorname{span}}
\newcommand{\im}{\operatorname{Im}}
\newcommand{\urk}{\operatorname{urk}}

\newcommand{\rk}{\operatorname{rk}}
\newcommand{\codim}{\operatorname{codim}}
\renewcommand{\setminus}{\smallsetminus}

\def\F{\mathbb{F}}
\def\K{\mathbb{K}}

\def\calH{\mathcal{H}}

\def\calJ{\mathcal{J}}

\def\calL{\mathcal{L}}

\def\calR{\mathcal{R}}
\def\calS{\mathcal{S}}
\def\calT{\mathcal{T}}
\def\calU{\mathcal{U}}
\def\calV{\mathcal{V}}
\def\calW{\mathcal{W}}
\def\calX{\mathcal{X}}

\def\lcro{\mathopen{[\![}}
\def\rcro{\mathclose{]\!]}}

\theoremstyle{definition}
\newtheorem{Def}{Definition}[section]
\newtheorem{Not}[Def]{Notation}

\theoremstyle{plain}
\newtheorem{theo}{Theorem}[section]
\newtheorem{prop}[theo]{Proposition}
\newtheorem{cor}[theo]{Corollary}
\newtheorem{lemma}[theo]{Lemma}
\newtheorem{claim}{Claim}

\theoremstyle{plain}

\theoremstyle{remark}
\newtheorem{Rems}{Remarks}
\newtheorem{Rem}[Rems]{Remark}

\title{Large spaces of bounded rank matrices revisited}
\author{Cl\'ement de Seguins Pazzis\footnote{Universit\'e de Versailles Saint-Quentin-en-Yvelines, Laboratoire de Math\'ematiques
de Versailles, 45 avenue des Etats-Unis, 78035 Versailles cedex, France}
\footnote{e-mail address: dsp.prof@gmail.com}}

\begin{document}

\thispagestyle{plain}

\maketitle

\begin{abstract}
Let $n,p,r$ be positive integers with $n \geq p\geq r$. A rank-$\overline{r}$ subset
of $n$ by $p$ matrices (with entries in a field) is a subset in which every matrix has  rank less than or equal to
$r$. A classical theorem of Flanders states that the dimension of a rank-$\overline{r}$ linear subspace must be less than or equal to $nr$,
and it characterizes the spaces with the critical dimension $nr$.
Linear subspaces with dimension close to the critical one were later studied by Atkinson, Lloyd and Beasley over
fields with large cardinality; their results were recently extended to all fields \cite{dSPclass}.

Using a new method, we obtain a classification of rank-$\overline{r}$ affine subspaces with large dimension, over all fields.
This classification is then used to double the range of (large) dimensions for which the structure of rank-$\overline{r}$ linear subspaces
is known for all fields.
\end{abstract}



\section{Introduction}

\subsection{The context}

Throughout the text, we fix an arbitrary field and denote it by $\K$.
Given non-negative integers $n$ and $p$, we denote by $\Mat_{n,p}(\K)$ the set of all matrices with $n$ rows, $p$ columns and entries in $\K$.
The rank of a matrix $M$ is denoted by $\rk(M)$, while the transpose of $M$ is denoted by $M^T$. We set $\Mat_n(\K):=\Mat_{n,n}(\K)$ and
we denote by $\GL_n(\K)$ the group of units of the ring $\Mat_n(\K)$.
We denote by $T_n^+(\K)$ the subspace of all upper-triangular matrices of $\Mat_n(\K)$.

Subsets $\calV$ and $\calW$ of $\Mat_{n,p}(\K)$ are called \textbf{equivalent} when there exist invertible matrices $P \in \GL_n(\K)$ and
$Q \in \GL_p(\K)$ such that $\calV=P\,\calW\,Q$ (in other words, $\calV$ and $\calW$ represent, in a different choice of bases,
the same set of linear transformations from a $p$-dimensional vector space to an $n$-dimensional vector space).

The \textbf{upper-rank} of a non-empty subset $\calV$ of $\Mat_{n,p}(\K)$, denoted by $\urk \calV$,
is defined as the maximal rank among the matrices of $\calV$.
Given a non-negative integer $r \in \lcro 0,\min(n,p)\rcro$,
a \textbf{rank-$\overline{r}$} subset of $\Mat_{n,p}(\K)$ is a subset $\calV$ such that $\urk \calV \leq r$.
A classical example of such subsets is the so-called compression spaces:
given integers $s \in \lcro 0,n\rcro$ and $t \in \lcro 0,p\rcro$, one defines
$$\calR(s,t):=\biggl\{\begin{bmatrix}
A & C \\
B & [0]_{(n-s) \times (p-t)}
\end{bmatrix} \mid A \in \Mat_{s,t}(\K), \; B \in \Mat_{n-s,t}(\K), \; C \in \Mat_{s,p-t}(\K)\biggr\}.$$
If $s+t \leq \min(n,p)$, then one checks that
$\calR(s,t)$ is a rank-$\overline{s+t}$ linear subspace of $\Mat_{n,p}(\K)$ with dimension $nt+s(p-t)$.
A rank-$\overline{r}$ \textbf{compression space} is a matrix subspace of $\Mat_{n,p}(\K)$ that is equivalent to
$\calR(s,t)$ for some non-negative integers $s$ and $t$ such that $s+t=r$ and $r \leq \min(n,p)$.
A subset $\calV$ of $\Mat_{n,p}(\K)$ is called \textbf{$r$-decomposable} when it
is included in a rank-$\overline{r}$ compression space: in terms of operators, this means that there are non-negative integers $s$ and $t$
such that $s+t=r$, a $(p-t)$-dimensional linear subspace $G$ of $\K^p$ and an $s$-dimensional linear subspace $H$ of $\K^n$ such that
every matrix of $\calV$ maps $G$ into $H$.

Of course, every subset of a rank-$\overline{r}$ subset is also a rank-$\overline{r}$ subset,
and every subset that is equivalent to a rank-$\overline{r}$ subset is a rank-$\overline{r}$ subset.
Thus, in trying to understand the structure of rank-$\overline{r}$ subspaces, one should focus on the equivalence classes of the
\emph{maximal} ones. It is easy to prove that if $s+t=r$, then
$\calR(s,t)$ is a maximal rank-$\overline{r}$ affine subspace and the compression spaces $\calR(i,r-i)$, for $i \in \lcro 0,r\rcro$,
are pairwise inequivalent. However, not every maximal rank-$\overline{r}$ linear subspace is
a compression space. A classical example is the one where $n$ is odd and greater than $1$, and where $p=n$ and $r=n-1$:
then, the space $\Mata_n(\K)$ of all \emph{alternating} $n$ by $n$ matrices is a maximal rank-$\overline{n-1}$ linear subspace
of $\Mat_n(\K)$ (see \cite{FillmoreLaurieRadjavi} for fields with more than $2$ elements, and \cite{dSPprimitiveF2} for fields with two elements);
yet it is easily checked that it is not $(n-1)$-decomposable.

Classifying the maximal rank-$\overline{r}$ subspaces is generally viewed as an intractable problem.
To get meaningful results, one needs to restrict the scope of the research.
One such possible restriction is to focus on small values of $r$ only:
solutions to this problem are known for $r \leq 3$ except for very small fields (see \cite{AtkinsonPrim}). For general values of $r$, another approach
is to focus on the so-called primitive subspaces \cite{AtkLloydPrim,EisenbudHarris}; this approach is generally well-suited
to classify rank-$\overline{r}$ spaces for small values of $r$, but it also has surprising connections
with the topic of large spaces of nilpotent matrices \cite{dSPAtkinsontoGerstenhaber}.
Finally, the most classical approach, which dates back to works of Dieudonn\'e \cite{Dieudonne} and Flanders \cite{Flanders}, consists in studying the rank-$\overline{r}$ subspaces with \emph{large} dimension: in this article, we shall follow this approach.

The basic result in the theory of large spaces of bounded rank matrices is the following one.

\begin{theo}[Flanders's theorem]\label{flanderstheo}
Let $n \geq p \geq r$ be non-negative integers.
Let $\calV$ be an affine subspace of $\Mat_{n,p}(\K)$ with $\urk \calV \leq r$. \\
Then,
$$\dim \calV \leq nr.$$
Moreover, if $\dim \calV=nr$ then:
\begin{enumerate}[(i)]
\item Either $\calV$ is equivalent to $\calR(0,r)$;
\item Or $n=p$ and $\calV$ is equivalent to $\calR(r,0)$;
\item Or $(n,p,r)=(2,2,1)$, $\# \K=2$ and $\calV$ is equivalent to the affine space
$$\calU_2(\K):=\biggl\{
\begin{bmatrix}
x & y \\
0 & x+1
\end{bmatrix} \mid (x,y)\in \K^2\biggr\}.$$
\end{enumerate}
\end{theo}

Actually, Flanders \cite{Flanders} only proved the above result for linear subspaces and under the assumption $\# \K>r$.
The rationale for his cardinality assumption stems from Flanders's use of polynomials that are constructed by
considering minors of matrices: the typical argument is to consider two matrices $A$ and $B$ in a linear subspace
$\calS$ and, if $A$ has rank $r$, to write that all the $r+1$ by $r+1$ minors of $A+tB$ are zero whatever the choice of $B$.
Then, by carefully choosing such minors, one obtains precious information on the shape of $B$ (typically, one basic information is
that $B$ maps the kernel of $A$ into the range of $A$). However, such methods are not suited to small finite
fields.

Interestingly, Dieudonn\'e \cite{Dieudonne} had established Theorem \ref{flanderstheo} earlier in the special case when $n=p$
and $r=n-1$ (i.e.\ he considered affine spaces of square singular matrices) for all fields.
It is only much later that Meshulam \cite{Meshulam} managed to remove Flanders's cardinality assumption for linear subspaces.
Later still, the generalization to affine spaces was achieved \cite{affpres}, and
even more recently the case of general division rings was encompassed \cite{dSPFlandersskew}.

Before we go on, it is important to discuss the relevance of considering general affine subspaces rather than
just linear subspaces. At first glance, this extension might seem gratuitous.
It is definitely the case that in almost every theorem dealing with spaces of matrices with rank conditions, only linear subspaces are considered.
Yet, in most of them the results would be dramatically different if one were to consider affine subspaces instead of linear subspaces.
A prime example is the one of subspaces of square matrices in which all the non-zero matrices are invertible: in it,
the maximal dimension is less than or equal to $n$ if one considers linear subspaces, and $\frac{n(n-1)}{2}$ if one considers affine subspaces. Now, in the case of our problem the results are not significantly different if one considers affine subspaces, and under
the popular cardinality assumption $\# \K>r$ they happen to be straightforward consequences of the results on linear subspaces!
Indeed, let $\calV$ be a rank-$\overline{r}$ affine subspace of $\Mat_{n,p}(\K)$ that does not contain the zero matrix,
denote by $V$ its translation vector space and assume that $\# \K>r$.
If an $(r+1)$-homogeneous polynomial on $\Mat_{n,p}(\K)$ vanishes everywhere on $\calV$ then it must vanish everywhere on $V$.
Applying this to all $r+1$ by $r+1$ minors, we deduce that all the matrices of $V$ have rank at most $r$,
and it follows that the linear subspace spanned by $\calV$ has upper-rank at most $r$.
Thus, under the assumption $\# \K>r$, rank-$\overline{r}$ affine subspaces are just affine subspaces of rank-$\overline{r}$ linear subspaces!
However, for small finite fields this result fails, as is demonstrated by the case when $\# \K=2$ and $\calV=\calU_2(\K)$.

Apart from the challenge that it poses, one might also wonder about the underlying motivation for this extension to affine subspaces.
Here is one, to start with: the first statement in Flanders's theorem can be restated as saying that
a linear subspace $\calV$ of $\Mat_{n,p}(\K)$ contains at least one matrix with rank greater than $r$ provided that its dimension is greater than $nr$.
Now, a natural extension is to ask ``how many" matrices with large rank can be expected to be found in such a large subspace.
In terms of linear algebra, a natural way of formulating that problem is to look at the \emph{span} of the matrices with rank more than $r$
in $\calV$. This is where affine subspaces come into play: if that span is not the whole of $\calV$, then
it must be included in a linear hyperplane $H$ of $\calV$; then, by choosing a matrix $A \in \calV \setminus H$,
we construct the affine hyperplane $A+H$ of $\calV$; this hyperplane does not go through zero, and it has upper-rank at most $r$.
With that line of reasoning, one can derive from Flanders's theorem that if a linear subspace $\calV$ of $\Mat_{n,p}(\K)$ has dimension greater than $nr$, then it is spanned by its matrices with rank greater than $r$, unless $n=p=2$, $r=1$, $\# \K=2$ and
$\calV$ is equivalent to $T_2^+(\K)$ (see Propositions 1 and 2 from \cite{dSPlargerank}).

Understanding when large spaces are spanned by their matrices of large rank was the reason that got us
interested in affine subspaces in the first place. Yet, by working on the topic, we slowly came to realize that
\emph{enlarging the discussion to affine subspaces is fundamental to the understanding of linear subspaces over finite fields.}

\vskip 4mm
Let us come back to our general problem.
In Flanders's theorem, rank-$\overline{r}$ spaces with the critical dimension $nr$ are fully understood.
Could it be that every rank-$\overline{r}$ space $\calV$ whose dimension is close enough to the critical
one is equivalent to a subspace of $\calR(r,0)$ or of $\calR(0,r)$?
If so, what is the optimal lower bound on its dimension for $\calV$ to have that property?

The question was partly answered as follows in the nineteen eighties.

\begin{theo}[Atkinson-Lloyd-Beasley]
Let $n \geq p \geq r$ be positive integers. Assume that $\# \K>r$.
Let $V$ be a rank-$\overline{r}$ linear subspace of $\Mat_{n,p}(\K)$
such that
$$\dim V \geq nr-(n-p+r)+1.$$
Then, $V$ is equivalent to a subspace of one of the spaces $\calR(r,0)$, $\calR(r-1,1)$, $\calR(1,r-1)$ or $\calR(0,r)$.
\end{theo}

The square case was achieved by Atkinson and Lloyd \cite{AtkLloyd} and was later used by Beasley \cite{Beasley}
to derive the general case. In the works of those authors, the cardinality assumption has the same rationale
as in Flanders's, their proofs being based upon Flanders's line of reasoning.
The case of an arbitrary field remained an open problem for over twenty years, in part due to the existence of a counter-example
that appeared in \cite{Meshulam}: if $\# \K=2$, the space
$$\calJ_3(\K):=\left\{\begin{bmatrix}
a & c & d \\
0 & b & e \\
0 & 0 & a+b
\end{bmatrix} \mid (a,b,c,d,e)\in \K^5\right\}$$
of all upper-triangular trace-zero matrices over $\K$ is a rank-$\overline{2}$ space but it is easy to check
that it is not $2$-decomposable.
The cardinality assumption was finally removed in \cite{dSPclass}, where we proved that the example from Meshulam's article
is the only exception up to equivalence. Below is the precise result.

\begin{theo}[de Seguins Pazzis]
Let $n \geq p \geq r$ be positive integers.
Let $V$ be a rank-$\overline{r}$ linear subspace of $\Mat_{n,p}(\K)$
such that
$$\dim V \geq nr-(n-p+r)+1.$$
Then:
\begin{enumerate}[(a)]
\item Either $V$ is equivalent to a subspace of one of the spaces $\calR(r,0)$, $\calR(r-1,1)$, $\calR(1,r-1)$ or $\calR(0,r)$;
\item Or $(n,p,r)=(3,3,2)$, $\# \K=2$ and $V$ is equivalent to $\calJ_3(\K)$.
\end{enumerate}
\end{theo}

In the last section of \cite{dSPclass}, the basic techniques from Flanders's proof were shown to yield
an extension of the Atkinson-Lloyd-Beasley theorem towards lower dimensions.
The range of dimensions for which the structure of rank-$\overline{r}$ spaces is known
is essentially doubled under the assumption $\# \K>r$.

\begin{theo}[de Seguins Pazzis]\label{dSPdouble}
Let $n \geq p \geq r \geq 2$ be positive integers. Assume that $\# \K>r$.
Let $V$ be a rank-$\overline{r}$ linear subspace of $\Mat_{n,p}(\K)$
such that
$$\dim V \geq nr-2(n-p+r)+4.$$
Then, $V$ is equivalent to a subspace of $\calR(i,r-i)$ for some $i \in \{0,1,2,r-2,r-1,r\}$.
\end{theo}

Note that the lower bound $nr-2(n-p+r)+4$ in this theorem is exactly the dimension of the $\calR(2,r-2)$ compression space.

Considering the above theorem, the natural question to ask is how low the bound on the dimension of $V$ can be taken so as to ensure
that $V$ is $r$-decomposable. We were tempted at some point to think that a good lower bound could be the minimal dimension among the rank-$\overline{r}$ compression spaces (at least in the case of square matrices). Yet, this is untrue for large values of $r$, as we shall
now demonstrate. To see this, we give a general construction of very large spaces of matrices with upper-rank $r$ less than $p$
and that are not $r$-decomposable. Following Atkinson and Lloyd \cite{AtkLloydPrim}, let us consider an arbitrary non-negative integer $s>0$
together with a linear subspace $W$ of singular matrices of $\Mat_s(\K)$. Let $n \geq p \geq s$.
Then, the space $W \vee \Mat_{n-s,p-s}(\K)$ of all matrices of the form
$$\begin{bmatrix}
A & B \\
[0]_{(n-s) \times s} & C
\end{bmatrix} \quad \text{with $A \in W$, $B \in \Mat_{s,p-s}(\K)$ and $C \in \Mat_{n-s,p-s}(\K)$}$$
is a rank-$\overline{p-1}$ linear subspace, and it is easy to show that if it were
$(p-1)$-decomposable then $W$ would also be $(s-1)$-decomposable.
A classical example is the one where $W=\Mata_3(\K)$, the space of all $3$ by $3$ alternating matrices with entries in $\K$:
it is a rank-$\overline{2}$ subspace of $\Mat_3(\K)$ with dimension $3$, but it not included in a rank-$\overline{2}$ compression space.
If $n \geq p \geq 3$, then $\Mata_3(\K) \vee \Mat_{n-3,p-3}(\K)$ is an example of a rank-$\overline{p-1}$ subspace of $\Mat_{n,p}(\K)$
that is not $(p-1)$-decomposable and yet has
dimension
$$3+(p-3)n=n(p-1)-2(n-p+(p-1))+1.$$
Note how close this dimension is to the critical one in Theorem \ref{dSPdouble}.
Moreover, if $p$ and $n$ are large enough it is easy to check that $3+(p-3)n>\dim \calR(3,p-4)$, which shows
how wrong the naive conjecture is, even for square matrices.
For $\F_2$, we can take $\calJ_3(\F_2)$ instead of $\Mata_3(\F_2)$, and then we obtain a rank-$\overline{p-1}$ subspace
of $\Mat_{n,p}(\K)$ that is not $(p-1)$-decomposable; yet, it has dimension
$$n(p-1)-2(n-p+(p-1))+3.$$

Thus, as far as compression spaces are concerned, Theorem \ref{dSPdouble} is very close to optimality,
at least when $r=p-1$, and it is on the whole well-suited to values of $r$ that are large with respect to $p$.
When $r$ is very small with respect to $p$ (below some bound of the order of magnitude of
$\sqrt{p}$), and provided that the field $\K$ has large cardinality,
another theorem of Atkinson and Lloyd \cite[Theorem 2]{AtkLloydPrim} states that the largest dimension for
a rank-$\overline{r}$ linear subspace of $\Mat_{n,p}(\K)$ which is not $r$-decomposable is $n(r-2)+3$.

\subsection{Main results}

In short, our aim here is to generalize Theorem \ref{dSPdouble} to arbitrary fields, by using a
new strategy.

There are three main stages. The first one consists more or less of an extension of Atkinson and Lloyd's classification
theorem to affine subspaces over arbitrary fields. We coin it as the \textbf{first classification theorem}.

\begin{theo}[First classification theorem]\label{firstclasstheo}
Let $\calS$ be an affine subspace of $\Mat_{n,p}(\K)$ in which every matrix has rank at most $r\geq 1$.
Assume that $\dim \calS > nr-(n-p+r)+1$.
Then, $\calS$ is $r$-decomposable.
More precisely:
\begin{enumerate}[(a)]
\item Either $\calS$ is equivalent to a subspace of $\calR(0,r)$;
\item Or $n=p$ and $\calS$ is equivalent to a subspace of $\calR(r,0)$.
\end{enumerate}
\end{theo}

Note that Atkinson and Lloyd's theorem is slightly more interesting as it takes into account the critical case when
$\dim \calS=nr-(n-p+r)+1$. Over $\F_2$, their result fails for affine subspaces, not only because of the $\calJ_3(\F_2)$
example: for all $p \geq 2$ and $n \geq p$, we see that $\calU_2(\F_2) \vee \Mat_{n-2,p-2}(\K)$
is a rank-$\overline{p-1}$ affine subspace with the critical dimension $n(p-1)-(n-p+(p-1))+1$ but it is
not equivalent to a subspace of $\calR(r,0)$ or of $\calR(0,r)$, and it is not a compression space since it is not a linear subspace.

In the next step, we shall slightly extend the above theorem for fields with more than $2$ elements by
allowing two additional dimensions below the maximal one.
There is however a counter-example for $\F_3$: if $\# \K=3$, the
affine space
$$\calU_3(\K):=\left\{\begin{bmatrix}
x & a & b \\
0 & x+1 & c \\
0 & 0 & x-1
\end{bmatrix} \mid (x,a,b,c)\in \K^4\right\}$$
has upper-rank $2$ (obviously $\forall x \in \K, \; x(x+1)(x-1)=0$)
and it has the critical codimension $nr-(n-p+r)$ (here $n=p=3$ and $r=2$), but it is not $2$-decomposable
since otherwise its translation vector space would consist
of singular matrices, contradicting the obvious fact that this translation vector space contains the identity matrix $I_3$!

We shall coin our second result as the \textbf{refined first classification theorem}.

\begin{theo}[Refined first classification theorem]
Let $n,p,r$ be non-negative integers such that $n \geq p \geq r$.
Let $\calS$ be a rank-$\overline{r}$ affine subspace of $\Mat_{n,p}(\K)$.
Assume that $\dim \calS \geq nr-(n-p+r)$ and $\# \K>2$.
Then:
\begin{enumerate}[(a)]
\item Either $\calS$ is $r$-decomposable;
\item Or $\# \K=3$ and $\calS$ is equivalent to $\calU_3(\K)$.
\end{enumerate}
\end{theo}

\begin{Rem}\label{1strefinedremark}
By computing dimensions, we can be more precise as to what
kinds of $r$-decomposable spaces are possible in the above theorem.
The main idea is that, if $n,p,r$ are fixed with $r \geq 2$, the function $s \mapsto \dim \calR(s,r-s)$
is a polynomial of degree $2$ with positive coefficient on $s^2$, and hence it is strictly convex.

Assuming that $\calS$ is $r$-decomposable and that $\dim \calS \geq nr-(n-p+r)$:
\begin{enumerate}[(a)]
\item If $n>p+2$ then $\calS$ is equivalent to a subspace of $\calR(0,r)$
or of $\calR(1,r-1)$.
\item If $n=p+2$, then $\calS$ is equivalent to a subspace of $\calR(0,r)$ or of $\calR(1,r-1)$, or
$r=2$ and $\calS$ is equivalent to $\calR(2,0)$.
\item If $n=p+1$, then $\calS$ is equivalent to a subspace of $\calR(0,r)$, of $\calR(1,r-1)$ or of $\calR(r,0)$,
or $r=3$ and $\calS$ is equivalent to $\calR(2,1)$.
\item If $n=p$, then $\calS$ is equivalent to a subspace of $\calR(s,r-s)$ for some $s \in \{0,1,r-1,r\}$,
or $r=4$ and $\calS$ is equivalent to $\calR(2,2)$.
\end{enumerate}
Moreover, if $\dim \calS> nr-(n-p+r)$ and $\calS$ is equivalent to a subspace of $\calR(1,r-1)$ or of $\calR(r-1,1)$,
then $\calS$ is equivalent to $\calR(1,r-1)$ or to $\calR(r-1,1)$, and in the latter case $n=p$. \\
Finally, if $\dim \calS> nr-(n-p+r)$, $n=p+1$ and $\calS$ is equivalent to a subspace of $\calR(r,0)$, then
it is equivalent to $\calR(r,0)$.
\end{Rem}

The first classification theorem and its refinement will be used to
obtain our ultimate result, which we coin as the \textbf{second classification theorem}.

\begin{Not}
We set
$$\epsilon(\K):=\begin{cases}
0 & \text{if $\# \K>2$} \\
2 & \text{if $\# \K=2$.}
\end{cases}$$
\end{Not}

\begin{theo}[Second classification theorem]\label{secondclasstheo}
Let $S$ be a linear subspace of $\Mat_{n,p}(\K)$ in which every matrix has rank at most $r$.
Assume that $\dim S \geq nr-2(n-p+r)+2+\epsilon(\K)$.
Then:
\begin{itemize}
\item Either $S$ is $r$-decomposable;
\item Or $(n,p,r)=(4,4,3)$, $\# \K=3$ and $S$ is equivalent to the space
$$\calU_4(\K):=\left\{\begin{bmatrix}
x & a & b & c \\
0 & y & d & e \\
0 & 0 & x+y & f \\
0 & 0 & 0 & x-y
\end{bmatrix} \mid (a,b,c,d,e,f,x,y)\in \K^8\right\}.$$
\end{itemize}
\end{theo}

If $\# \K=3$, then we can use the identity
$$\forall (x,y)\in \K^2, \; xy(x+y)(x-y)=x^3y-xy^3=xy-xy=0$$
to see that $\calU_4(\K)$ has upper-rank $3$, and on the other hand it has dimension $8$.
In that case $(n,p,r)=(4,4,3)$ and we compute that $nr-2(n-p+r)+2=8$. Finally, $\calU_4(\K)$ is
not $3$-decomposable (otherwise, this would be true over any extension of $\K$, and obviously $\calU_4(\K)$ contains invertible matrices if $\# \K>3$).

In contrast with the first classification theorem and its refined version,
this second classification theorem deals only with linear subspaces. Extending it to affine subspaces seems very difficult
(for reasons that we have already explained, it fails for $\F_2$, to start with).

\begin{Rem}\label{2ndremark}
Again, we can be more precise as to what specific kinds of $r$-decomposable spaces are possible in the above theorem.

Assuming that $S$ is $r$-decomposable and that
$\dim S \geq nr-2(n-p+r)+2+\epsilon(\K)$:
\begin{enumerate}[(a)]
\item If $n>p+2$, then $S$ is equivalent to a subspace of $\calR(0,r)$, $\calR(1,r-1)$ or $\calR(2,r-2)$,
or $r \leq 2+\frac{2}{n-(p+2)}$ and $S$ is equivalent to a subspace of $\calR(r,0)$,
or $r \leq 3+\frac{2}{n-(p+1)}$ and $S$ is equivalent to a subspace of $\calR(r-1,1)$.
\item If $n=p+2$,  then $S$ is equivalent to a subspace of $\calR(0,r)$, $\calR(1,r-1)$, $\calR(2,r-2)$ or $\calR(r,0)$,
or $r=5$ and $S$ is equivalent to $\calR(3,2)$ or $\calR(4,1)$, or $r=4$ and $S$ is equivalent to a subspace of $\calR(3,1)$.
\item If $n=p+1$, then $S$ is equivalent to a subspace of $\calR(0,r)$, $\calR(1,r-1)$, $\calR(2,r-2)$, $\calR(r,0)$ or $\calR(r-1,1)$,
or $r=5$ and $S$ is equivalent to a subspace of $\calR(3,2)$ with codimension at most $1$,
or $r=6$ and $S$ is equivalent to $\calR(3,3)$ or to $\calR(4,2)$.
\item If $n=p$, then $S$ is equivalent to a subspace of $\calR(s,r-s)$ for some $s \in \{0,1,2,r-2,r-1,r\}$,
or $r=6$ and $S$ is equivalent to a subspace of $\calR(3,3)$ with codimension at most $1$,
or $r=7$ and $S$ is equivalent to $\calR(4,3)$ or $\calR(3,4)$.
\end{enumerate}
\end{Rem}

\subsection{Main method}\label{methodsection}

Now, we can explain the main strategy of our proof.
We shall use a new method that consists of a mixture of some elements of Dieudonn\'e's proof \cite{Dieudonne},
some ideas from \cite{dSPclass}, and new insights (some of which have been laid out in \cite{dSPFlandersskew}).
As the strategy will be used no less than three times, an explanation of its main components is in order.

The first thing to say is that our method consists in performing an induction over all of $n,p,r$.
Let $\calV$ be a rank-$\overline{r}$ affine subspace, with translation vector space denoted by $V$.
We shall take a close look at the rank $1$ matrices in $V$.
More precisely, we are interested in the intersection of $V$ with the maximal linear subspaces of rank $1$ matrices.

The method actually splits into two sub-methods.
In the first one, which we label as the ``\textbf{Erase one row and column}" (in short: \textbf{ERC}) method,
we assume that $V$ contains a rank $1$ matrix, say $E_{1,1}$.
Then, we can split every matrix $M$ of $\calV$ up as
$$M=\begin{bmatrix}
? & [?]_{1 \times (p-1)} \\
[?]_{(n-1) \times 1} & K(M)
\end{bmatrix} \quad \text{with $K(M) \in \Mat_{n-1,p-1}(\K)$.}$$
Then, it is easy to check (see Lemma \ref{extractionlemma}) that $K(\calV)$ is a rank-$\overline{r-1}$ affine subspace of $\Mat_{n-1,p-1}(\K)$.
On the other hand, the rank theorem shows that
$$\dim K(\calV) \geq \dim \calV-(p-1)-\dim U,$$
where $U$ denotes the vector space of all matrices of $V$ whose last $p-1$ columns equal zero.
If the dimension of $U$ is small enough, then we can hope to apply the induction hypothesis to the space
$K(\calV)$.

If $K(\calV)$ is included in the compression space $\calR(i,r-1-i)$, then $\calV$ is included in
$\calR(i+1,r-i)$, which is not a rank-$\overline{r}$ compression space in general. However,
with extra work, it can be shown in specific cases that $\calV$ is actually equivalent to a subspace of
$\calR(i,r-1)$ or $\calR(i+1,r-i-1)$. A key feature of this \textbf{lifting process} resides in the
use of earlier theorems for \emph{affine subspaces}: thus, in lifting for the first classification theorem, one uses
Flanders's theorem, in lifting for the refined first classification theorem, one uses the first classification theorem,
and in lifting for the second classification theorem, one uses the refined first classification theorem!

To see things in a more general manner, we shall employ the following notation:

\begin{Not}
Let $S$ be a subset of $\Mat_{n,p}(\K)$, $D$ be a $1$-dimensional linear subspace of $\K^n$, and
$H$ be a linear hyperplane of $\K^p$. We define
$$S_H:=\{M \in S : \; H \subset \Ker M\}$$
and
$$S^D:=\{M \in S : \; \im M \subset D\}.$$
\end{Not}

Thus, to apply the ERC method,
we need to find a linear hyperplane $H$ of $\K^p$ such that $V_H$ is non-zero and with small dimension (in the above we have considered the special case where $H=\{0\} \times \K^{p-1}$),
or we need to find a $1$-dimensional linear subspace $D$ of $\K^n$ such that $V^D$ is non-zero and with small dimension.

\vskip 3mm
The second sub-method, which we coin as the ``Erase one column" method (in short: \textbf{EC method}),
deals with the case when we have a linear hyperplane $H$ of $\K^p$ such that $S_H=\{0\}$,
say $H=\{0\} \times \K^{p-1}$.
Then, we split every matrix $M$ of $\calV$ up as
$$M=\begin{bmatrix}
[?]_{n \times 1} & H(M)
\end{bmatrix} \quad \text{with $H(M) \in \Mat_{n,p-1}(\K)$.}$$
Obviously, $H(\calV)$ is an affine subspace of $\Mat_{n,p-1}(\K)$ with the same dimension as $\calV$
and upper-rank less than or equal to $r$. By induction on $p$,
we can retrieve information on the structure of $H(\calV)$.
We are then confronted with a lifting problem that is somewhat similar to the one in the ERC method.
However, in that case we have to consider the extra case when $r=p-1$.
Then, we have the so-called \textbf{special lifting problem}:
there is an affine map $F : H(\calV) \rightarrow \K^n$ such that
$$\forall M \in \calV, \quad M=\begin{bmatrix}
F(H(M)) & H(M)
\end{bmatrix}.$$
In most situations, it will then be possible to prove that $F$ is range-compatible
(or, worse, quasi-range-compatible, see Section \ref{RCsection}),
i.e.\ $F$ maps every matrix of $H(\calV)$ to a vector of its range.
Using the classification of range-compatible maps over large matrix spaces
\cite{dSPRC3,dSPRC1}, it is then possible except in very specific situations -- in which the structure of $H(\calV)$ is already fairly simple --
to find a non-zero vector that is annihilated by all the matrices in $\calV$, which yields
that $\calV$ is equivalent to a subspace of $\calR(0,r)$.

Transposing the EC method, we get the ER method (``Erase one row") which is needless to describe in detail.

Now, in general we want to apply the ERC, the EC or the ER method.
To do this, we need to find a linear hyperplane $H$ for which the dimension of $V_H$ is small, or a
$1$-dimensional linear subspace $D$ for which the dimension of $V^D$ is small.
This will be achieved by proving general theorems on the dimension of such spaces when
$\calV$ is an arbitrary rank-$\overline{r}$ affine subspace (with no specific assumption on its dimension).

Finally, let us discuss one of the later features of our article which is used in the proof of several lifting results.
At some point, we will have an affine space $\calS \subset \Mat_{n,p}(\K)$ of matrices
and we will want to prove that $\calS$ is a subspace of some compression space.
Say that we have a linear subspace $V$ of the translation vector space $S$ of
$\calS$ and a matrix $A \in \calS$ such that every matrix in $A+V$ has rank at most
$r$. Flanders's theorem says that the dimension of $V$ cannot be too large. Yet, say that $\dim V$ is close enough
to the critical dimension: then, by using one of our classification theorems, we have access to the structure of $A+V$,
which basically should show that $A+V$ is included in a rank-$\overline{r}$ compression space $\calR_A$.
Thus, $V$ is also included in $\calR_A$.
Yet, still assuming that the dimension of $V$ is large enough, we shall find that $V$
is included in a \emph{unique} compression space $\calR$, which must then equal $\calR_A$.
By varying $A$ we shall be able to find that
$\calS$ is also included in $\calR$. The result that are needed to perform this will be called \textbf{forcing lemmas.}

\subsection{Structure of the article}

Section \ref{basicsection} consists of some basic lemmas that will be used throughout the article.
There, we will also recall the results on range-compatible and quasi-range-compatible maps that were
proved recently \cite{dSPRC3} and which will be used in the proofs of the special lifting results.

The next three sections are devoted to the proofs of our three main theorems.
In each section, the global structure is the following one:
\begin{itemize}
\item We start by deriving the lifting results from the latest theorem we have proved (in the
case of the first classification theorem, the latest theorem is Flanders's).
The special lifting result will come last.
\item Then, the inductive proof is performed.
\item In the case of the first classification theorem and of its refined version, forcing lemmas are
derived from them in the last paragraph of the corresponding section.
\end{itemize}

In the last section of the article, we will discuss a possible direction for further research on the topic.

\section{Basic results}\label{basicsection}

\subsection{The extraction lemma}

\begin{lemma}[Extraction lemma]\label{extractionlemma}
Let $n,p,r,q$ be positive integers with $q \leq \min(n,p)$ and $r \leq \min(n,p)$.
Let $M=\begin{bmatrix}
A & C \\
B & D
\end{bmatrix}$ be a matrix of $\Mat_{n,p}(\K)$, with $A \in \Mat_q(\K)$.
Set $N:=\begin{bmatrix}
I_q & [0]_{q \times (p-q)} \\
[0]_{(n-q) \times q} & [0]_{(n-q) \times (p-q)}
\end{bmatrix}$. Assume that $\# \K>q$ and that
$$\forall t \in \K, \; \rk(A+tN) \leq r.$$
Then, $\rk(D) \leq r-q$.
\end{lemma}

\begin{proof}
Assume on the contrary that $D$ has rank $s>r-q$.
Multiplying on the right and on the left by well-chosen non-singular matrices, we can assume that
$D=\begin{bmatrix}
I_s & [0]_{s \times (p-q-s)} \\
[0]_{(n-q-s)\times s} & [0]_{(n-q-s) \times (p-q-s)}
\end{bmatrix}$. Then, for all $t \in \K$, the $(s+q)$ by $(s+q)$ submatrix
deduced from $A+tN$ by deleting the last $p-q-s$ columns and the last $n-q-s$ rows is singular,
and its determinant is a polynomial function of $t$ whose degree is less than or equal to $q$ and
whose coefficient along $t^q$ equals $1$. As $\# \K>q$, this polynomial should be zero, however.
\end{proof}

In practice, we shall frequently use the following generalized version of this result, which is a straightforward corollary:

\begin{cor}\label{extractioncor}
Let $n,p,r$ be non-negative integers such that $r \leq \min(n,p)$. Let $q \in \lcro 0,r\rcro$.
Let $M \in \Mat_{n,p}(\K)$ and $N \in \Mat_{n,p}(\K)$ and let $I$ and $J$ be respective subsets of $\lcro 1,n\rcro$ and $\lcro 1,p\rcro$
with cardinality $q$. Assume that all the columns of $N$ indexed outside of $I$ are zero and all its rows indexed outside of $J$ are zero.
Assume further that $\rk N=q$, that $\rk(M+tN) \leq r$ for all $t \in \K$, and that $\# \K> q$.
Denote by $D$ the submatrix of $M$ obtained by deleting the rows indexed over $I$ and the columns indexed over $J$.
Then, $\rk D \leq r-q$.
\end{cor}

\subsection{Affine spaces of matrices with rank at most $1$}

\begin{Not}
Let $n,p,n',p'$ be non-negative integers with $n \leq n'$ and $p \leq p'$, and let $\calW$ be a subset of
$\Mat_{n,p}(\K)$. We denote by
$\widetilde{\calW}^{(n',p')}$ the subset
$$\Biggl\{\begin{bmatrix}
A & [0]_{n \times (p'-p)} \\
[0]_{(n'-n) \times p} & [0]_{(n'-n) \times (p'-p)}
\end{bmatrix} \mid A \in \calW\Biggr\}$$
of $\Mat_{n',p'}(\K)$.
\end{Not}

\begin{prop}[Classification of affine matrix spaces with rank at most $1$]\label{affinerank1}
Let $\calS$ be an affine subspace of $\Mat_{n,p}(\K)$ in which every matrix has rank at most $1$.
Then:
\begin{itemize}
\item Either all the non-zero matrices of $\calS$ have the same kernel;
\item Or all the non-zero matrices of $\calS$ have the same image;
\item Or $\# \K=2$, $n \geq 2$, $p \geq 2$ and $\calS$ is equivalent to an affine subspace of $\widetilde{\calU_2(\K)}^{(n,p)}$.
\end{itemize}
\end{prop}

Here, the case of linear subspaces is well-known. The generalization to affine subspaces seems to be new.

\begin{proof}
We consider the following condition:
\begin{itemize}
\item[(J)] For each pair $(A,B)$ of non-zero matrices of $\calS$, either $A$ and $B$ have the same kernel or they have the same image.
\end{itemize}
To start with, we prove that under (J) one of the first two outcomes of our proposition holds.
Assume that (J) holds and that the non-zero matrices of $\calS$ do not have the same kernel.
Then, we can find non-zero matrices $A$ and $B$ with distinct kernels. Let $C \in \calS \setminus \{0\}$. Then,
$\Ker A \neq \Ker C$ or $\Ker B \neq \Ker C$, and hence by (J) we have either $\im A=\im C$ or $\im B=\im C$.
On the other hand, (J) yields that $\im A=\im B$, whence in any case $\im A=\im C$.
Thus, all the non-zero matrices of $\calS$ have the same image.

In the rest of the proof, we assume that (J) does not hold, and we aim at proving that the third stated outcome holds.
Thus, we have rank $1$ matrices $A$ and $B$ in $\calS$ such that $\Ker A\neq \Ker B$ and $\im A \neq \im B$.

In particular, we can find vectors $X \in \Ker B \setminus \Ker A$ and $Y \in \Ker A \setminus \Ker B$.
For all $t \in \K \setminus \{0,1\}$,
one has $(tA+(1-t)B)X=t AX \in \im A \setminus \{0\}$ and $(tA+(1-t)B)Y=(1-t) BY \in \im B \setminus \{0\}$.
If $\# \K>2$, this contradicts the assumption that $\rk(tA+(1-t)B) \leq 1$. Therefore, $\# \K=2$.

Now, let $C \in \calS \setminus \{A,B\}$. Then, $D:=A+B+C$ belongs to $\calS$ and $C+D=A+B$.
As before we find that $(A+B)X=AX \in \im A \setminus \{0\}$ and $(A+B)Y=BY \in \im B \setminus \{0\}$, whereas
$\rk (A+B) \leq \rk A+\rk B=2$.
Hence, $\rk (A+B)=2$ and $\im (A+B)=\im A\oplus \im B$.
Then, since $C$ and $D$ have rank at most $1$ and $\rk(C+D)=\rk(A+B)=2$, it is a classical result that
$\im (C+D)=\im C \oplus \im D$. In particular, $\im C \subset \im A+\im B$.
With the same line of reasoning applied to kernels, we find that $\Ker A\cap \Ker B$ has codimension $2$ in $\K^p$
and $\Ker A \cap \Ker B \subset \Ker C$.
With well-chosen matrices $P \in \GL_n(\K)$ and $Q \in \GL_p(\K)$, we obtain $PAQ=E_{1,1}$ and $PBQ=E_{2,2}$,
and replacing $\calS$ with the equivalent space $P\calS Q$ takes us to the reduced situation where $A=E_{1,1}$ and $B=E_{2,2}$.
The previous results can then be translated as saying that every matrix of $\calS$ has zero columns starting from the third one,
and zero rows starting from the third one. Hence, $\calS=\widetilde{\calT}^{(n,p)}$ for some affine subspace
$\calT$ of $\Mat_2(\K)$ that contains $E_{1,1}$ and $E_{2,2}$ and in which every non-zero matrix has rank $1$.

To complete the proof, it remains to show that $\calT$ is equivalent to a subset of $\calU_2(\K)$.
If $\calT=\{A,B\}$, then we are done. Assume now that $\{A,B\} \subsetneq \calT$.
Then, $\dim \calT \geq 2$, and by Flanders's theorem the space $\calT$ is equivalent to $\calU_2(\K)$
as on the one hand no non-zero vector of $\K^2$ is annihilated by both $E_{1,1}$ and $E_{2,2}$, and on the other hand
no non-zero vector of $\K^2$ is annihilated by both
$E_{1,1}^T$ and $E_{2,2}^T$.
\end{proof}

\subsection{On the rank $1$ matrices in the translation vector space of a rank-$\overline{k}$ space}

Let $\calS$ be a rank-$\overline{k}$ affine subspace of $\Mat_{n,p}(\K)$.
In our proof of the classification theorems, we shall need to find a linear hyperplane $H$ of $\K^p$
such that the dimension of $S_H$ is small, or a $1$-dimensional linear subspace $D$ of $\K^n$ such that the dimension of $S^D$ is small.
This will be obtained thanks to the following series of lemmas.
The first one is taken from \cite{dSPFlandersskew}.

\begin{lemma}[Lemma 6 of \cite{dSPFlandersskew}]\label{2ndkey}
Let $n$ and $p$ be non-negative integers.
Let $\calS$ be an affine subspace of $\Mat_{n,p}(\K)$ with upper-rank at most $r$.
Assume that $\dim S_H \geq r$ for every linear hyperplane $H$ of $\K^p$.
Then, $\calS$ is equivalent to $\calR(r,0)$.
\end{lemma}

The next lemma is an elaboration of the previous one; it will be used in the proof of the
refined first classification theorem and in the one of the second classification theorem.

\begin{lemma}\label{3rdkey}
Let $n,p,r$ be non-negative integers with $r>0$.
Let $\calS$ be an affine subspace of $\Mat_{n,p}(\K)$ with upper-rank $r<\min(n,p)$.
Then, one of the following three outcomes must occur:
\begin{enumerate}[(a)]
\item There is a linear hyperplane $H$ of $\K^p$ such that $\dim S_H \leq \frac{r-1}{2}\cdot$
\item There is a $1$-dimensional linear subspace $D$ of $\K^n$ such that $\dim S^D \leq \frac{r-1}{2}\cdot$
\item $r$ is even and $\calS$ is equivalent to $\calR(r/2,r/2)$.
\end{enumerate}
\end{lemma}

\begin{proof}
Assuming that none of outcomes (a) and (b) holds, we aim at proving that outcome (c) holds.

Without loss of generality, we can assume that $\calS$ contains
$$M_0=\begin{bmatrix}
A & [0]_{r \times (p-r)} \\
[0]_{(n-r) \times r} & [0]_{(n-r) \times (p-r)}
\end{bmatrix}$$
for some invertible matrix $A \in \GL_r(\K)$.

Let $H$ be an arbitrary linear hyperplane of $\K^p$ that includes $\K^r \times \{0\}$,
and $D$ be a $1$-dimensional linear subspace of $\K^n$ that is included in $\{0\} \times \K^{n-r}$.
For any matrix $N$ of $S_H \cup S^D$, we can write
$$N=\begin{bmatrix}
[0]_{r \times r} & C(N) \\
B(N) & D(N)
\end{bmatrix}$$
with $B(N) \in \Mat_{n-r,r}(\K)$, $C(N) \in \Mat_{r,p-r}(\K)$ and $D(N) \in \Mat_{n-r,p-r}(\K)$,
and we note that $B(N)=0$ if $N \in S_H$ whereas $C(N)=0$ if $N \in S^D$.
Now, let $N_1 \in S_H$ and $N_2 \in S^D$.
The matrices
$$M_0+N_1=\begin{bmatrix}
A & C(N_1) \\
[0]_{(n-r) \times r} & D(N_1)
\end{bmatrix} \quad \text{and} \quad
M_0+N_2=\begin{bmatrix}
A & [0]_{r \times (p-r)} \\
B(N_2) & D(N_2)
\end{bmatrix}$$
belong to $\calS$, and hence $\rk(M_0+N_1) \leq r$ and $\rk(M_0+N_2) \leq r$, which leads to $D(N_1)=0=D(N_2)$.
Next, the matrix
$$M_0+N_1+N_2=\begin{bmatrix}
A & C(N_1) \\
B(N_2) & [0]_{(n-r) \times (p-r)}
\end{bmatrix}$$
also belongs to $\calS$. Yet, by Gaussian elimination this matrix is equivalent to
$$\begin{bmatrix}
A & C(N_1) \\
0 & -B(N_2)A^{-1}C(N_1)
\end{bmatrix},$$
which leads to $B(N_2)A^{-1}C(N_1)=0$.
Setting
$$T_H:=\sum_{N \in S_H} \im C(N) \quad \text{and} \quad T^D:=\sum_{N \in S^D} \im B(N)^T,$$
which are linear subspaces of $\K^r$, we have
$$\dim S_H=\dim T_H \quad \text{and} \quad \dim S^D=\dim T^D,$$
and $T_H$ is right-orthogonal to $T^D$ for the non-degenerate bilinear form
$$b : (X,Y) \in (\K^r)^2 \longmapsto X^T A^{-1} Y$$
on $\K^r$.

In particular, this shows that $\dim S_H+\dim S^D \leq r$. Yet,
since $\dim S_H \geq \frac{r}{2}$ and $\dim S^D \geq \frac{r}{2}$, we deduce that $r$ is even and that
$\dim S_H=\dim S^D=\frac{r}{2}$, which further leads to $T_H$ being the (right)-orthogonal complement of $T^D$ under $b$.
Set $s:=\frac{r}{2}\cdot$
Hence, varying $H$ and $D$ shows that there are linear subspaces $V$ and $W$ of $\K^r$, both with dimension $s$,
such that $W$ is the left-orthogonal complement of $V$ under $b$ and $T_H=V$ and $T^D=W$ for every linear hyperplane
$H$ of $\K^p$ that includes $\K^r \times \{0\}$ and every $1$-dimensional linear subspace $D$ of $\K^n$ that is included in $\{0\} \times \K^{n-r}$.

Without further loss of generality, we can now assume that $V=W=\K^s \times \{0\}$.
In that reduced situation, we deduce that:
\begin{itemize}
\item For all $(i,j)\in \lcro 1,s\rcro \times \lcro r+1,p\rcro$, the space $S$ contains $E_{i,j}$;
\item For all $(i,j)\in \lcro r+1,n\rcro \times \lcro 1,s\rcro$, the space $S$ contains $E_{i,j}$.
\end{itemize}

Now, we shall use an invariance argument to obtain several more elementary matrices in $S$.
Let $j \in \lcro 1,r\rcro$. Consider the affine space $\calS'$ deduced from $\calS$ by
the elementary column operation $C_j \leftarrow C_j-C_p$, and denote by $S'$ its translation vector space.
We see that $\calS'$ satisfies the same assumptions as $\calS$ and still contains $M_0$.
Hence, we have linear subspaces $V'$ and $W'$ of $\K^r$ that are attached to $\calS'$ as $V$ and $W$ were attached to $\calS$.
Yet, since $S$ contains $E_{n,1},\dots,E_{n,s}$, so does $S'$. It follows that $W'$ includes $\K^s \times \{0\}$
and as the dimensions are equal we deduce that $W'=\K^s \times \{0\}=W$, and hence $V=V'$, both spaces being equal to the right-orthogonal of
$W$ under $b$. It follows that $S'$ contains $E_{1,p},\dots,E_{s,p}$, whence $S$ contains $E_{1,p}+E_{1,j},\dots,E_{s,p}+E_{s,j}$.
Since we already knew that $S$ contains $E_{1,p},\dots,E_{s,p}$, we conclude that $S$ contains $E_{1,j},\dots,E_{s,j}$.
Hence, $S$ contains $E_{i,j}$ for all $(i,j)\in \lcro 1,s\rcro \times \lcro 1,p\rcro$.

By using a similar method (with row operations instead of column operations), we obtain that $S$ contains $E_{i,j}$ for all
$(i,j)\in \lcro 1,n\rcro \times \lcro 1,s\rcro$.
In particular, we see that $\calR(s,s) \subset S$.
To conclude, we demonstrate that $\calS=\calR(s,s)$.
Indeed, let $M \in \calS$, and let $(i_1,\dots,i_s)$ and $(j_1,\dots,j_s)$ be arbitrary increasing sequences
in $\lcro s+1,n\rcro$ and $\lcro s+1,p\rcro$, respectively.
Denote by $N$ the submatrix of $M$ obtained by deleting the rows indexed over $\lcro 1,s\rcro \cup \{i_1,\dots,i_s\}$
and the columns indexed over $\lcro 1,s\rcro \cup \{j_1,\dots,j_s\}$.
For all lists $(x_1,\dots,x_s)$ and $(y_1,\dots,y_s)$ of scalars in $\K$,
the matrix $M+\underset{k=1}{\overset{s}{\sum}} x_k E_{i_k,k}+ \underset{k=1}{\overset{s}{\sum}} y_k E_{k,j_k}$ has rank at most $r$.
Applying Corollary \ref{extractioncor} repeatedly shows that $N=0$.
As $r<p$ and $r<n$, varying the $(i_k)$ and $(j_k)$ sequences shows that
all the entries of $M$ indexed over $\lcro s+1,n\rcro \times \lcro s+1,p\rcro$ equal $0$.
Thus, $\calS \subset \calR(s,s)$. Since $\calR(s,s) \subset S$, we conclude that
$\calS=\calR(s,s)$, and hence condition (c) holds.
\end{proof}

\subsubsection{A review of range-compatible and quasi-range-compatible maps}\label{RCsection}

\begin{Def}
Let $\calS$ be a subset of $\calL(U,V)$, where $U$ and $V$ are vector spaces over an arbitrary field.
A map $F : \calS \rightarrow V$ is called \textbf{range-compatible} whenever
$$\forall s \in \calS, \; F(s) \in \im s.$$
It is called \textbf{quasi-range-compatible} when there exists a $1$-dimensional linear subspace
$D$ of $V$ such that
$$\forall s \in \calS, \; D \not\subset \im s \Rightarrow F(s) \in \im s.$$
It is called \textbf{local} when there exists a vector $x \in U$ such that
$$\forall s \in \calS, \; F(s)=s(x).$$
\end{Def}

Obviously, every range-compatible map is quasi-range-compatible, and every local map is range-compatible.
In each case, the converse does not hold in general, even if $\calS$ is assumed to be a linear subspace of
$\calL(U,V)$ and $F$ is assumed to be linear.
The above notions are naturally adapted to matrix spaces by using the standard identification between $\Mat_{n,p}(\K)$ and
$\calL(\K^p,\K^n)$.

The notion of a range-compatible map was introduced very recently \cite{dSPRC1}.
It is motivated by its connection to the topic of our article (see Section 2 of \cite{dSPclass}),
by its connection to linear invertibility preservers (see \cite{dSPlargelinpres}),
and finally it is closely connected to the fashionable notion of algebraic reflexivity (see Section 1.1 of \cite{dSPRC2}
for a thorough discussion).

In this article, we shall use recent theorems on range-compatible maps to obtain our so-called ``special" lifting results.
Below are the main theorems that we shall use. The first one deals with affine range-compatible maps on affine
subspaces of matrices.

\begin{theo}[Theorem 3.1 of \cite{dSPRC3}]\label{RCtheo1}
Let $n$ and $p$ be non-negative integers, and $\calS$ be an affine subspace of $\Mat_{n,p}(\K)$ such that
$\codim \calS \leq n-2$. Then, every range-compatible affine map on $\calS$ is local.
\end{theo}

The next result deals with quasi-range-compatible maps on affine subspaces.
It is a straightforward corollary to Proposition 4.2 and Theorems 5.2 and 5.3 of \cite{dSPRC3}.

\begin{theo}\label{RCtheo2}
Let $n$ and $p$ be non-negative integers with $n \geq 2$, and $\calS$ be an affine subspace of $\Mat_{n,p}(\K)$ such that
$\codim \calS \leq n-1$ and $\# \K>2$. Let $F$ be a quasi-range-compatible affine map on $\calS$.
Then, one of the following conditions must hold:
\begin{enumerate}[(i)]
\item The map $F$ is local.
\item There exist vectors $X \in \K^p \setminus \{0\}$, $X' \in \K^p$, a $2$-dimensional linear subspace $P$ of $\K^n$ that includes
$\calS X$ and an endomorphism $\varphi$ of $P$ such that
$$F : M \mapsto \varphi(MX)+MX'.$$
\end{enumerate}
\end{theo}

Our final result deals with quasi-range-compatible linear maps on large linear subspaces of matrices.
It combines Theorem 4.4, Proposition 4.2 and Proposition 4.3 of \cite{dSPRC3}.

\begin{theo}\label{RCtheo3}
Let $n$ and $p$ be non-negative integers, and $S$ be a linear subspace of $\Mat_{n,p}(\K)$.
Assume that $\codim S \leq 2n-4-\epsilon(\K)$.
Let $F$ be a quasi-range-compatible linear map on $S$.
Then, one of the following situations must hold:
\begin{enumerate}[(i)]
\item The map $F$ is local.
\item There exist vectors $X \in \K^p \setminus \{0\}$, $X' \in \K^p$, a $2$-dimensional linear subspace $P$ of $\K^n$ that includes
$SX$ and an endomorphism $\varphi$ of $P$ such that
$$F : M \mapsto \varphi(MX)+MX'.$$
\end{enumerate}
\end{theo}

\section{Proof of the first classification theorem}\label{1stclasssection}

This section is devoted to the proof of the first classification theorem.
In the first part, we prove two lifting results, and in the second part we use those results to prove the first classification theorem
by induction.

\subsection{Lifting results}

\begin{prop}[Lifting lemma 1]\label{lifting1}
Let $n,p,r$ be positive integers such that $r<\min(n,p)$.
Let $\calV$ be a rank-$\overline{r}$ affine subspace of $\calR(1,r)$
such that
$$\dim \calV \geq nr-(n-p+r)+2.$$
Then, $\calV \subset \calR(0,r)$.
\end{prop}

\begin{proof}
We write every matrix $M$ of $\calV$ as
$$M=\begin{bmatrix}
[?]_{1 \times r} & C(M) \\
B(M) & [0]_{(n-1) \times (p-r)}
\end{bmatrix}$$
with $C(M) \in \Mat_{1,p-r}(\K)$ and $B(M) \in \Mat_{n-1,r}(\K)$.
Assume that $C$ is not identically zero on $\calV$, and pick a matrix $M_0 \in \calV$ such that $C(M_0)\neq 0$.
Denote by $\calW$ the affine subspace of $\calV$ consisting of all its matrices $M$ such that $C(M)=C(M_0)$.
Then, we see that
$$\forall M \in \calW, \; \rk B(M) \leq r-1.$$
Therefore, $B(\calW)$ is a rank-$\overline{r-1}$ affine subspace of $\Mat_{n-1,r}(\K)$.
If follows from Flanders's theorem that
$$\dim B(\calW) \leq (n-1)(r-1).$$
Hence,
$$\dim \calV \leq \dim B(\calW)+p \leq  nr-n-r+1+p,$$
contradicting our assumptions.
Thus, $C(M)=0$ for all $M \in \calV$, which shows that $\calV \subset \calR(0,r)$.
\end{proof}

\begin{prop}[Special lifting lemma 1]\label{speciallifting1}
Let $n$ and $r$ be positive integers such that $n>r$. Let $\calW$ be an affine subspace of
$\Mat_{n,r}(\K)$ such that
$$\dim \calW \geq nr-n+3.$$
Let $f : \calW \rightarrow \K^n$ be an affine map such that
the affine subspace
$$\calV:=\biggl\{\begin{bmatrix}
f(N) & N
\end{bmatrix}\mid N \in \calW\biggr\}$$
of $\Mat_{n,r+1}(\K)$ is a rank-$\overline{r}$ space. Then, $\calV$ is equivalent to a subset of $\calR(0,r)$.
\end{prop}

\begin{proof}
We shall prove that $f$ is range-compatible.

Let $G$ be a linear hyperplane of $\K^n$, and choose a linear form $\varphi$ on $\K^n$ with kernel $G$.
Set $\calW':=\{N \in \calW : \; \im N \subset \Ker \varphi\}$.
Then, $\calW'$ can be identified with a linear subspace of $\calL(\K^r,\Ker \varphi)$.
Consider the map
$$g : N \in \calW' \mapsto \varphi(f(N)).$$
For all $N \in \calW'$ such that $g(N) \neq 0$ we must have $\rk N \leq r-1$ because
$\rk \begin{bmatrix}
f(N) & N
\end{bmatrix} \leq r$.
If $g \neq 0$, we choose a non-zero element $a$ in the range of $g$, and hence
$g^{-1} \{a\}$ is a rank-$\overline{r}$ affine subspace of $\calW'$ with codimension at most $1$;
then, by Flanders's theorem,
$$\dim g^{-1} \{a\} \leq (n-1)(r-1),$$
which leads to
$$\dim \calW \leq r+\dim \calW' \leq r+1+\dim g^{-1} \{a\} \leq nr-n+2,$$
contradicting our assumptions.
Thus, $g=0$, which shows that
$$\forall N \in \calW, \; \im N \subset G \Rightarrow f(N) \in G.$$
Now, let $N \in \calW$. For each linear hyperplane $G$ of $\K^n$ that includes $\im N$, we have
$f(N) \in G$, and hence $f(N) \in \im N$.
Therefore, $f$ is range-compatible.

The map $f$ is local according to Theorem \ref{RCtheo1}: this yields a vector $X \in \K^r$ such that
every matrix of $\calV$ vanishes at the non-zero vector $\begin{bmatrix}
1 \\
-X
\end{bmatrix}$, and hence $\calV$ is equivalent to a subset of $\calR(0,r)$.
\end{proof}

\subsection{Setting the proof up}

In the next three sections, we perform the inductive proof of the first classification theorem.
We work by induction over $n,p,r$.
Let $\calS$ be an affine subspace of $\Mat_{n,p}(\K)$ in which every matrix has rank at most $r$.
Assume that $\dim \calS\geq nr-(n-p+r)+2$.

We assume that $\calS$ is inequivalent to $\calR(r,0)$, and we
try to prove that $n=p$ and that $\calS$ is actually equivalent to $\calR(0,r)$.
Note that if $n>p$ then $\dim \calR(r,0)<nr-(n-p+r)+2$.

Denote by $S$ the translation vector space of $\calS$.
We shall split the discussion into three cases:
\begin{itemize}
\item Case 0: For every linear hyperplane $H$ of $\K^p$, one has $\dim S_H \geq r$.
\item Case 1: There exists a linear hyperplane $H$ of $\K^p$ such that $0<\dim S_H < r$.
\item Case 2: There exists a linear hyperplane $H$ of $\K^p$ such that $\dim S_H=0$.
\end{itemize}

In Case 0 we use Lemma \ref{2ndkey}: in that situation we know that $\calS$ is equivalent to $\calR(r,0)$ and hence $n=p$.
In Case 1, we shall use the ERC method, and in Case 2 we shall use the EC method.

\subsection{Case 1: There exists a linear hyperplane $H$ of $\K^p$ such that $0<\dim S_H < r$.}

In particular, $r \geq 2$.
Let us apply the ERC method.
We lose no generality in assuming that $S_H$ contains $E_{1,1}$ (so that $H=\{0\} \times \K^{p-1}$).
Then, we split every matrix $M \in \calS$ as
$$M=\begin{bmatrix}
? & [?]_{1 \times (p-1)} \\
[?]_{(n-1) \times 1} & P(M)
\end{bmatrix} \quad \text{with $P(M) \in \Mat_{n-1,p-1}(\K)$.}$$
As $S$ contains $E_{1,1}$, the extraction lemma (Lemma \ref{extractionlemma}) shows that $\urk P(\calS) \leq r-1$.
On the other hand, the rank theorem shows that
\begin{multline*}
\dim P(\calS) \geq \dim \calS-(p-1)-\dim S_H \geq
nr-(n-p+r)+2-(p-1)-(r-1)\\
=(n-1)(r-1)-(r-1)+2.
\end{multline*}
Hence the induction hypothesis applies to $P(\calS)$.
Without loss of generality, we can then assume that either $P(\calS) \subset \calR(0,r-1)$, or $n=p$ and $P(\calS) \subset \calR(r-1,0)$.

If $P(\calS) \subset \calR(0,r-1)$, then $\calS$ is a subspace of $\calR(1,r)$, and Proposition \ref{lifting1}
yields that $\calS$ is a subspace of $\calR(0,r)$.

Assume finally that $P(\calS) \subset \calR(r-1,0)$ and $n=p$, so that $\calS \subset \calR(r,1)$.
Then, Proposition \ref{lifting1} applies to $\calS^T$ and shows that $\calS \subset \calR(r,0)$.

\subsection{Case 2: There exists a linear hyperplane $H$ of $\K^p$ such that $S_H=\{0\}$.}

Without loss of generality, we can assume that $H=\{0\} \times \K^{p-1}$.
Then, we apply the EC method: we write every matrix $M$ of $\calS$ as
$$M=\begin{bmatrix}
[?]_{n \times 1} & J(M)
\end{bmatrix} \quad \text{with $J(M) \in \Mat_{n,p-1}(\K)$.}$$
We know that $\urk J(\calS) \leq r$ and
$\dim J(\calS)=\dim \calS \geq nr-(n-(p-1)+r)+2$.
If $r<p-1$, then we know by induction that $J(\calS)$ is equivalent to an affine subspace of $\calR(0,r)$,
and hence we can assume without loss of generality
that $J(\calS) \subset \calR(0,r)$; then, if we consider the subspace obtained by deleting the last column
in the matrices of $\calS$ we know by induction that it is equivalent to $\calR(0,r)$,
and a similar conclusion follows for $\calS$.

From now on, we assume that $r=p-1$. Set
$$\calW:=J(\calS).$$
Then, as $S_H=\{0\}$ we have an affine map
$$f : \calW \rightarrow \K^n$$
such that
$$\calS=\Bigl\{\begin{bmatrix}
f(N) & N
\end{bmatrix} \mid N \in \calW\Bigr\}.$$
Moreover,
$$\dim \calW=\dim \calS \geq nr-(n-p+r)+2 = nr-n+3.$$
Therefore, Proposition \ref{speciallifting1} shows that $\calS$ is equivalent to a subset of $\calR(0,r)$.

This completes the proof of the first classification theorem.

\subsection{A corollary of the first classification theorem}

\begin{cor}[Forcing lemma 1]\label{forcing1}
Let $\calV$ be a rank-$\overline{r}$ affine subspace of $\Mat_{n,r+1}(\K)$, with translation vector space $V$.
Assume that $\dim \calV \geq nr-n+3$.
\begin{itemize}
\item If $V \subset \calR(0,r)$ then $\calV \subset \calR(0,r)$.
\item If $V \subset \calR(r,0)$ then $\calV \subset \calR(r,0)$.
\end{itemize}
\end{cor}

\begin{proof}
Denote by $e_{r+1}$ the last vector of the standard basis of $\K^{r+1}$.
Assume that $V \subset \calR(0,r)$.
The assumptions show that $\dim \calV \geq nr-(n-p+r)+2$, where $p:=r+1$, and hence the first classification theorem applies to
$\calV$. There are two possibilities:
\begin{itemize}
\item \textbf{Case 1:} $\calV$ is equivalent to a subspace of $\calR(0,r)$, yielding a non-zero vector $x \in \K^{r+1}$
such that $\calV x=\{0\}$. Then, $V x=\{0\}$. If $x$ and $e_{r+1}$ were non-collinear, we would find that
$\dim V \leq n(r-1)$, which is false. Therefore, $x$ and $e_{r+1}$ are collinear, which yields that $\calV e_{r+1}=\{0\}$.
Therefore, $\calV \subset \calR(0,r)$.

\item \textbf{Case 2:} $n=r+1$ and $\calV$ is equivalent to a subspace of $\calR(r,0)$, meaning that we have a linear hyperplane $H$ of $\K^n$ such that $\im N \subset H$ for all $N \in \calV$. Thus $\im N \subset H$ for all $N \in V$.
    Combining this with $V \subset \calR(0,r)$ yields that $\dim V \leq nr-r$, contradicting our assumptions.
\end{itemize}

If $V \subset \calR(r,0)$, then $n=r+1$ since $\dim \calV \geq nr-n+3$; applying the first case to $\calV^T$
then leads to $\calV \subset \calR(r,0)$.
\end{proof}

If the same line of reasoning, one obtains, by using Proposition \ref{affinerank1} instead of the first classification theorem:

\begin{cor}[Forcing lemma 2]\label{forcing2}
Let $\calV$ be a rank-$\overline{1}$ affine subspace of $\Mat_{n,p}(\K)$, with translation vector space $V$.
Assume that $\dim \calV \geq 2$ and that $\# \K>2$.
\begin{itemize}
\item If $V \subset \calR(0,1)$ then $\calV \subset \calR(0,1)$.
\item If $V \subset \calR(1,0)$ then $\calV \subset \calR(1,0)$.
\end{itemize}
\end{cor}

\section{Proof of the refined first classification theorem}\label{refined1stclasssection}

In this section, we prove the refined first classification theorem by induction on $n$.
This proof is structured as follows:
\begin{itemize}
\item We start by tackling the case when $n=p=3$ and $r=2$ (Section \ref{n=3r=2section}), using techniques that are substantially different from those of the rest of our article.
\item In Section \ref{liftingsection2}, we prove two lifting results.
\item In Section \ref{specialliftingsection2}, we prove the special lifting lemma that is needed in the proof of the refined
first classification theorem.
\item We wrap up the inductive proof of the refined first classification theorem in Section \ref{refinedproofsection}.
\item In the last section, we derive the forcing lemmas that will be used in our proof of the second classification theorem.
\end{itemize}

\subsection{The case $n=p=3$ and $r=2$}\label{n=3r=2section}

\begin{prop}\label{n=3r=2}
Let $\calV$ be a rank-$\overline{2}$ affine subspace of $\Mat_3(\K)$ with $\dim \calV\geq 4$. Assume that $\# \K>2$.
Then, either $\calV$ is $2$-decomposable, or $\# \K=3$ and $\calV$ is equivalent to $\calU_3(\K)$.
\end{prop}

Before we prove the result, note that we can give a direct proof that $\calU_3(\K)^T$ is equivalent to $\calU_3(\K)$ if $\# \K=3$:
indeed, with $J:=\begin{bmatrix}
0 & 0 & 1 \\
0 & 1 & 0 \\
1 & 0 & 0
\end{bmatrix}$, and $D:=\begin{bmatrix}
-1 & 0 & 0 \\
0 & -1 & 0 \\
0 & 0 & -1
\end{bmatrix}$, one checks that $DJ \,\calU_3(\K)^T J=\calU_3(\K)$.

\begin{proof}
Assume that $\calV$ is not $2$-decomposable.
Denote by $V$ the translation vector space of $\calV$.
Our aim is to prove that $\# \K=3$ and $\calV$ is equivalent to $\calU_3(\K)$.

\vskip 2mm
\noindent \textbf{Step 1: The space $V$ contains a rank $1$ matrix.}  \\
Assume the contrary. Let $X \in \K^3 \setminus \{0\}$. Then, $\{N \in V : \; NX=0\}$ is a linear subspace of $V$ with dimension at least $1$
because $\dim V \geq 4$.
Choosing a non-zero matrix $N \in V$ such that $NX=0$, we find that $\rk N \leq 2$ and hence $\rk N=2$.
Applying the extraction lemma (Lemma \ref{extractionlemma}), we deduce that every matrix $M$ in $\calV$ maps $\Ker N=\K X$ into $\im N$.
It follows that $\dim (VX) \leq 2$.

Now, we consider the operator space
$$\widehat{V}:=\{M \in V \mapsto MX \mid X \in \K^3\}.$$
Note that no non-zero vector in $V$ is annihilated by all the operators in $\widehat{V}$.
We have just shown that every operator in $\widehat{V}$ has rank at most $2$.
Then, we apply the classification of rank-$\overline{2}$ linear subspaces (see Section 4 of \cite{AtkinsonPrim}).
The space $\widehat{V}$ cannot be represented by a subspace of $\calR(1,1)$ in well-chosen bases, because $\dim V>1$
and $V$ contains no rank $1$ matrix. As $\dim V \geq 4$, this only leaves open the possibility that
there exists a $2$-dimensional linear subspace $P$ of $\K^3$ such that every operator in $\widehat{V}$
has its range included in $P$. This yields a non-zero vector $Y \in \K^3$ such that $Y^T N=0$ for all $N \in V$.
Working with $\calV^T$, we also obtain a non-zero vector $X \in \K^3$ such that $NX=0$ for all $N \in V$.
Then, $V$ is included in the $4$-dimensional space of all matrices $N \in \Mat_3(\K)$ such that $NX=0$ and $Y^TN=0$,
and it follows that $V$ equals that space. Yet, that space contains a rank $1$ matrix. Therefore, the claimed result is proved.

\vskip 2mm
Now, we have a rank $1$ matrix $N_0 \in V$. Without loss of generality, we can assume that
$N_0=E_{1,3}$. Let us write every matrix $M \in \calV$ as
$$M=\begin{bmatrix}
[?]_{1 \times 2} & ? \\
K(M) & [?]_{2 \times 1}
\end{bmatrix} \quad \text{with $K(M) \in \Mat_2(\K)$.}$$

\noindent
\textbf{Step 2: Reduction to the case when $K(\calV) \subset \calR(1,0)$.} \\
By the extraction lemma we find that $\rk K(M) \leq 1$ for all $M \in \calV$.
By the classification of rank-$\overline{1}$ affine subspaces (Proposition \ref{affinerank1}), we deduce, as $\# \K>2$, that
$K(\calV)$ is equivalent to a subspace of $\calR(1,0)$ or of $\calR(0,1)$.
Without loss of generality, we can then assume that $K(\calV) \subset \calR(1,0)$ or $K(\calV) \subset \calR(0,1)$.
Setting $J:=\begin{bmatrix}
0 & 0 & 1 \\
1 & 0 & 0 \\
0 & 1 & 0
\end{bmatrix}$, we see that in the second case the space $\calV'=J \calV^T J$ satisfies the assumptions of the first case,
and hence if we can prove that $\calV'$ is equivalent to $\calU_3(\K)$ then $\calV$ is equivalent to $\calU_3(\K)^T$, and hence
$\calV$ is equivalent to $\calU_3(\K)$.

Thus, in the rest of the proof, we can assume that $K(\calV) \subset \calR(1,0)$.
Hence, we can write every $M \in \Vect(\calV)$ as
$$M=\begin{bmatrix}
A(M) & [?]_{2 \times 1} \\
[0]_{1 \times 2} & \alpha(M)
\end{bmatrix} \quad \text{with $A(M) \in \Mat_2(\K)$ and $\alpha(M) \in \K$.}$$
Note that $\alpha(M) \neq 0 \Rightarrow \rk A(M) \leq 1$, for all $M \in \calV$.

\vskip 2mm
\noindent
\textbf{Step 3: $\urk A(\calV)=2$ and $\alpha \neq 0$.} \\
The map $\alpha$ is non-zero since $\calV$ is not equivalent to a subspace of $\calR(2,0)$.
If $A(\calV)$ were a rank-$\overline{1}$ space then it would be equivalent to a subspace of $\calR(1,0)$
or $\calR(0,1)$, and hence $\calV$ would be equivalent to a subspace of $\calR(1,1)$ or $\calR(0,2)$, contradicting our assumptions.
Thus, there exists a rank $2$ matrix in $A(\calV)$.

\vskip 2mm
\noindent
\textbf{Step 4: $\alpha(M)$ is an affine function of $A(M)$.} \\
Assume that there exists $M_0 \in V$ such that $\alpha(M_0) \neq 0$ and $A(M_0)=0$.
Then, choosing $M_2 \in \calV$ such that $\rk A(M_2)=2$, we can find $\lambda \in \K$ such that $\alpha(M_2+\lambda M_0)=1$,
and then $A(M_2+\lambda M_0)=A(M_2)$ has rank $2$, contradicting an earlier statement.
Thus, $\alpha(M_0)=0$ for all $M_0 \in V$ such that $A(M_0)=0$, which yields a
non-zero affine map $\gamma : A(\calV) \rightarrow \K$ such that
$$\forall M \in \calV, \; \alpha(M)=\gamma(A(M)).$$

The map $\gamma$ is non-constant as there exists $M \in \calV$ such that $\rk A(M)=2$.
Let $a \in \K \setminus \{0\}$. Then,
$$\calT_a:=\gamma^{-1} \{a\}$$
is an affine
hyperplane of $A(\calV)$ consisting of matrices with rank less than $2$.
Its translation vector space $T$ does not depend on the choice of $a$.

\vskip 2mm
\noindent \textbf{Step 5: $\dim T=1$.} \\
Note that $4 \leq \dim \calV \leq 2+\dim A(\calV)$, and hence $\dim T \geq 1$.
On the other hand $\dim T \leq 2$ by Flanders's theorem.

Assume that $\dim T=2$. Then, by Flanders's theorem $\calT_a$ is a linear subspace of $\Mat_2(\K)$, whence
it equals $T$. However, the affine hyperplanes $\calT_b$ of $A(\calV)$, for $b \in \K \setminus \{0\}$, should be strictly parallel,
which is a contradiction since $\# \K>2$. Hence, $\dim T=1$.

\vskip 2mm
Next, as a consequence of Proposition \ref{affinerank1}, every non-zero matrix in $T$ has rank $1$.
Without loss of generality, we can then assume that $T=\Vect(E_{1,2})$.

\vskip 2mm
\noindent \textbf{Step 6: Every matrix in $\calV$ is upper-triangular.} \\
Applying the extraction lemma, we obtain that
$$\forall M=(m_{i,j}) \in \calT_a, \quad m_{2,1}=0.$$
Let us pick distinct non-zero elements $a$ and $b$ in $\K \setminus \{0\}$.
Then, as $\calT_a \cup \calT_b$ generates the affine space $A(\calV)$, we deduce that
every matrix in $A(\calV)$ is upper-triangular. Therefore, every matrix in $\calV$ is upper-triangular.

\vskip 2mm
\noindent \textbf{Step 7: The space $\calV$ is equivalent to $\calU_3(\K)$.} \\
We have three affine forms $a_1,a_2,a_3$ on $\calV$ such that
$$\forall M \in \calV, \quad
M=\begin{bmatrix}
a_1(M) & ? & ? \\
0 & a_2(M) & ? \\
0 & 0 & a_3(M)
\end{bmatrix}.$$
The maps $a_1,a_2,a_3$ are all non-zero since $\calV$ is not $2$-decomposable.
On the other hand, as every matrix in $\calV$ is singular, we have
$$\forall M \in \calV, \; a_1(M)a_2(M)a_3(M)=0.$$
Thus, the subsets $a_i^{-1} \{0\}$, each of which is either empty or an affine hyperplane of $\calV$, cover $\calV$.
Since $\# \K>2$, the space $\calV$ cannot be covered by two of its affine hyperplanes, and hence
$a_i^{-1}\{0\}$ is an affine hyperplane for all $i \in \{1,2,3\}$.
If we can find an affine hyperplane $\calH$ of $\calV$ that is parallel to $a_1^{-1} \{0\}$
and is different from both $a_2^{-1} \{0\}$ and $a_3^{-1} \{0\}$, then
$a_2^{-1} \{0\} \cap \calH$ and $a_3^{-1} \{0\} \cap \calH$ are proper affine subspaces of $\calH$
(or empty) and they cover $\calH$, which is not possible since $\# \K>2$.
As $\# \K> 2$ it follows that $a_1^{-1} \{0\}$, $a_2^{-1} \{0\}$ and $a_3^{-1} \{0\}$ are parallel and pairwise distinct,
and that $\# \K=3$. Thus, $a_2-a_1$ and $a_3-a_1$ are distinct non-zero constant maps.
If $a_2=a_1+1$ and $a_3=a_1+2$, then $\calV$ is a subspace of $\calU_3(\K)$, and as their dimensions are equal
those spaces are equal. \\
If $a_2=a_1+2$ and $a_3=a_1+1$, then multiplying each column by $-1$ shows that $\calV$ is equivalent to a subspace of $\calU_3(\K)$,
and just like in the first case this would show that $\calV$ is equivalent to $\calU_3(\K)$.
In any case, we have the expected conclusion.
\end{proof}

With the same method as in the proof of Corollary \ref{forcing1}, we derive:

\begin{cor}[Forcing lemma 3]\label{forcing3}
Let $\calV$ be a rank-$\overline{2}$ affine subspace of $\Mat_3(\K)$, with $\dim \calV \geq 5$.
Denote by $V$ its translation vector space.
\begin{itemize}
\item If $V \subset \calR(2,0)$ then $\calV \subset \calR(2,0)$.
\item If $V \subset \calR(0,2)$ then $\calV \subset \calR(0,2)$.
\item If $V=\calR(1,1)$ then $\calV =\calR(1,1)$.
\end{itemize}
\end{cor}

\subsection{Lifting results}\label{liftingsection2}

\begin{prop}[Lifting lemma 2.1]\label{liftingprop2.1}
Let $n,p,r$ be non-negative integers such that $r<\min(n,p)$.
Let $\calV$ be an affine subspace of $\Mat_{n,p}(\K)$ with upper-rank at most $r$.
Assume that $\calV \subset \calR(1,r)$, that $n \geq 3$, that $\# \K>2$ and that
$\dim \calV \geq nr-(n-p+r)$.
Then, either $\calV$ is $r$-decomposable, or else $\# \K=3$ and $\calV$ is equivalent to $\calU_3(\K)$.
\end{prop}

\begin{Rem}\label{lift2.1remark}
Note that here we do not assume that $n \geq p$.
In some instances, we shall need to apply this result in the situation where
$p$ and $n$ are reversed. In this prospect, it is useful to note that
$nr-(n-p+r) \geq pr-(p-n+r)$ if and only if $(n-p)(r-2) \geq 0$.
\end{Rem}

\begin{proof}
Throughout the proof, we assume that $\calV$ is not $r$-decomposable
and we aim at proving that $\# \K=3$ and that $\calV$ is equivalent to $\calU_3(\K)$.

Denote by $V$ the translation vector space of $\calV$. Since $\calV \subset \calR(1,r)$,
we can split every matrix $M$ of $\Vect(\calV)$ up as
$$M=\begin{bmatrix}
[?]_{1 \times r} & C(M) \\
B(M) & [0]_{(n-1) \times (p-r)}
\end{bmatrix} \quad \text{with $B(M) \in \Mat_{n-1,r}(\K)$ and $C(M) \in \Mat_{1,p-r}(\K)$.}$$
As $\calV$ is not a subspace of $\calR(0,r)$, the map $C$ does not vanish everywhere on it.
Moreover, since $\calV$ is not $r$-decomposable, we obtain:
\begin{itemize}
\item[(A)] The space $B(\calV)$ is not $(r-1)$-decomposable.
\end{itemize}

Denote by $V'$ the linear subspace of $V$ consisting of its matrices $M$ such that $C(M)=0$.
Note that for all $M \in \calV$, we have
$$C(M) \neq 0 \Rightarrow \rk B(M) <r.$$
Note also that in $\Mat_{n-1,r}(\K)$, the lower-bound in the first classification theorem for rank-$\overline{r-1}$ spaces reads
$$(n-1)(r-1)-\bigl((n-1)-r+(r-1)\bigr)+2=(n-1)(r-1)-n+4.$$
Note finally that, by the rank theorem,
$$\dim B(V') \geq \dim \calV-p \geq (n-1)(r-1)-1.$$

Let $M_1 \in \calV$ be such that $C(M_1) \neq 0$ (note that such a matrix exists).
Then,
$$\calT:=\bigl\{M \in \calV : C(M)=C(M_1)\bigr\}$$
is an affine subspace of $\calV$ with translation vector space $V'$.
Obviously, $B(\calT)$ is a rank-$\overline{r-1}$ affine subspace of $\Mat_{n-1,r}(\K)$.
The translation vector space of $B(\calT)$ is $B(V')$.
In most cases, we shall be able to apply the first classification theorem to $B(\calT)$, and then use one of the previous forcing lemmas.

\vskip 3mm
\noindent \textbf{Case 1:} $n>4$. \\
Then, the first classification theorem applies to $B(\calT)$.
Without further loss of generality, we can then assume that
$B(V') \subset \calR(0,r-1)$, or that $n-1=r$ and $B(V') \subset \calR(r-1,0)$,

\begin{itemize}
\item \textbf{Subcase 1.1:  $B(V')\subset \calR(0,r-1)$.} \\
Applying forcing lemma 1 (i.e.\ Corollary \ref{forcing1}) to $B(\calT)$, we obtain that $B(\calT) \subset \calR(0,r-1)$.
Varying the matrix $M_1$ shows that the last column of $B(M)$ vanishes for every $M \in \calV$ for which $C(M) \neq 0$.
Thus, $\calV$ is the union of the two affine subspaces $\{M \in \calV : \; C(M)=0\}$ and $\{M \in \calV : \; B(M) \in \calR(0,r-1)\}$.
As $\# \K>2$, they cannot be both proper subspaces, and as the first one is a proper subspace we deduce that
$B(\calV) \subset \calR(0,r-1)$. This contradicts statement (A).

\item \textbf{Subcase 1.2: $B(V')\subset \calR(r-1,0)$ and $n-1=r$.} \\
Then, Corollary \ref{forcing1} shows that $B(\calT) \subset \calR(r-1,0)$.
With exactly the same line of reasoning as in Subcase 1.1, we obtain that $B(\calV) \subset \calR(r-1,0)$,
which contradicts (A).
\end{itemize}

If follows that $n \leq 4$.

\vskip 3mm
\noindent \textbf{Case 2:} $n=4$ and $r=3$. \\
Then, $\dim B(\calT) \geq 5$, and hence Proposition \ref{n=3r=2} yields that
$B(\calT)$ is equivalent to a subspace of $\calR(2,0)$ or $\calR(0,2)$
or it is equivalent to $\calR(1,1)$ (note that $\calU_3(\K)$ has dimension $4$, if $\# \K=3$).
Then, we lose no generality in assuming that $B(V')$ is included in $\calR(2,0)$, $\calR(0,2)$ or $\calR(1,1)$.
By forcing lemma 3 (Corollary \ref{forcing3}) and the same line of reasoning as in Case 1, we would obtain that $B(\calV)$ is included in $\calR(2,0)$, $\calR(0,2)$ or $\calR(1,1)$, contradicting (A).

\vskip 3mm
\noindent \textbf{Case 3:} $n=4$ and $r=2$. \\
Then, $B(V')$ is equivalent to a subspace of $\calR(1,0)$ or $\calR(0,1)$, and $\dim B(V') \geq 2$.
With the same line of reasoning as in Case 1 -- this time by using Corollary \ref{forcing2} -- we obtain that $B(\calV)$ is equivalent to a subspace of $\calR(1,0)$ or
$\calR(0,1)$, contradicting (A).

\vskip 3mm
\noindent \textbf{Case 4:} $n=3$ and $r=2$. \\
This is the only remaining case. Note that $p\geq r+1=3$.
Assume that there exists $M_0 \in V$ such that $B(M_0)=0$ and $C(M_0) \neq 0$.
For all $M \in \calV$, we see that $C(M) \neq 0$ or $C(M+M_0) \neq 0$, and as $B(M)=B(M_0+M)$ we deduce that $\rk B(M) \leq 1$.
Then, we would obtain that $B(\calV)$ is equivalent to a subspace of $\calR(0,1)$ or $\calR(1,0)$, contradicting (A) once more.
Thus, for all $M \in V$, equality $B(M)=0$ implies $C(M)=0$.

Next, we lose no generality in assuming that some matrix $M \in \calV$ is such that the first entry of $C(M)$ is non-zero.
Let us write every matrix $M \in \calV$ as $M=\begin{bmatrix}
A(M) & [?]_{3 \times (p-3)}
\end{bmatrix}$ with $A(M) \in \Mat_3(\K)$. Then, $A(\calV)$ is a rank-$\overline{2}$ affine subspace of $\Mat_3(\K)$ and
the result we have just proved shows that $\dim A(\calV)=\dim \calV$. Thus, $\dim A(\calV) \geq 4+(p-3)$
and we can apply Proposition \ref{n=3r=2} to $A(\calV)$.
We split the discussion into four subcases.
In all of them, we denote by $(e_1,e_2,e_3)$ the standard basis of $\K^3$, and we note that $A(\calV)e_3\subset \K e_1$
and $A(\calV) e_3 \neq \{0\}$.

\begin{itemize}
\item \textbf{Subcase 4.1: $A(\calV)$ is equivalent to subspace of $\calR(2,0)$.} \\
Then, we have a $2$-dimensional linear subspace $P$ of $\K^3$ such that every matrix of $A(\calV)$ has its range included in $P$.
In particular as $A(\calV)e_3$ contains a non-zero element of $\K e_1$ we find that $e_1 \in P$, and it follows that every matrix of $\calV$
has its range included in $P$. Thus, $\calV$ is equivalent to a subspace of $\calR(2,0)$, contradicting the assumption that it is not
$2$-decomposable.

\item \textbf{Subcase 4.2: $A(\calV)$ is equivalent to subspace of $\calR(0,2)$.} \\
Then, we have a non-zero vector $x \in \K^3$ such that $A(\calV)x =\{0\}$.
Since $A(\calV) e_3 \neq \{0\}$, we have $x\not\in \K e_3$, and
hence $x=y+\lambda e_3$ for some $y \in \Vect(e_1,e_2) \setminus \{0\}$ and some $\lambda \in \K$.
As $A(\calV) e_3 \subset \K e_1$ we deduce that $A(\calV)y \subset \K e_1$. Therefore,
$\calV$ is equivalent to a subspace of $\calR(1,1)$, contradicting (A).

\item \textbf{Subcase 4.3: $A(\calV)$ is equivalent to a subspace of $\calR(1,1)$.} \\
Then, we have a $2$-dimensional linear subspace $P$ of $\K^3$ together with $1$-dimensional linear subspace $D$ of $\K^3$ such that
$A(\calV) P \subset D$. If $e_1 \in D$ then we deduce that $\calV$ is equivalent to a subspace of $\calR(1,1)$.
Hence, $e_1 \not\in D$ and $e_3\not\in P$. It follows that every matrix of $\calV$ has its range included in $D+\K e_1$,
and hence $\calV$ is equivalent to a subspace of $\calR(2,0)$. Again, this contradicts (A).

\item \textbf{Subcase 4.4: $p=3$, $\# \K=3$ and $A(\calV)$ is equivalent to $\calU_3(\K)$.} \\
Then, $\calV=A(\calV)$ is equivalent to $\calU_3(\K)$, which completes the proof.
\end{itemize}
\end{proof}

\begin{prop}[Lifting lemma 2.2]\label{liftingprop2.2}
Let $\calV$ be a rank-$\overline{r}$ affine subspace of $\Mat_{n,p}(\K)$, with $n \geq p > r \geq 2$.
Assume that $\dim \calV \geq nr-(n-p+r)$, that $\calV \subset \calR(2,r-1)$ and that $\# \K>2$. \\
Then, either $\calV$ is a $r$-decomposable or
else $\# \K=3$ and $\calV$ is equivalent to $\calU_3(\K)$.
\end{prop}

\begin{proof}
As in the proof of Proposition \ref{liftingprop2.1},
we split every matrix $M$ of $\Vect(\calV)$ up as
$$M=\begin{bmatrix}
[?]_{2 \times (r-1)} & C(M) \\
B(M) & [0]_{(n-2) \times (p-r+1)}
\end{bmatrix}$$
with $B(M) \in \Mat_{n-2,r-1}(\K)$ and $C(M) \in \Mat_{2,p-r+1}(\K)$.
We denote by $V$ the translation vector space of $\calV$, and we set
$$V':=\bigl\{M \in V : \; C(M)=0\bigr\}.$$

Throughout the proof, we assume that $\calV$ is not $r$-decomposable.
It follows that $B(\calV)$ is not $(r-2)$-decomposable and that
$C(\calV)$ is not $1$-decomposable. In particular, by Proposition \ref{affinerank1},
this yields $\urk C(\calV)=2$.

Assume first that $r=2$. Then, we see that $\calV^T$ satisfies the assumptions of lifting lemma 2.1
(see Remark \ref{lift2.1remark}). However $\calV^T$ is not $r$-decomposable
since $\calV$ is not either. Thus, $\# \K=3$ and $\calV^T$ is equivalent to $\calU_3(\K)$, and hence $\calV$ is also equivalent to $\calU_3(\K)$
(see the remark underneath Proposition \ref{n=3r=2}).

In the rest of the proof, we assume that $r \geq 3$. We shall prove that this is in conflict with
the assumption that $\calV$ is not $r$-decomposable.

Let $M_1 \in \calV$ be such that $\rk C(M_1)=2$ (note that such a matrix exists).
Set
$$\calT:=\bigl\{M \in \calV : C(M)=C(M_1)\bigr\}.$$
Then, $B(V')$ is the translation vector space of $B(\calT)$. Moreover,
$$\forall M \in \calT, \; \rk B(M) \leq r-2$$
and hence Flanders's theorem for affine subspaces yields
$$\dim B(\calT) \leq (n-2)(r-2).$$
On the other hand, the rank theorem yields
$$\dim B(\calT) \geq \dim B(V') \geq \dim \calV-2p+\codim C(\calV).$$
However,
$$(n-2)(r-2)+2p-(nr-(n-p+r))=4-(n-p+r).$$
Thus,
$$(n-p)+r+\codim C(\calV) \leq 4.$$
In particular, $r \leq 4$. From there, we split the discussion into two main subcases.

\vskip 3mm
\noindent \textbf{Case 1: $r+\codim C(\calV)=4$.} \\
In particular $\codim C(\calV) \leq 1$ and $\dim B(\calT)=\dim B(V')=(n-2)(r-2)$.
Then, Flanders's theorem applies to $B(\calT)$ and shows that it is equivalent to $\calR(0,r-2)$ or to $\calR(r-2,0)$.
Thus, $B(V')=B(\calT)$. Then, $B(M) \in B(V')$ for all $M \in \calV$ such that $\rk C(M)=2$.
Yet, $C(\calV)$ is generated as an affine space by its rank $2$ matrices: indeed, if there existed an affine hyperplane $\calH$ of
$C(\calV)$ such that every matrix of $C(\calV) \setminus \calH$ has rank at most $1$, then, as $\# \K>2$, we could pick
distinct affine hyperplanes $\calH_1$ and $\calH_2$ of $C(\calV)$ that are parallel to $\calH$ and distinct from $\calH$; then,
$\dim \calH_1=\dim \calH_2 \geq (p-r+1)2-2$, and $p-r+1 \geq 2$, whence
Flanders's theorem would yield that $\calH_1$ and $\calH_2$ are both linear subspaces of $\Mat_{2,p-r+1}(\K)$, which is absurd since
they are disjoint.
As $B(V')$ is an affine subspace of $\Mat_{n-2,r-1}(\K)$, it follows that $B(\calV) \subset B(V')$.
Therefore, $B(\calV)$ is $(r-2)$-decomposable, contradicting an earlier result.

\vskip 3mm
\noindent \textbf{Case 2: $r=3$, $n=p$ and $C(\calV)=\Mat_{2,p-r+1}(\K)$.} \\
Note that $r-1=2$. We aim at proving that $n=p=4$ and $B(\calV) \subsetneq \Mat_2(\K)$.
Assume that $n>4$. Note that $\dim B(V') \geq 2$ since $n>4$ and $r+\codim C(\calV)=3$.
As $r-2=1$, Proposition \ref{affinerank1} applies to $B(\calT)$ and shows that
$B(V')$ is equivalent to a subspace of either $\calR(0,1)$ or $\calR(1,0)$.
No generality is then lost in assuming that $B(V') \subset \calR(0,1)$ or $B(V') \subset \calR(1,0)$.

Assume first that $B(V') \subset \calR(0,1)$. Then, by forcing lemma 2 (Corollary \ref{forcing2}),
we see that $B(\calT) \subset \calR(0,1)$. Thus, $B(M) \in \calR(0,1)$ for all $M \in \calV$ such that $\rk C(M)=2$.
As the affine space $C(\calV)=\Mat_{2,p-r+1}(\K)$ is generated by its rank $2$ matrices, it follows that
$B(\calV) \subset \calR(0,1)$.

Similarly if $B(V') \subset \calR(1,0)$ then one proves that $B(\calV) \subset \calR(1,0)$.
In any case, we have contradicted the assumption that $B(\calV)$ is not $1$-decomposable.

Thus, $n=4$ and $\dim B(V') \geq 1$ (and $p=4$). As $B(\calT)$ is a rank-$\overline{1}$ affine subspace we deduce from Proposition \ref{affinerank1}
than $\urk B(V') \leq 1$. Without loss of generality, we can then assume that $B(V')$ contains $E_{1,2}$.
By the extraction lemma, it follows that $B(\calT)$ is included in $T_2^+(\K)$.
Using once more the fact that $C(\calV)$ is generated as an affine space by its rank $2$ matrices, we deduce that
$B(\calV) \subset T_2^+(\K)$. Then, $\calV^T$ satisfies the same basic assumptions as $\calV$, but now
$C(\calV^T)$ is a proper subspace of $\Mat_2(\K)$. Applying Case 1 to $\calV^T$, we obtain a final contradiction.
This completes the proof.
\end{proof}

\subsection{The special lifting lemma}\label{specialliftingsection2}

\begin{prop}[Special lifting lemma 2]\label{speciallifting2}
Let $n>r>1$ be positive integers with $n \geq 4$.
Let $\calW$ be an affine subspace of $\Mat_{n,r}(\K)$ and $f : \calW \rightarrow \K^n$ be an affine map.
Assume that every matrix in
$$\calV:=\Bigl\{\begin{bmatrix}
f(N) & N
\end{bmatrix} \mid N \in \calW\Bigr\}$$
has rank at most $r$.
Assume further that $\codim \calW \leq n-1$ and $\# \K>2$.
Then, $\calV$ is $r$-decomposable.
\end{prop}

As in the proof of special lifting lemma 1, the key consists in using the theory of quasi-range-compatible maps.
Assume first that $f$ is quasi-range-compatible. Then,
we know from Theorem \ref{RCtheo2} that there are two options:
\begin{itemize}
\item Either $f$ is local, in which case $\calV$ is equivalent to a subspace of $\calR(0,r)$.
\item Or there is a $2$-dimensional linear subspace $P$ of $\K^n$, a non-zero vector $X \in\K^r$ such that $\calW X \subset P$,
and a vector $Y \in \K^r$ and a linear endomorphism $\varphi$ of $P$ such that
$$f : N \mapsto NY+\varphi(NX).$$
Without loss of generality, we can then assume that $X$ is the first vector of the standard basis and that $P=\K^2 \times \{0\}$.
Then, we write $Y=\begin{bmatrix}
y_1 & \cdots & y_r
\end{bmatrix}^T$. Performing the column operation $C_1 \leftarrow C_1-\underset{k=1}{\overset{r}{\sum}} y_k C_{k+1}$ on $\calV$, we reduce the situation to the one where $Y=0$. Then, every matrix $M$ of $\calV$ splits as
$$M=\begin{bmatrix}
[?]_{2 \times 2} & [?]_{2 \times (r-1)} \\
[0]_{(n-2) \times 2} & [?]_{(n-2) \times (r-1)}
\end{bmatrix}.$$
Thus, $\calV$ is equivalent to a subspace of $\calR(2,r-1)$. As $n>3$ and $\dim \calV \geq n(r-1)+1=nr-(n-(r+1)+r)$,
lifting lemma 2.2 (Proposition \ref{liftingprop2.1}) yields that $\calV$ is $r$-decomposable.
\end{itemize}

In the remainder of the proof, we assume that $f$ is not quasi-range-compatible.
Our aim from there is to demonstrate that $\calW$ is equivalent to a subspace of $\calR(1,r-1)$ or
of $\calR(r,0)$.

\begin{Def}
We shall say that a non-zero linear form $\varphi$ on $\K^n$ is \textbf{bad} (for $f$)
if there exists $N \in \calW$ such that $\im N \subset \Ker \varphi$ and $f(N)\not\in \Ker \varphi$.
\end{Def}

\begin{claim}\label{claim1}
The dual space of $\K^n$ possesses a basis of bad linear forms.
\end{claim}

\begin{proof}
Assuming that the contrary holds, the intersection $Z$ of the kernels of the bad linear forms
is non-zero. Then, we choose a $1$-dimensional linear subspace $D$ of $Z$, and we claim that $f$ is quasi-range-compatible with
respect to $D$. Indeed, let $N \in \calW$ be such that $D \not\subset \im N$.
Then, the orthogonal $(\im N)^o$ of $\im N$ in the dual space $(\K^n)^\star$
is not included in $D^o$. Hence, there is a basis $(\varphi_1,\dots,\varphi_k)$ of $(\im N)^o$
in which no vector belongs to $D^o$. Thus, none of the $\Ker \varphi_i$'s includes $D$,
and hence none of the linear forms $\varphi_i$ is bad.
It follows from the definition of a bad linear form that $f(N) \in \Ker \varphi_i$ for all $i \in \lcro 1,k\rcro$.
Then, $(\varphi_1,\dots,\varphi_k)$ being a basis of $(\im N)^o$, we find $\im N=\underset{i=1}{\overset{k}{\bigcap}} \Ker \varphi_i$
and hence $f(N) \in \im N$.
This contradicts the assumption that $f$ should not be quasi-range-compatible.
\end{proof}

Now, let $\varphi \in (\K^n)^\star$ be a bad linear form.
In particular, there exists $N \in \calW$ with $\im N \subset \Ker \varphi$.
We can then consider the affine subspace
$$\calW_\varphi:=\{N \in \calW : \; \im N \subset \Ker \varphi\}.$$

\begin{claim}\label{claim2}
Let $\varphi$ be a bad linear form on $\K^n$.
Then, the operator space $\calW_\varphi$ has upper-rank less than $r$.
\end{claim}

\begin{proof}
The map $g : N \in \calW_\varphi \mapsto \varphi(f(N))$ is a non-zero affine form on $\calW_\varphi$, and for all
$N \in \calW_\varphi$, we have $g(N)\neq 0 \Rightarrow \rk N <r$ since $\rk \begin{bmatrix}
f(N) & N
\end{bmatrix} \leq r$.
Thus, the result is straightforward if $g$ is constant. In the rest of the proof we assume that $g$ is non-constant.

Choose $a \in \K \setminus \{0\}$ and
consider the affine subspace $\calU_a:=g^{-1} \{a\}$ of $\calW_\varphi$. Then, by the rank theorem
$$\dim \calU_a \geq \dim \calW-r-1 \geq (n-1)(r-1)-1.$$
Note that $\calU_a$ is naturally interpreted as a subspace of $\calL(\K^r,\Ker \varphi)$ and can therefore be represented by an affine
subspace of $\Mat_{n-1,r}(\K)$.

Assume first that $n \geq 5$. Then, $\calU_a$ satisfies the assumptions of the first classification theorem,
and hence it is equivalent to a subspace of $\calR(0,r-1)$ or of $\calR(r-1,0)$.
Then, without loss of generality, we can assume that its translation vector space
$U$, which does not depend on the choice of $a$, is included in either $\calR(0,r-1)$ or $\calR(r-1,0)$.
If $U \subset \calR(0,r-1)$, then by forcing lemma 1 we must have $\calU_a \subset \calR(0,r-1)$.
As $g$ is non-constant and $\# \K>2$, we can also choose $b \in \K \setminus \{0,a\}$, yielding
$\calU_b \subset \calR(0,r-1)$. Then, as $\calU_a$ and $\calU_b$ are distinct parallel hyperplanes of $\calW_\varphi$,
their union generates the affine space $\calW_\varphi$, and hence $\calW_\varphi \subset \calR(0,r-1)$.
Similarly, if $U \subset \calR(r-1,0)$ then we obtain that $\calW_\varphi \subset \calR(r-1,0)$.

It remains to consider the case when $n=4$.
\begin{itemize}
\item Assume first that $r=3$, so that $r-1=2$. Note that $\dim \calU_a \geq 5$.
Then, by Proposition \ref{n=3r=2}, $\calU_a$ is equivalent to a subspace of $\calR(0,2)$ or $\calR(2,0)$, or it is equivalent to $\calR(1,1)$.
With the same line of reasoning as above, one uses forcing lemma 3 to obtain that
$\calW_\varphi$ is equivalent to a subspace of $\calR(0,2)$, of $\calR(2,0)$ or of $\calR(1,1)$.

\item Assume finally that $r=2$, so that $r-1=1$. Note that $\dim \calU_a \geq 2$.
Then, with the same line of reasoning as above, but using forcing lemma 2, we prove that $\calW_\varphi$
is equivalent to a subspace of $\calR(0,1)$ or of $\calR(1,0)$.
\end{itemize}
In any case we have shown that $\urk \calW_\varphi \leq r-1$, as claimed.
\end{proof}

Next, we have $\dim \calW_\varphi \geq (n-1)(r-1)$, and hence Flanders's theorem applies to $\calW_\varphi$.
In particular, it shows that it is a linear subspace, whence $\calW$ is itself a linear subspace of $\Mat_{n,r}(\K)$
(note that a bad linear form actually exists!).
Moreover, there are two mutually exclusive cases to consider:
\begin{itemize}
\item Either there exists a non-zero vector $x \in \K^r$ such that $Nx=0$ for all $N \in \calW_\varphi$,
in which case we say that $\varphi$ has \textbf{type 1};
\item Or $n=r+1$ and there exists a linear hyperplane $G$ of $\Ker \varphi$ such that $\im N \subset G$ for all $N \in \calW_\varphi$,
in which case we say that $\varphi$ has \textbf{type 2}. In that case, $\calW_\varphi$ is exactly the space of all matrices
$N \in \Mat_{n,r}(\K)$ with range included in $G$.
\end{itemize}

\begin{claim}\label{claim3}
All the bad linear forms have the same type.
\end{claim}

\begin{proof}
Assume on the contrary that we can find a bad linear form $\varphi_1$ of type $1$ and a bad linear form $\varphi_2$ of type $2$.
Obviously, $\varphi_1$ and $\varphi_2$ are non-collinear.
Note that $r=n-1$. Then, we have a linear subspace $G$ of codimension $2$ in $\K^n$ such that $\calW_{\varphi_2}$
is the set of all matrices $N \in \Mat_{n,r}(\K)$ such that $\im N \subset G$, and we have a vector $x \in \K^r \setminus \{0\}$
such that every matrix of $\calW_{\varphi_1}$ vanishes at $x$.

As $n \geq 4$, we can find a non-zero vector $y \in G \cap \Ker \varphi_1$, and then a rank $1$ matrix
$N \in \Mat_{n,r}(\K)$ with range $\K y$ and such that $N x \neq 0$. Then,
we find that $N \in \calW_{\varphi_2}$ (since $\K y \subset G$) and $\im N \subset \Ker \varphi_1$.
Then, $N \in \calW_{\varphi_1}$, contradicting $Nx \neq 0$.
\end{proof}

\begin{claim}\label{claim4}
If all the bad linear forms have type $2$, then $n=r+1$ and
$\calW$ is equivalent to $\calR(r,0)$.
\end{claim}

\begin{proof}
Assume that all the bad linear forms have type 2.
In particular, as there are bad linear forms, we must have $n=r+1$ and $\codim \calW=n-1$.
Let $\varphi$ be a bad linear form, and $G \subset \K^n$ be a linear subspace with codimension $2$ that is attached to it.
As $n \geq 3$ and as we have a basis of bad linear forms, it turns out that
we can choose another bad linear form $\varphi'$ such that $G \not\subset \Ker \varphi'$.
Then, we obtain another linear subspace $G' \subset \K^n$ with codimension $2$ that it attached to $\varphi'$, so that
$G' \subset \Ker \varphi'$.
Hence $G \neq G'$.
However, $\calW$ contains every matrix with range included in $G$, and every matrix with range included in $G'$.
We conclude that $\calW$ contains every matrix with range included in $G+G'$. Note that $\dim(G+G')\geq n-1$.
As $\codim \calW= n-1=r$, we deduce that
$G+G'$ is a linear hyperplane of $\K^n$ and that $\calW$ is the space of all matrices with range included in $G+G'$.
Thus, $\calW$ is equivalent to $\calR(r,0)$.
\end{proof}

However, if $\calW$ is equivalent to $\calR(r,0)$ and $n=r+1$, then $\calV$ is equivalent to a subspace of $\calR(r,1)$,
and hence lifting lemma 2.1 applies to $\calV^T$ (because $n=r+1$) and shows that $\calV$ is $r$-decomposable.
Thus, in the rest of the proof, we assume that all the bad linear forms have type $1$.
We aim at proving that $\calW$ is equivalent to a subspace of $\calR(1,r-1)$.

\begin{claim}\label{claim5}
There is a vector $x \in \K^r \setminus \{0\}$ such that, for every
bad linear form $\varphi$,
$$\forall N \in \calW, \; \im N \subset \Ker \varphi \Rightarrow N x=0.$$
\end{claim}

\begin{proof}
Assume on the contrary that there are two bad linear forms $\varphi_1$ and $\varphi_2$
and non-collinear vectors $x_1,x_2$ in $\K^r \setminus \{0\}$ such that
$$\forall i \in \{1,2\}, \; \forall N \in \calW, \; \im N \subset \Ker \varphi_i \Rightarrow Nx_i=0.$$
Then, we choose a complementary subspace $H$ of $\K x_1$ in $\K^r$.
It follows that, for all $N \in \calW$, if $\im N \subset \Ker \varphi_1$ and $N$ maps $H$ into $\Ker \varphi_2$, then
$Nx_1=0$ and hence $\im N \subset \Ker \varphi_2$ and $Nx_2=0$. It follows from the rank theorem that
$$\codim \calW \geq (n-1)+(n-2) \geq n.$$
This contradicts our assumptions.
\end{proof}

\begin{claim}\label{claim6}
The space $\calW x$ is a linear subspace of $\K^n$ with dimension at most $1$.
\end{claim}

\begin{proof}
We already know that $\calW$ is a linear subspace of $\Mat_{n,r}(\K)$, whence $\calW x$ is a linear subspace of $\K^n$.
Set $t:=\dim \calW x$. Let us take a basis $(\varphi_1,\dots,\varphi_n)$ of $(\K^n)^\star$ consisting of bad linear forms.
Let $i \in \lcro 1,n\rcro$ and set
$$V_i:=\{z \in \K^r \mapsto \varphi_i(Nz) \mid N \in \calW\},$$
which is a linear subspace of $(\K^r)^\star$. Then, by the factorization lemma for linear maps,
we find a linear map $L_i : V_i \rightarrow \K^n$ such that
$$\forall N \in \calW, \; Nx=L_i(z \mapsto \varphi_i(Nz)).$$
In particular, $\rk L_i=t$ for all $i \in \lcro 1,n\rcro$.
The linear map
$$N \in \calW \longmapsto \bigl(z \mapsto \varphi_i(Nz)\bigr)_{1 \leq i \leq n}$$
is injective. Its restriction to $\{N \in \calW : \; Nx=0\}$ has its range included in $\underset{k=1}{\overset{n}{\prod}} \Ker L_i$,
and hence
$$\dim \calW-t \leq \sum_{i=1}^n \dim(\Ker L_i) \leq n(r-t).$$
Thus, $\dim \calW \leq nr+t(1-n)$.
If $t \geq 2$, we deduce that
$$\dim \calW \leq nr+2(1-n),$$
which contradicts our assumptions because $nr+2(1-n) \leq n(r-1)$.
Thus, $t \leq 1$, as claimed.
\end{proof}

Finally, $\calW$ is equivalent to a subspace of $\calR(1,r-1)$, and hence $\calV$ is equivalent to a subspace of
$\calR(1,r)$. By lifting lemma 2.1 (Proposition \ref{liftingprop2.1}), the space
$\calV$ is $r$-decomposable.

This completes the proof of special lifting lemma 2.

\subsection{Wrapping the proof up}\label{refinedproofsection}

Now, we are ready to complete the proof of the refined first classification theorem.
As in the proof of the first classification theorem, we proceed by induction over $n,p,r$.
The case $r\leq 1$ is already known for all integers $n$ and $p$, by Proposition \ref{affinerank1}.
Throughout the section, we assume that $\# \K>2$.

Let $n,p,r$ be non-negative integers such that $n \geq p \geq r \geq 2$, and let $\calS$ be a rank-$\overline{r}$
affine subspace of $\Mat_{n,p}(\K)$ such that $\dim \calS \geq nr-(n-p+r)$.
Denote by $S$ its translation vector space.
If $r=p$ we simply have $\calS \subset \calR(0,r)$. Thus, in the rest of the proof, we assume that
$2 \leq r \leq p-1$. In particular $3 \leq p \leq n$. The case $n=3$ has been dealt with in Proposition \ref{n=3r=2}.
In the rest of the proof, we assume that $n \geq 4$.

Denote the upper-rank of $\calS$ by $s$. Then, $s \leq r$.
If $s$ is even and $\calS$ is equivalent to $\calR(s/2,s/2)$, then $\calS$ is $r$-decomposable.
In the rest of the proof, we assume that such is not the case.
By Lemma \ref{3rdkey}, we can then find a $1$-dimensional linear subspace $D$ of $\K^n$
such that $\dim S^D \leq \lfloor \frac{s-1}{2}\rfloor$, or a linear hyperplane $H$ of $\K^p$ such that $\dim S_H \leq \lfloor \frac{s-1}{2}\rfloor$.
Note that $\lfloor \frac{s-1}{2}\rfloor \leq r-2$.

From there, we split the discussion into four main cases.

\subsubsection{Case 1: There exists a $1$-dimensional linear subspace $D$ of $\K^n$ such that $1 \leq \dim S^D \leq \lfloor \frac{r-1}{2}\rfloor$.}

We apply the ERC method.

Without loss of generality, we can assume that $D$ is spanned by the first vector of the standard basis of $\K^n$
and that $S^D$ contains $E_{1,1}$. Then, by the extraction lemma, every matrix $M \in \calS$
splits as
$$M=\begin{bmatrix}
? & [?]_{1 \times (p-1)} \\
[?]_{(n-1) \times 1} & P(M)
\end{bmatrix} \quad \text{with $P(M) \in \Mat_{n-1,p-1}(\K)$,}$$
and $P(\calS)$ is a rank-$\overline{r-1}$ affine subspace of $\Mat_{n-1,p-1}(\K)$.
Then,
$$\dim P(\calS) \geq \dim \calS-(n-1)-\dim S^D.$$
One checks that
$$\bigl(nr-(n-p+r)\bigr)-\bigl((n-1)(r-1)-((n-1)-(p-1)+(r-1))\bigr)=(n-1)+(r-1).$$
As $\dim S^D \leq r-2$,
it follows that
$$\dim P(\calS) \geq (n-1)(r-1)-\bigl((n-1)-(p-1)+(r-1)\bigr)+1.$$
Hence, the induction hypothesis applies to $P(\calS)$.
With regards to Remark \ref{1strefinedremark}, only the following options are possible:
\begin{itemize}
\item The space $P(\calS)$ is equivalent to a subspace of $\calR(0,r-1)$. Then, $\calS$ is equivalent to a subspace of $\calR(1,r)$,
and hence lifting lemma 2.1 (Proposition \ref{liftingprop2.1}) yields the desired conclusion.
\item The space $P(\calS)$ is equivalent to a subspace of $\calR(1,r-2)$. Then, $\calS$ is equivalent to a subspace of $\calR(2,r-1)$,
and hence lifting lemma 2.2 (Proposition \ref{liftingprop2.2}) yields the desired conclusion.
\item The space $P(\calS)$ is equivalent to a subspace of $\calR(r-1,0)$. Then, $\calS$ is equivalent to a subspace of $\calR(r,1)$,
and hence lifting lemma 2.1 applied to $\calS^T$ yields the desired conclusion (see Remark \ref{lift2.1remark}).
\item The space $P(\calS)$ is equivalent to a subspace of $\calR(r-2,1)$, and $n=p$.
Then, $\calS$ is equivalent to a subspace of $\calR(r-1,2)$. By lifting lemma 2.2, the space
$\calS^T$ is $r$-decomposable (note that $n \geq 4$, which discards the exceptional solution),
and hence so is $\calS$.
\end{itemize}

\subsubsection{Case 2: There exists a linear hyperplane $H$ of $\K^p$ such that $1 \leq \dim S_H \leq \lfloor \frac{r-1}{2}\rfloor$.}

Again, we apply the ERC method. Without loss of generality, we can assume that $H=\{0\} \times \K^{p-1}$
and that $S_H$ contains $E_{1,1}$.
This time, with the same notation as in Case 1, we find
$$\dim P(\calS) \geq \dim \calS-(p-1)-\dim S_H.$$
Then, we can follow the same line of reasoning as in Case 1 because $p \leq n$.

\subsubsection{Case 3: There exists a linear hyperplane $H$ of $\K^p$ such that $S_H=\{0\}$.}

We apply the EC method.
Without loss of generality, we can assume that $H=\{0\} \times \K^{p-1}$.
Then, we split every matrix $M \in \calS$ up as
$$M=\begin{bmatrix}
[?]_{n \times 1} & J(M)
\end{bmatrix} \quad \text{with $J(M) \in \Mat_{n,p-1}(\K)$.}$$
Obviously, $J(\calS)$ is a rank-$\overline{r}$ affine subspace of $\Mat_{n,p-1}(\K)$.
Now,
$$\dim J(\calS)=\dim \calS \geq nr-(n-(p-1)+r)+1.$$
Thus, by induction we know that $J(\calS)$ is $r$-decomposable.
Note that $n>p-1$. By Remark \ref{1strefinedremark}, this leaves us only with four cases to consider.

\vskip 3mm
\noindent \textbf{Subcase 3.1: $J(\calS)$ is equivalent to a subspace of $\calR(1,r-1)$.} \\
Then, $\calS$ is equivalent to a subspace of $\calR(1,r)$, and we conclude by lifting lemma 2.1.

\vskip 3mm
\noindent \textbf{Subcase 3.2: $n=p$ and $J(\calS)$ is equivalent to a subspace of $\calR(r,0)$.} \\
Then, $\calS$ is equivalent to a subspace of $\calR(r,1)$, and hence lifting lemma 2.1 applies to
$\calS^T$, which yields the conclusion.

\vskip 3mm
\noindent \textbf{Subcase 3.3: $J(\calS)$ is equivalent to a subspace of $\calR(0,r)$ and $r<p-1$.} \\
Then, there is a non-zero vector of $\K^{p-1}$ on which all the matrices of $J(\calS)$ vanish, yielding
a non-zero vector of $\K^p$ on which all the matrices of $\calS$ vanish.
Thus, there is an affine subspace $\calT$ of $\Mat_{n,p-1}(\K)$ such that
$\calS$ is equivalent to the set of all matrices of the form $\begin{bmatrix}
N & [0]_{n \times 1}
\end{bmatrix}$ with $N \in \calT$. We have $\dim \calT=\dim \calS$, and by induction
$\calT$ is $r$-decomposable, which suffices to see that $\calS$ is $r$-decomposable.

\vskip 3mm
\noindent \textbf{Subcase 3.4: $J(\calS)$ is equivalent to a subspace of $\calR(0,r)$ and $r=p-1$.} \\
Set $\calW:=J(\calS)$. As $S_H=\{0\}$, we have
an affine map $f : \calW \rightarrow \K^n$ such that
$$\calS=\Bigl\{\begin{bmatrix}
f(N) & N
\end{bmatrix} \mid N \in \calW \Bigr\}.$$
Then, as $r=p-1$ we see that $\codim \calW \leq n-1$
and the conclusion follows from special lifting lemma 2 (that is, from Proposition \ref{speciallifting2}).

\subsubsection{Case 4: There exists a $1$-dimensional linear subspace $D$ of $\K^n$ such that $S^D=\{0\}$.}

If $n=p$, then $\calS^T$ satisfies the assumptions of Case 3, and the conclusion follows.
Assume now that $n>p$. Let us apply the ER method.
Without loss of generality, we can assume that $D$ is spanned by the first vector of the canonical basis of $\K^n$.
Then, we split every matrix $M \in \calS$ up as
$$M=\begin{bmatrix}
[?]_{1 \times p} \\
A(M)
\end{bmatrix} \quad \text{with $A(M) \in \Mat_{n-1,p}(\K)$.}$$
Note that $A(\calS)$ is a rank-$\overline{r}$ affine subspace of $\Mat_{n-1,p}(\K)$.
We still have $n-1 \geq p$, and now
$$\dim A(\calS)=\dim \calS \geq (n-1)r-((n-1)-p+r)+(r-1).$$

If $r \geq 3$, then the first classification theorem yields that $A(\calS)$ is equivalent to a subspace of
$\calR(r,0)$ or $\calR(0,r)$. If $r=2$, then we have to resort to induction:
in that case there is the extra possibility that $A(\calS)$ be equivalent to a subspace of $\calR(1,1)$.
We tackle each case separately.

\vskip 3mm
\noindent \textbf{Subcase 4.1: $A(\calS)$ is equivalent to a subspace of $\calR(0,r)$.} \\
Then, $\calS$ is equivalent to a subspace of $\calR(1,r)$, and hence lifting lemma 2.1 shows
that $\calS$ is $r$-decomposable.

\vskip 3mm
\noindent \textbf{Subcase 4.2: $A(\calS)$ is equivalent to a subspace of $\calR(1,1)$, and $r=2$.} \\
Then, $\calS$ is equivalent to a subspace of $\calR(2,1)$, and hence lifting lemma 2.2 shows
that $\calS$ is $r$-decomposable.

\vskip 3mm
\noindent \textbf{Subcase 4.3: $A(\calS)$ is equivalent to a subspace of $\calR(r,0)$, and $p=n-1$.} \\
Then, as $r<p$ we learn that $\calS$ is equivalent to a subspace of $\calR(n-1,0)$.
It follows that we can find a rank-$\overline{r}$ affine subspace $\calT$ of $\Mat_{n-1,p}(\K)$ such that
$\calS$ is equivalent to the space of all matrices $\begin{bmatrix}
N \\
[0]_{1 \times p}
\end{bmatrix}$ with $N \in \calT$. By induction
(note that $\dim \calT>(n-1)r-((n-1)-r+p)$, which discards the exceptional solution), we learn that $\calT$ is
$r$-decomposable, and it follows that $\calS$ is $r$-decomposable.

\vskip 3mm
This completes the proof of the refined first classification theorem.

\subsection{A corollary: the last forcing lemma}

We finish this section by deriving our last forcing lemma from the refined first classification theorem.

\begin{cor}[Forcing lemma 4]\label{forcing4}
Let $n$ and $r$ be positive integers such that $n \geq r+1$.
Let $\calV$ be a rank-$\overline{r}$ affine subspace of $\Mat_{n,r+1}(\K)$, with translation vector space $V$.
Assume that $\dim \calV \geq nr-n+2$ and that $\# \K>2$.
\begin{enumerate}[(a)]
\item If $V \subset \calR(r,0)$ then $\calV \subset \calR(r,0)$.
\item If $V \subset \calR(0,r)$ then $\calV \subset \calR(0,r)$.
\item If $V =\calR(1,r-1)$ then $\calV \subset \calR(1,r-1)$.
\item If $V =\calR(r-1,1)$ then $\calV \subset \calR(r-1,1)$.
\end{enumerate}
\end{cor}

\begin{proof}
First of all, if $\calV$ is equivalent to a subspace of $\calR(1,r-1)$ or of $\calR(r-1,1)$,
then by $\dim \calV \geq nr-n+2$ we obtain that $\calV$ is equivalent to $\calR(1,r-1)$ or to $\calR(r-1,1)$,
and hence $\calV=V$: in that case all four implications are obvious.
The same line of reasoning applies to the case when $n>r+1$ and $\calV$ is equivalent to a subspace of $\calR(r,0)$.

In the rest of the proof, we assume that $\calV$ is inequivalent to a subspace of $\calR(1,r-1)$ or of $\calR(r-1,1)$,
and that if $n>r+1$ then $\calV$ is inequivalent to a subspace of $\calR(r,0)$.

The refined first classification theorem applies to $\calV$.
As $\dim \calV >nr-(n-(r+1)+r)$, the only possibilities (see Remark \ref{1strefinedremark})
are that $\calV$ be equivalent to a subspace of $\calR(0,r)$, $\calR(r,0)$, $\calR(1,r-1)$ or $\calR(r-1,1)$.
Hence, $\calV$ is equivalent to a subspace of $\calR(0,r)$, or
it is equivalent to a subspace of $\calR(r,0)$ and $n=r+1$. In any case $r \geq 2$. 

Assume for the time being that $V$ is equivalent to a subspace of $\calR(0,r)$, yielding a non-zero vector $x \in \K^{r+1}$
such that $\forall M \in V, \; Mx=0$.
\begin{itemize}
\item Assume further that $n=r+1$ and that $\calV$ is equivalent to a subspace of $\calR(r,0)$.
Then, the same is true of $V$, and hence we have that every matrix of $V$ vanishes at $x$
all the while having its image included in a fixed $r$-dimensional linear subspace of $\K^n$, which leads to
$\dim V \leq r^2$. Yet,
$$nr-n+2-r^2 \geq (r+1)(r-1)+2-r^2=1.$$
This contradicts our assumptions on the dimension of $\calV$.

\item Assume that $V$ is equivalent to a subspace of $\calR(1,r-1)$. Then, as the dimensions are equal we find that
$V$ is equivalent to $\calR(1,r-1)$, contradicting the fact that all the matrices of $V$ vanish at $x$.

\item Similarly, if $V$ is equivalent to a subspace of $\calR(r-1,1)$ then it is equivalent to $\calR(r-1,1)$, contradicting the
fact that all the matrices of $V$ vanish at $x$.
\end{itemize}
Thus, if $V$ is equivalent to a subspace of $\calR(0,r)$, it can neither be equivalent to a subspace of $\calR(r,0)$, nor to a subspace of
$\calR(1,r-1)$, nor to a subspace of $\calR(r-1,1)$.
By transposing, we see that if $n=r+1$ and $V$ is equivalent to a subspace of $\calR(r,0)$, then
it can neither be equivalent to a subspace of $\calR(1,r-1)$ nor to a subspace of $\calR(r-1,1)$.

We are now ready to conclude.

\begin{enumerate}[(a)]
\item Assume first that $V \subset \calR(0,r)$. By the above, $\calV$ must be equivalent to a subspace of $\calR(0,r)$.
Hence, we have a non-zero vector $x \in \K^{r+1}$ such that
$Mx=0$ for all $M \in \calV$. Then, $Mx=0$ for all $M \in V$. Yet, $Me_{r+1}=0$ for all $M \in V$, where
$e_{r+1}$ denotes the last vector of the standard basis of $\K^{r+1}$. If $e_{r+1}$ and $x$ were not collinear, this
would lead to $\dim V \leq n(r-1)$, contradicting the fact that $n(r-1)<nr-n+2$.
Hence, $e_{r+1} \in \K x$, and every matrix of $\calV$ vanishes at $e_{r+1}$: therefore,
$\calV \subset \calR(0,r)$, as claimed.

\item Assume that $V \subset \calR(r,0)$. By a previous step, $\calV$ must be equivalent to a subspace of $\calR(r,0)$,
and then $n=r+1$. Then, the result follows from point (a) applied to $\calV^T$.

\item Assume that $V \subset \calR(1,r-1)$.
If $n>r+1$, then $\calV$ must be equivalent to a subspace of $\calR(0,r)$,
and hence $V$ is also equivalent to a subspace of $\calR(0,r)$, which contradicts an earlier statement.
If $n=r+1$, then $\calV$ is equivalent to a subspace of $\calR(r,0)$ or $\calR(0,r)$, and hence the same holds for $V$:
in any case we find a contradiction with an earlier step.

\item Finally, if $V \subset \calR(r-1,1)$ then as $\dim V \geq nr-n+2$
 we must have $n=r+1$, and we derive a contradiction by applying the previous
step to $\calV^T$.
\end{enumerate}
This completes the proof of forcing lemma 4.
\end{proof}

\section{Proof of the second classification theorem}\label{2ndclasssection}

This section is devoted to the proof of the second classification theorem.
Its structure is globally similar to the one of the preceding section, with a couple of exceptions.
We shall start by classifying rank-$\overline{3}$ linear subspace of upper-triangular
$4$ by $4$ matrices with dimension greater than or equal to $8$ (Section \ref{triang4section}).
This result will be used in the proof of the three lifting lemmas (Section \ref{lastliftingsection}).
Then, we will complete the classification of rank-$\overline{3}$ linear spaces of $4$ by $4$ matrices
with dimension greater than or equal to $8$ (Section \ref{4by4class}).
Afterwards, we will prove the special lifting lemma (Section \ref{lastspeciallifting}).
Finally, we will prove the second classification theorem by induction (Section \ref{wrapupsecond}),
with the help of all the previous results.

\subsection{The simple case of $4$ by $4$ upper-triangular matrices}\label{triang4section}

\begin{lemma}\label{triang4}
Let $V$ be a rank-$\overline{3}$ linear subspace of $T_4^+(\K)$ such that $\dim V \geq 8$ and $\# \K>2$.
Then, either $V$ is $3$-decomposable, or $\# \K=3$ and
$V$ is equivalent to $\calU_4(\K)$.
\end{lemma}

\begin{proof}
Assume that $V$ is not $3$-decomposable.

We consider the space $\Delta \subset \K^4$ consisting of the diagonal vectors of the matrices in $V$.
As $\dim V\geq 8$ we must have $\dim \Delta \geq 2$.
For $x \in \Delta$, let us write
$$x=\begin{bmatrix}
a_1(x) & a_2(x) & a_3(x) & a_4(x)
\end{bmatrix}^T.$$
The $a_i$ maps are linear forms on $\Delta$. None of them is zero for the contrary would yield that
$V$ is equivalent to a subspace of $\calR(0,3)$, $\calR(1,2)$, $\calR(2,1)$ or $\calR(3,0)$.
On the other hand $\dim \Delta <4$ as $\Delta$ does not contain $\begin{bmatrix}
1 & 1 & 1 & 1
\end{bmatrix}^T$ (indeed, every matrix of $V$ is singular).

Assume that $\dim \Delta=3$. Since the intersection of the kernels of the $a_i$ maps
is zero, we can extract a basis of the dual space of $\Delta$ from $(a_1,a_2,a_3,a_4)$.
Then, in the pre-dual basis of $\Delta$, the polynomial mapping $x \mapsto a_1(x)a_2(x)a_3(x)a_4(x)$ has degree less than $3$ in each variable,
and as it vanishes everywhere on $\Delta$ we obtain a contradiction from the fact that the $a_i$'s are all non-zero and that $\# \K>2$.

It follows that
$$\dim \Delta=2.$$
Hence, $\Delta=\underset{k=1}{\overset{4}{\bigcup}} \Ker a_k$, and it follows that
the $2$-dimensional space $\Delta$ is the union of four of its $1$-dimensional linear subspaces.
In turn, this shows that $\# \K=3$ and that the kernels $\Ker a_k$ are exactly the four $1$-dimensional linear subspaces of $\Delta$.

We deduce that, for some $(a,b)\in (\K \setminus \{0\})^2$, either
$$\Delta=\Bigl\{\begin{bmatrix}
x & y & a(x+y) & b (x-y)
\end{bmatrix}^T \mid (x,y)\in \K^2\Bigr\}$$
or
$$\Delta=\Bigl\{\begin{bmatrix}
x & y & a(x-y) & b (x+y)
\end{bmatrix}^T \mid (x,y)\in \K^2\Bigr\}.$$
In any case, we see by multiplying rows by appropriate scalars that $V$ is equivalent to a subspace of $\calU_4(\K)$.
Since $\dim \calU_4(\K)=8$ we conclude that $V$ is equivalent to $\calU_4(\K)$.
\end{proof}

\subsection{Lifting lemmas}\label{lastliftingsection}

\begin{prop}[Lifting lemma 3.1]\label{liftingprop4}
Let $n,p,r$ be positive integers such that $r<\min(n,p)$.
Let $V$ be a rank-$\overline{r}$ linear subspace of $\Mat_{n,p}(\K)$ such that
$\dim V \geq nr-2(n-p+r)+2+\epsilon(\K)$ and $V \subset \calR(1,r)$. \\
Then, either $V$ is $r$-decomposable, or $\# \K=3$ and $V$ is equivalent to $\calU_4(\K)$.
\end{prop}

\begin{proof}
We write every matrix $M$ of $V$
as
$$M=\begin{bmatrix}
[?]_{1 \times r} & C(M) \\
B(M) & [0]_{(n-1) \times (p-r)}
\end{bmatrix} \quad \text{with $B(M) \in \Mat_{n-1,r}(\K)$ and $C(M) \in \Mat_{1,p-r}(\K)$.}$$
Throughout the proof, we assume that $V$ is not $r$-decomposable,
and we seek to prove that $\# \K=3$ and that $V$ is equivalent to $\calU_4(\K)$.
In particular, $C(V) \neq \{0\}$ for the contrary would yield $V \subset \calR(0,r)$, and $B(V)$
is not $(r-1)$-decomposable.

Set
$$V':=\bigl\{M \in V : \; C(M)=0\bigr\}.$$
Let $M_1 \in C(V)$ be such that $C(M_1) \neq 0$ (such a matrix exists), and set
$$\calT:=M_1+V'.$$
Note that for all $M \in \calT$ we have $C(M)\neq 0$, whence $\rk B(M)<r$.

We have
$$\dim B(V') \geq \dim V-p
\geq nr-2n-2r+p+2+\epsilon(\K).$$
On the other hand,
$$(n-1)(r-1)-\bigl((n-1)-r+(r-1)\bigr)
=nr-2n-r+3.$$
Thus,
\begin{equation}\label{lift31inequality}
\dim B(\calT) \geq \bigl((n-1)(r-1)-((n-1)-r+(r-1))\bigr)+\bigl(p-r-1+\epsilon(\K)\bigr).
\end{equation}
Thus, the refined first classification theorem applies to $B(\calT)$ if $\# \K>2$, and the first classification theorem
applies to it if $\# \K=2$.

\vskip 3mm
\noindent \textbf{Case 1: $B(\calT)$ is equivalent to a subspace of $\calR(i,r-1-i)$ for some
$i \in \{0,1,r-2,r-1\}$.}

Without loss of generality we can then assume that $B(V')$ is included in $\calR(i,r-1-i)$ for some $i \in \{0,1,r-2,r-1\}$ which we now fix.

\noindent \textbf{Subcase 1.1: $p=r+1$.} \\
Then, $B(\calT)$ spans $B(V)$, and we deduce that $B(V) \subset \calR(i,r-1-i)$, contradicting the fact that
$B(V)$ is not $(r-1)$-decomposable.

\noindent \textbf{Subcase 1.2: $p>r+1$.} \\
Then, we note that $\dim B(\calT) \geq (n-1)(r-1)-\bigl((n-1)-r+(r-1)\bigr)+1+\epsilon(\K)$.
Assume for instance that $B(V') \subset \calR(0,r-1)$.
Then, by forcing lemma 4 if $\# \K>2$, and by forcing lemma 1 otherwise, we learn that $B(\calT) \subset \calR(0,r-1)$.
Varying the matrix $M_1$ we started from yields that $B(V) \subset \calR(0,r-1)$.

Likewise, if $B(V')$ is included in, respectively, $\calR(1,r-2)$, $\calR(r-2,1)$ or $\calR(r-1,0)$,
then one shows that the same holds for $B(V)$.

In any case, we have contradicted the fact that $B(V)$ is not $(r-1)$-decomposable.

\vskip 3mm
\noindent \textbf{Case 2: $r=5$, $\# \K>2$ and $B(\calT)$ is equivalent to $\calR(2,2)$.} \\
Then, $B(\calT)$ is a linear subspace, which yields a matrix $M_2$ such that $B(M_2)=0$
and $C(M_2) \neq 0$. Then, for all $M \in V$, we can find a scalar $\lambda$ such that
$C(M+\lambda M_2) \neq 0$, so that $B(M)=B(M+\lambda M_2)$ has rank less than $r$.
It follows that $\urk B(V) \leq r-1$. Then, as $\calR(2,2)$ is a maximal rank-$\overline{4}$ linear subspace of $\Mat_{n-1,5}(\K)$,
we deduce that $B(V)$ is equivalent to $\calR(2,2)$, contradicting our assumptions.

\vskip 3mm
\noindent \textbf{Case 3: $\# \K=3$ and $B(\calT)$ is equivalent to $\calU_3(\K)$.} \\
Without loss of generality, we can assume that $B(\calT) \subset T_3^+(\K)$.
Then, inequality \eqref{lift31inequality} shows that $p=r+1$, and hence $B(\calT)$ spans $B(V)$.
It follows that $B(V) \subset T_3^+(\K)$. By permuting columns, we deduce that
$V$ is equivalent to a subspace of $T_4^+(\K)$. Then, it follows from
Lemma \ref{triang4} that $V$ is equivalent to $\calU_4(\K)$, which completes the proof.
\end{proof}

\begin{prop}[Lifting lemma 3.2]\label{liftingprop5}
Let $n,p,r$ be positive integers such that $r <\min(n,p)$.
Let $V$ be a rank-$\overline{r}$ linear subspace of $\Mat_{n,p}(\K)$ such that
$\dim V \geq nr-2(n-p+r)+2+\epsilon(\K)$ and $V \subset \calR(2,r-1)$. \\
Then, either $V$ is $r$-decomposable, or
$\# \K=3$ and $V$ is equivalent to $\calU_4(\K)$.
\end{prop}

\begin{proof}
We split every matrix $M$ of $V$
up as
$$M=\begin{bmatrix}
[?]_{2 \times (r-1)} & C(M) \\
B(M) & [0]_{(n-2) \times (p-r+1)}
\end{bmatrix}$$
with $B(M) \in \Mat_{n-2,r-1}(\K)$ and $C(M) \in \Mat_{2,p-r+1}(\K)$. Set
$$V':=\bigl\{M \in V : \; C(M)=0\bigl\}.$$

Throughout the proof, we assume that $V$ is not $r$-decomposable and we aim at proving that $\# \K=3$ and that $V$ is equivalent to $\calU_4(\K)$.

It follows that $C(V)$ is not $1$-decomposable and that $B(V)$ is not $(r-2)$-decomposable.
In particular $\urk C(V)=2$, owing to the classification of vector spaces of matrices with rank at most $1$.
Note that, if $\# \K>2$, this yields that $C(V)$ is spanned by its rank $2$ matrices.

\vskip 3mm
\noindent \textbf{Step 1: $\urk B(V)=r-1$.} \\
Assume on the contrary that $\urk B(V) \leq r-2$.
Then, $r \geq 4$ for the contrary would yield that $B(V)$ is $(r-2)$-decomposable.
It follows that $n \geq 5$.

In the first classification theorem for rank-$\overline{r-2}$ subspaces of $\Mat_{n-2,r-1}(\K)$, the lower bound is
$$(n-2)(r-2)-(n-2-(r-1)+r-2)+2=nr-3n-2r+9.$$
Moreover,
$$\dim B(V) \geq \dim V-2p \geq nr-2n-2r+\epsilon(\K)+2.$$
If $n \geq 7-\epsilon(\K)$ then the first classification theorem would yield that $B(V)$ is equivalent to a subspace of
$\calR(r-2,0)$ or $\calR(0,r-2)$, contradicting an early result.
Thus, $n \leq 6-\epsilon(\K)$. In particular, this shows that $\# \K>2$.
Then, the refined first classification theorem yields that $B(V)$ is $(r-2)$-decomposable (note that
the special case of $\calU_3(\K)$ is discarded as it is not a linear subspace), a contradiction.
This completes our first step.

\vskip 3mm
Let us start from an arbitrary matrix $M_1 \in V$ such that $\rk C(M_1)=2$, and let us consider the
affine space
$$\calT:=\{M \in V : \; C(M)=C(M_1)\}.$$
Note that the translation vector space of $\calT$ is $V'$ and that $\urk B(\calT) \leq r-2$.
In particular, using Step 1,
$$B(V') \subsetneq B(V).$$
On the other hand
$$\dim B(V')=\dim B(\calT) \geq \dim V-2p+\codim C(V)$$
and hence
$$\dim B(V') \geq nr-2n-2r+\epsilon(\K)+2+\codim C(V).$$
Moreover, Flanders's theorem yields
$$\dim B(\calT) \leq (n-2)(r-2).$$

\vskip 3mm
\noindent \textbf{Step 2: $\# \K>2$.} \\
Assume on the contrary that $\# \K=2$. Then, $\epsilon(\K)=2$ and we deduce from the above that $\codim C(V)=0$,
i.e.\ $C(V)=\Mat_{2,p-r+1}(\K)$, and that $\dim B(V')=(n-2)(r-2)$.
Note that $C(V)$ is spanned by its rank $1$ matrices. As $B(V') \subsetneq B(V)$, it follows
that we can find a matrix $M_2 \in V$ such that $\rk C(M_2)=1$ and $B(M_2) \not\in B(V')$.
By \cite[Lemma 1.2]{dSPfullranklines}, we can choose $M_1 \in V$ such that every matrix of $C(M_1)+\K C(M_2)$
has rank $2$. Then, instead of $\calT$ we consider the affine space
$\calT':=\bigl\{M \in V : \; C(M)\in C(M_1)+\K C(M_2)\bigr\}$.
Then, $\urk B(\calT')\leq r-2$. On the other hand the translation vector space of $B(\calT')$ includes $B(V') \oplus \K B(M_2)$,
whence $\dim B(\calT') >(n-2)(r-2)$, contradicting Flanders's theorem.
Therefore, $\# \K>2$, as claimed.

\vskip 3mm
\noindent \textbf{Step 3: $n \leq 5-\codim C(V)$.} \\
Assume on the contrary that $n \geq 6-\codim C(V)$.
Combining this with the above inequalities yields
$$\dim B(V') \geq (n-2)(r-2)-\bigl((n-2)-(r-1)+(r-2)\bigr)+1.$$
Then, we use the forcing method.
The refined first classification theorem applies to $B(\calT)$, and hence
without loss of generality we can assume either that $B(V')$ is included in $\calR(0,r-2)$ or $\calR(r-2,0)$,
or that $B(V')=\calR(1,r-3)$ or $B(V')=\calR(r-3,1)$.
Assume for instance that the first case holds. Then, one deduces from point (a) of forcing lemma 4
(that is Corollary \ref{forcing4}) that $B(\calT)$ is included in $\calR(0,r-2)$ whatever the choice of $M_1$.
As $C(V)$ is spanned by its rank $2$ matrices, it follows that $B(V) \subset \calR(0,r-2)$, contradicting an earlier result.
Similarly, any of the other three cases yields a contradiction, this time by applying one of points (b) to (d) from forcing lemma 4.
Hence, $n \leq 5-\codim C(V)$, as claimed.

\vskip 3mm
\noindent \textbf{Step 4: $n \leq 4$.} \\
Assume on the contrary that $n \geq 5$. Then, by Step 3 we find
$n=5$ and $C(V)=\Mat_{2,p-r+1}(\K)$.
As in Step 2, we can choose $M_2 \in V$ such that $B(M_2) \not\in B(V')$ and $\rk C(M_2)=1$, and then we choose
$M_1 \in V$ such that every matrix of $C(M_1)+\K C(M_2)$ has rank $2$.
Set $\calT':=\bigl\{M \in V : \; C(M)\in C(M_1)+\K C(M_2)\bigr\}$.
Then, $\urk B(\calT') \leq r-2$ and $\dim B(\calT')>(n-1)(r-2)-((n-1)-(r-1)+r-2)$,
whence the refined first classification theorem yields that $B(\calT')$ is $(r-2)$-decomposable.
Hence, every matrix in either $B(\calT')$ or its translation vector space has rank less than or equal to $r-2$, and in particular
$\rk B(M_2) \leq r-2$ and $\urk B(V') \leq r-2$. Then, for all $M \in V$ such that $\rk C(M) \leq 1$, either $B(M) \in B(V')$ and hence
$\rk B(M) \leq r-2$, or $B(M) \not\in B(V')$ and hence taking $M_2:=M$ in the above yields
$\rk B(M) \leq r-2$. Finally, if $\rk C(M)=2$ then we readily find $\rk B(M) \leq r-2$.
Therefore, $\urk B(V) \leq r-2$, contradicting Step 1.

\vskip 3mm
\noindent \textbf{Step 5: $n=4$.} \\
Assume on the contrary that $n\leq 3$. Note that $r \leq 1$ would lead to $V$ being $r$-decomposable,
by Proposition \ref{affinerank1}. Hence, $r=2$.
Thus, $B(M)=0$ for all $M \in V$ such that $\rk C(M)=2$. As $C(V)$ is spanned by its rank $2$ matrices, this leads to $B(V)=\{0\}$,
 contradicting Step 1. Therefore, $n=4$.

\vskip 2mm
Note, as $n=4$, that $\codim C(V) \leq 1$.

\vskip 3mm
\noindent \textbf{Step 6: $p=r+1$ and $C(V)$ is equivalent to $T_2^+(\K)$.} \\
Assume that for every rank $1$ matrix $N \in C(V)$, there exists a matrix $N' \in C(V)$ such that every matrix of
$N'+\K N$ has rank $2$. Let $N \in C(V)$ be of rank $1$ (note that such a matrix exists since $\codim C(V) \leq 1$).
Then, we choose a matrix $N' \in C(V)$ such that every matrix of
$N'+\K N$ has rank $2$. The affine space $\calT':=\bigl\{M \in V : \;
C(M) \in N'+\K N\bigr\}$ satisfies $\urk B(\calT') \leq r-2$.
As $r-2 \leq 1$, we deduce that every matrix in the translation vector space of $B(\calT')$ has rank at most $r-2$
(this is obvious if $r-2=0$, otherwise one can use the classification of affine spaces of matrices with upper-rank at most $1$).
Then, $\rk B(M) \leq r-2$ for every $M \in V$ such that $C(M) \in \K N$. Hence $\rk B(M) \leq r-2$ for every $M \in V$ such that $\rk C(M) \leq 1$. Yet,
as this was known to hold for all the other matrices of $V$, we conclude that $\urk B(V) \leq r-2$, contradicting
Step 1.

Thus, it is not true that for every rank $1$ matrix $N \in C(V)$, there exists a matrix $N' \in C(V)$ such that every matrix of
$N'+\K N$ has rank $2$. By \cite[Lemma 1.2]{dSPfullranklines} this requires that $\codim C(V)=1$, i.e.\ $C(V)$ is a linear hyperplane of $\Mat_{2,p-r-1}(\K)$. Moreover, Theorem 1.3 of \cite{dSPfullranklines} applied to $C(V)^T$ then shows that $p-r+1=2$, that is $p=r+1$.

Thus, $C(V)$ is a linear hyperplane of $\Mat_2(\K)$. Then, by replacing $C(V)$ with an equivalent subspace,
we are reduced to two cases: either $C(V)$ is equivalent to $\Mats_2(\K)$ (the space of all $2$ by $2$ symmetric matrices)
or it is equivalent to $T_2^+(\K)$. Yet, for every rank $1$ matrix $N$ of
$\Mats_2(\K)$, there exists a rank $2$ matrix $N' \in \Mats_2(\K)$ such that every matrix of $N'+\K N$ has rank $2$:
indeed, we can assume that $N=a E_{1,1}$ for some $a \in \K \setminus \{0\}$ (as $N$ must be congruent to such a matrix),
in which case it suffices to take $N':=E_{2,1}+E_{1,2}$. Hence, $C(V)$ is equivalent to $T_2^+(\K)$, as claimed.

\vskip 3mm
\noindent \textbf{Step 7: $r=3$, $\# \K=3$ and $V$ is equivalent to $\calU_4(\K)$.} \\
As $C(V)$ is equivalent to $T_2^+(\K)$, the space $V$ is equivalent to a subspace of $\calR(1,r)$,
and hence lifting lemma 3.1 yields that $\# \K=3$ and $V$ is equivalent to $\calU_4(\K)$.
\end{proof}

\begin{Rem}\label{liftingremark3}
In lifting lemmas 3.1 and 3.2, we did not assume that $n \geq p$.
In some instances, we shall need to apply these results to situations where
$p$ and $n$ are reversed. In this prospect, it is useful to note that
$nr-2(n-p+r) \geq pr-2(p-n+r)$ if and only if $(n-p)(r-4) \geq 0$.
Moreover, $\calU_4(\K)$ is easily seen to be equivalent to its transpose:
indeed, by setting
$$K:=\begin{bmatrix}
0 & 0 & 0 & 1 \\
0 & 0 & 1 & 0 \\
0 & 1 & 0 & 0 \\
1 & 0 & 0 & 0
\end{bmatrix} \quad \text{and} \quad D:=\begin{bmatrix}
-1 & 0 & 0 & 0 \\
0 & 1 & 0 & 0 \\
0 & 0 & -1 & 0 \\
0 & 0 & 0 & 1
\end{bmatrix},$$
one checks that $K \calU_4(\K)^T KD=\calU_4(\K)$.
\end{Rem}

\begin{prop}[Lifting lemma 3.3]\label{liftingprop6}
Let $n,p,r$ be positive integers such that $2 \leq r<p \leq n$. Assume furthermore
that $n \geq 6$ if $r\geq 4$.
Let $V$ be a rank-$\overline{r}$ linear subspace of $\Mat_{n,p}(\K)$ such that
$\dim V \geq nr-2(n-p+r)+2+\epsilon(\K)$ and $V \subset \calR(3,r-2)$. \\
Then, either $V$ is $r$-decomposable, or $\# \K=3$ and $V$ is equivalent to $\calU_4(\K)$.
\end{prop}

\begin{proof}
Throughout the proof, we assume that the stated conclusion does not hold, and we
seek to find a contradiction.
We split every matrix $M$ of $V$ up as
$$M=\begin{bmatrix}
[?]_{3 \times (r-2)} & C(M) \\
B(M) & [0]_{(n-3) \times (p-r+2)}
\end{bmatrix}$$
with $B(M) \in \Mat_{n-3,r-2}(\K)$ and $C(M) \in \Mat_{3,p-r+2}(\K)$.
Set
$$V':=\bigl\{M \in V : \; C(M)=0\bigr\}.$$
As $V$ is not $r$-decomposable, we note that
$C(V)$ is not $2$-decomposable and that, if $r \geq 3$, the space
$B(V)$ is not $(r-3)$-decomposable.

\vskip 3mm
\noindent \textbf{Step 1: $\urk C(V)=3$.} \\
Assume on the contrary that $\urk C(V) \leq 2$.
Note that
$$\dim C(V) \geq \dim V-n(r-2) \geq 2(p-r)+2+\epsilon(\K) \geq 4+\epsilon(\K).$$
Then, by the classification of rank-$\overline{2}$ vector spaces (see Section 4 of \cite{AtkinsonPrim} for fields
with more than $2$ elements, and \cite{dSPprimitiveF2} for fields with two elements),
we find that $C(V)$ is $2$-decomposable, contradicting an earlier result.
Therefore, $\urk C(V)=3$.

In particular, it follows that $r \geq 3$ and $n \geq 4$.

\vskip 3mm
In the rest of the proof, we fix a matrix $M_1 \in V$ such that $\rk C(M_1)=3$, and we consider the affine subspace
$$\calT:=\bigl\{M \in V : \; C(M)=C(M_1)\bigr\},$$
whose translation vector space equals $V'$.
Note that $\urk B(\calT) \leq r-3$ and that $B(V')$ is the translation vector space of $B(\calT)$.

\vskip 3mm
\noindent \textbf{Step 2: $C(V)$ is spanned by its rank $3$ matrices.} \\
Assume that the contrary holds. Then, we have an affine hyperplane $\calH$ of $C(V)$ that does not go through
zero and which contains only matrices with rank less than $3$. By Flanders's theorem, we deduce that
$\codim C(V) \geq p-(r-2)$, unless $\# \K=2$ in which case we can only assert that
$\codim C(V) \geq (p-(r-2))-1$.
On the other hand, Flanders's theorem applied to $B(\calT)$ yields
$$\dim B(V')=\dim B(\calT) \leq (n-3)(r-3).$$
Yet, the rank theorem yields
$$\dim V \leq \dim B(V')+3p-\codim C(V).$$
As $\dim V \geq nr-2(n-p+r)+2+\epsilon(\K)$, we deduce that
$$\begin{cases}
n \leq 5 & \text{if $\# \K>2$} \\
n \leq 4 & \text{if $\# \K=2$.}
\end{cases}$$
If $n \leq 4$, then $n=p=4$ and $r=3$. In that case, as $n=p$ we see that $V^T$ satisfies the assumptions
of lifting lemma 3.1, and we conclude that $\# \K=3$ and that $V^T$
is equivalent to $\calU_4(\K)$, whence $V$ is equivalent to $\calU_4(\K)$ (see Remark \ref{liftingremark3}), contradicting our assumptions.
Thus, $n=5$ and $\# \K>2$. From our basic assumptions, we deduce that $r=3$.
Going back to the above line of reasoning we find that
$\codim \calH=p-(r-2)+1$. Then, by the refined first classification theorem,
$\calH$ must be $2$-decomposable, and hence
so is $C(V)=\Vect(\calH)$, which contradicts our assumptions.
Therefore, $C(V)$ is spanned by its rank $3$ matrices.

\vskip 3mm
\noindent \textbf{Step 3: $r \geq 5$.} \\
Assume first that $r=3$. Then, $B(M)=0$ for every $M \in V$ such that $C(M)$ has rank $3$.
As $C(V)$ is spanned by its rank $3$ matrices, it follows that $B(V)=\{0\}$, contradicting the fact that
$B(V)$ is not $(r-3)$-decomposable.
Hence, $r \geq 4$. If $r=4$, then we note that $V^T$ satisfies the assumptions of lifting lemma 3.2
(see Remark \ref{liftingremark3}), and we obtain a contradiction just like in Step 2. Hence, $r \geq 5$.

In particular, it follows that $n \geq 6$.

\vskip 3mm
Note that, since $\dim B(\calT) \geq \dim V-3p+\codim C(V)$, we have
\begin{multline*}
\dim B(\calT)-\Bigl((n-3)(r-3)-\bigl((n-3)-(r-2)+(r-3)\bigr)+1+\frac{\epsilon(\K)}{2}\Bigr) \\
\geq n+r+(n-p)-12+\frac{\epsilon(\K)}{2}+\codim C(V).
\end{multline*}

\vskip 3mm
\noindent \textbf{Step 4: $n+r+(n-p)+\codim C(V)+\frac{\epsilon(\K)}{2} \leq 11$.} \\
Assume on the contrary that $n+r+(n-p)+\codim C(V)+\frac{\epsilon(\K)}{2} \geq 12$. \\
Then, we can apply the refined first classification theorem to $B(\calT)$ if $\# \K>2$,
and the first classification theorem to $B(\calT)$ is $\# \K=2$ (note that in the former case the
exceptional situation of $\calU_3(\K)$ is avoided because the dimension of $B(\calT)$ is greater than the lower bound from the
refined first classification theorem).
Then, no generality is lost in assuming that $B(V')$ is included in either one of $\calR(0,r-3)$, $\calR(1,r-4)$, $\calR(r-4,1)$ or
$\calR(r-3,0)$ (and in the second and third cases, that $B(V')$ equals the given compression space).
By applying either forcing lemma 4 if $\# \K>2$, or forcing lemma 1 if $\# \K=2$, we see that if $B(V') \subset \calR(0,r-3)$
then $B(M) \in \calR(0,r-3)$ for all $M \in V$ such that $\rk C(M)=3$; then, as $C(V)$ is spanned by its rank $3$ matrices
we deduce that $B(V) \subset \calR(0,r-3)$. With the same line of reasoning, we obtain in any case that $B(V)$ is $(r-3)$-decomposable, contradicting an earlier result. Hence, $n+r+(n-p)+\codim C(V)+\frac{\epsilon(\K)}{2} \leq 11$.

\vskip 3mm
\noindent \textbf{Step 5: $\# \K>2$, $n=p=6$, $r=5$ and $C(V)=\Mat_3(\K)$.} \\
Note that $n \geq 6$, $r \geq 5$, $\codim C(V) \geq 0$, $n-p \geq 0$ and $\frac{\epsilon(\K)}{2}\geq 0$.
By Step 4, all those inequalities turn out to be equalities, which yields
the claimed result.

\vskip 3mm
\noindent \textbf{Step 6: $B(V)=\Mat_3(\K)$.} \\
As $n=p$, both the assumptions and the conclusion of the lemma we are trying to prove are invariant under transposing $V$.
Thus, applying Step 5 to $V^T$ yields the claimed result.

\vskip 3mm
\noindent \textbf{Step 7: $\urk B(V') \leq 2$.} \\
Note that $B(V') \subsetneq \Mat_3(\K)=B(V)$
as $B(V')$ is the translation vector space of a rank-$\overline{2}$ affine subspace of $\Mat_3(\K)$.
As $C(V)=\Mat_3(\K)$ is spanned by its rank $1$ matrices, we can choose a matrix $M_2 \in V$ such that
$\rk C(M_2)=1$ and $B(M_2) \not\in B(V')$. By \cite[Lemma 1.2]{dSPfullranklines} we can choose
$M'_1 \in V$ such that every matrix of $C(M'_1)+\K C(M_2)$ has rank $3$.
Then, we set
$$\calT':=\bigl\{M \in V : \; C(M) \in C(M'_1)+\K C(M_2)\bigr\},$$
which is an affine subspace of $V$ such that the translation vector space of $B(\calT')$ includes $B(V')$ as a \emph{proper} subspace.
Hence,
$$\dim B(\calT') \geq \dim B(V')+1 \geq (n-3)(r-3)-((n-3)-(r-2)+(r-3))+1.$$
Using the refined first classification theorem, we find that $B(\calT')$ is $2$-decomposable,
and hence $B(V')$ is $2$-decomposable.
The claimed result follows.

\vskip 3mm
\noindent \textbf{Step 8: $\urk B(V) \leq 2$.} \\
Let $M_2 \in V$. If $B(M_2) \in B(V')$ or $\rk C(M_2)=3$, then we already know that $\rk B(M_2) \leq 2$.
Assume now that $\rk C(M_2) \leq 2$ and $B(M_2) \not\in B(V')$ (so that $C(M_2) \neq 0$).
As in the previous step, \cite[Lemma 1.2]{dSPfullranklines} shows that we can
choose $M'_1 \in V$ such that every matrix of $C(M'_1)+\K C(M_2)$ has rank $3$. Then, we
consider $\calT':=\bigl\{M \in V : \; C(M) \in C(M'_1)+\K C(M_2)\bigr\}$, and we apply the refined first classification theorem to $B(\calT')$.
It follows that the translation vector space of $B(\calT')$, which contains $B(M_2)$, is $2$-decomposable.
Therefore, $\rk B(M_2) \leq 2$, which proves our claim.

\vskip 3mm
Obviously, the result of the last step contradicts the one of Step 6, which completes the proof.
\end{proof}

\subsection{Completing the special case when $n=4$, $r=3$ and $\# \K>2$}\label{4by4class}

\begin{prop}\label{4by4r=3}
Let $V$ be a rank-$\overline{3}$ linear subspace of $\Mat_4(\K)$. Assume that $\# \K>2$ and $\dim V \geq 8$.
Then, either $V$ is $3$-decomposable, or $\# \K=3$ and
$V$ is equivalent to $\calU_4(\K)$.
\end{prop}

\begin{proof}
If $V$ is equivalent to a subspace of $\calR(i,4-i)$ for some $i \in \{1,2,3\}$, then
the conclusion follows from one of lifting lemmas 3.1, 3.2 and 3.3.
In the rest of the proof, we assume that $V$ is not equivalent to a subspace of $\calR(i,4-i)$ for some $i \in \{1,2,3\}$.

\vskip 3mm
\noindent \textbf{Step 1: $V$ contains a matrix with rank $1$ or $2$.} \\
Assume on the contrary that all the non-zero matrices of $V$ have rank $3$.
Then, for all non-collinear vectors $x$ and $y$ in $\K^4$, the linear mapping
$\varphi_{x,y}: M \in V \mapsto (Mx,My) \in (\K^4)^2$ is injective, and as $\dim V \geq 8$ it must be an isomorphism.
Denote by $(e_1,e_2,e_3,e_4)$ the standard basis of $\K^4$. Using the surjectivity of $\varphi_{e_1,e_2}$, we
find some $M$ in $V$ with first column $e_1$.
Denote by $\calT$ the affine subspace of $V$ consisting of the matrices $M \in V$ such that $M e_1=e_1$.
Every matrix $M$ in $\calT$ splits as
$$M=\begin{bmatrix}
1 & [?]_{1 \times 3} \\
[0]_{3 \times 1} & P(M)
\end{bmatrix} \quad \text{with $P(M) \in \Mat_3(\K)$.}$$
Then, $P(\calT)$ is a rank-$\overline{2}$ affine subspace of $\Mat_3(\K)$.
However, $\dim P(\calT)=\dim \calT = 4$ since $V$ contains no rank $1$ matrix.
By Proposition \ref{n=3r=2}, $P(\calT)$ is $2$-decomposable, or it is equivalent to
$\calU_3(\K)$ and $\# \K=3$. In any case, $\dim P(\calT)z\leq 2$ for some non-zero vector $z \in \K^3$, to the effect that
there is a vector $y \in \K^4 \setminus \K e_1$ such that $\calT y \subsetneq \K^4$.
Thus, $\varphi_{e_1,y}$ is non-surjective, contradicting an earlier result.
Hence, $V$ contains a matrix with rank $1$ or $2$.

\vskip 3mm
\noindent \textbf{Step 2: $V$ contains a rank $1$ matrix.} \\
Assume that the contrary holds. Then, we can choose $M_0 \in V$ with rank $2$.
Without loss of generality, we can assume that
$$M_0=\begin{bmatrix}
[0]_{2 \times 2} & I_2 \\
[0]_{2 \times 2} & [0]_{2 \times 2}
\end{bmatrix}.$$
As $\# \K>2$, we get from the extraction lemma (see Corollary \ref{extractioncor}) that any $M \in V$ splits as
$$M=\begin{bmatrix}
[?]_{2 \times 2} & [?]_{2 \times 2}  \\
J(M) & [?]_{2 \times 2}
\end{bmatrix}$$
for some matrix $J(M) \in \Mat_2(\K)$ such that $\rk J(M) \leq 1$.
Thus, $J(V)$ is a rank-$\overline{1}$ linear subspace of $\Mat_2(\K)$, and hence
it is equivalent to a subspace of $\calR(0,1)$ or of $\calR(1,0)$.
If the second case holds, we see that the first case applies to $K V^T K$, where
$K:=\begin{bmatrix}
0 & 0 & 0 & 1 \\
0 & 0 & 1 & 0 \\
0 & 1 & 0 & 0 \\
1 & 0 & 0 & 0
\end{bmatrix}$. Thus, no generality is lost in assuming that $J(V)$ is equivalent to a subspace of $\calR(0,1)$,
and using column operations we can further reduce the situation to the one where the first column of every matrix of $J(V)$ equals zero.
As $V$ is not equivalent to a subspace of $\calR(1,3)$, we know that $\dim(Vx) \geq 2$ for all $x \in \K^4 \setminus \{0\}$.
Hence, we can find $M \in V$ such that $Me_1=e_1$.
Then, just like in Step 1 we consider the affine subspace $\calT:=\{M \in V : \; Me_1=e_1\}$ and its projection $P(\calT)$
onto the lower-right $3 \times 3$ block. This time around, we obtain that $\dim P(\calT) \geq 6$ since $\dim V e_1 \leq 2$.
Thus, by Flanders's theorem $P(\calT)$ is a \emph{linear} subspace
of $\Mat_3(\K)$. As $P(\calT)$ contains the zero matrix, there is a rank $1$ matrix in $\calT$.
This proves the claimed statement.

\vskip 3mm
Now, by Step 2, no generality is lost in assuming that $V$ contains $E_{1,1}$.
Then, by the extraction lemma, every matrix $M$ of $V$ splits as
$$M=\begin{bmatrix}
? & [?]_{1 \times 3} \\
[?]_{3 \times 1} & A(M)
\end{bmatrix}$$
for some $A(M) \in \Mat_3(\K)$ such that $\rk A(M) \leq 2$.

\vskip 3mm
\noindent \textbf{Step 3: $A(V)$ is equivalent to $\Mata_3(\K)$.} \\
Let us apply the classification of rank-$\overline{2}$ vector spaces to $A(V)$ (see the first paragraph from Section 4 of \cite{AtkinsonPrim}).
If, for some $i \in \{0,1,2\}$, the space $A(V)$ were equivalent to a subspace of $\calR(i,2-i)$, then
$V$ would be equivalent to a subspace of $\calR(i+1,3-i)$, contradicting our assumptions.
The claimed result follows.

\vskip 3mm
\noindent \textbf{Step 4: The final contradiction.} \\
Since $\dim V \geq 8$ and $\dim A(V)=3$, we get that the space of all matrices $M \in V$ such that $\im N \subset \im E_{1,1}$
has dimension at least $2$, and ditto for the space of all matrices $M \in V$ such that $\Ker E_{1,1} \subset \Ker N$.
As we have started from an arbitrary rank $1$ matrix of $V$, this can be generalized as follows:
\begin{center}
For every rank $1$ matrix $N \in V$, there exist rank $1$ matrices $N_1$ and $N_2$ in $V$ such that
 $$\Ker N_1 =\Ker N, \quad \im N_1 \neq \im N, \quad \im N_2=\im N\quad \text{and} \quad \Ker N_2 \neq \Ker N.$$
\end{center}
Then, we successively choose:
\begin{itemize}
\item A rank $1$ matrix $N_1 \in V$ such that $\im N_1=\im E_{1,1}$ and $\Ker N_1 \neq \Ker E_{1,1}$;
\item A rank $1$ matrix $N_2 \in V$ such that $\im N_2 \neq \im N_1$ and $\Ker N_2=\Ker N_1$.
\end{itemize}
It follows that $\Ker N_2 \neq \Ker E_{1,1}$ and $\im N_2 \neq \im E_{1,1}$, and as $N_2$ has rank $1$
we deduce that $A(N_2) \neq 0$. Since $A(V)$ is equivalent to $\Mata_3(\K)$ it follows that $\rk A(N_2)=2$.
Yet, $A(N_2)$, being a submatrix of a rank $1$ matrix, should have rank at most 1. This final contradiction completes the proof.
\end{proof}

\subsection{The last special lifting lemma}\label{lastspeciallifting}

\begin{prop}[Special lifting lemma 3]\label{specialliftingprop3}
Let $n>r$ be positive integers such that $r \geq 3$. Assume that $n \geq 5$ if $\# \K>2$.
Let $W$ be a linear subspace of $\Mat_{n,r}(\K)$ and $f : W \rightarrow \K^n$ be a linear map.
Assume that every matrix in
$$V:=\Bigl\{\begin{bmatrix}
f(N) & N
\end{bmatrix} \mid N \in W\Bigr\}$$
has rank at most $r$.
Assume furthermore that $\codim W \leq 2n-4-\epsilon(\K)$. \\
Then, $V$ is $r$-decomposable.
\end{prop}

The structure of the proof is essentially similar to the one of special lifting lemma 2.
The goal is to show that either $f$ is quasi-range-compatible, in which case we shall gather the result from the theory of quasi-range-compatible linear maps and from one of the lifting lemmas, or that $W$ can be reduced to a special form, in which case we shall use lifting lemmas 3.1 or 3.2.
We perform a \emph{reductio ad absurdum} by assuming that $V$ is not $r$-decomposable.
Throughout the proof, it will be important to note that $V$ satisfies the dimensional requirement from
lifting lemmas 3.1 and 3.2.

\vskip 3mm
Assume first that $f$ is quasi-range-compatible.
Then, by Theorem \ref{RCtheo3}, we have two cases to consider:

\begin{itemize}
\item \textbf{Case 1.} $f$ is local. Then, $V$ is equivalent to a subspace of $\calR(0,r)$, contradicting our assumptions.

\item \textbf{Case 2.}
There exist a (non-zero) vector $x \in \K^r$, a $2$-dimensional linear subspace $P$ of $\K^n$, a vector
$x' \in \K^r$ and an endomorphism $u$ of $P$ such that $Wx \subset  P$ and
$f : N \mapsto Nx'+u(Nx)$. Hence, $V$ is equivalent to a subspace of $\calR(2,r-1)$.
Then, lifting lemma 3.2 contradicts our assumptions.
\end{itemize}

In the rest of the proof, we assume that $f$ is not quasi-range-compatible.

As in Section \ref{specialliftingsection2}, we say that a non-zero linear form $\varphi$ on $\K^n$ is \textbf{bad}
when there exists $N \in W$ such that $\im N \subset \Ker \varphi$ and $\varphi(f(N))\neq 0$.

With the same proof as the one of Claim \ref{claim1}, the fact that $f$ is not quasi-range-compatible yields
the following result:

\begin{claim}\label{claim7}
There exists a basis of $(\K^n)^\star$ consisting of bad linear forms.
\end{claim}

Next, given a bad linear form $\varphi$, we set
$$W_\varphi:=\bigl\{N \in W : \; \im N \subset \Ker \varphi\bigr\}.$$

\begin{claim}\label{claim8}
Let $\varphi$ be a bad linear form. Then, $W_\varphi$ is a rank-$\overline{r-1}$ space.
\end{claim}

\begin{proof}
Let us consider the linear form
$$\gamma : N \in W_\varphi \mapsto \varphi(f(N)).$$
As $\varphi$ is bad, we know that $\gamma$ is non-zero.
Consider the affine hyperplane $\calU:=\gamma^{-1} \{1\}$ of $W_\varphi$.
For all $N \in \calU$, since $\varphi(f(N))=1$ and $\im N \subset \Ker \varphi$ we must have
$\rk N \leq r-1$. On the other hand, the rank theorem yields
$$\dim \calU \geq \dim W-r-1 \geq (n-1)(r-1)-(n-2)+\epsilon(\K).$$
As the matrices in $\calU$ have their range included in $\Ker \varphi$, the space $\calU$ can be represented
by a subspace of $\Mat_{n-1,r}(\K)$. Thus, we see that
either $\# \K>2$, in which case the refined first classification theorem applies to $\calU$
(and the exceptional situation where $\calU$ is equivalent to $\calU_3(\K)$ cannot occur since
in that case $n-1 \geq 4$), or $\# \K=2$, in which case the first classification theorem applies to $\calU$.
In any case, we obtain that $\calU$ is $(r-1)$-decomposable, and hence
it is also the case of $W_\varphi$ because $\calU$ spans it. Therefore, $\urk W_\varphi \leq r-1$.
\end{proof}

Let $\varphi$ be a bad linear form. The space $W_\varphi$ can be naturally identified with a linear subspace of $\calL(\K^r,\Ker \varphi)$
which, by choosing a basis of $\Ker \varphi$, can be represented by a subspace of $\Mat_{n-1,r}(\K)$.

\begin{claim}\label{claim9}
Let $\varphi$ be a bad linear form. Let $W_\varphi'$ be a subspace of $\Mat_{n-1,r}(\K)$
that represents $W_\varphi$.
Then, one and only one of the following results holds:
\begin{enumerate}[(1)]
\item $W'_\varphi$ is equivalent to a subspace of $\calR(0,r-1)$ ;
\item $W'_\varphi$ is equivalent to a subspace of $\calR(r-1,0)$, and $r \geq n-2$.
\item $W'_\varphi$ is equivalent to $\calR(1,r-2)$ and $\# \K>2$ ;
\item $W'_\varphi$ is equivalent to $\calR(r-2,1)$, and $r=n-1$ and $\# \K>2$.
\end{enumerate}
\end{claim}

\begin{proof}
Noting that
$$\dim W'_\varphi =\dim W_\varphi \geq \dim W-r \geq (n-1)(r-1)-(n-2)+1+\epsilon(\K),$$
the result follows either from the first classification theorem if $\# \K=2$, or from the refined one if $\# \K>2$.
\end{proof}

Given $i \in \{1,2,3,4\}$, we shall say that $\varphi$ has \textbf{type $i$} when it satisfies condition ($i$)
in the above claim. Thus:
\begin{itemize}
\item If $\varphi$ has type $1$ then there exists a non-zero vector $x$ of $\K^r$ such that
$\forall N \in W, \; \im N \subset \Ker \varphi \Rightarrow Nx=0$.
\item If $\varphi$ has type $2$ then there exists an $(r-1)$-dimensional linear subspace $Q$ of $\Ker \varphi$ such that
$\forall N \in W, \; \im N \subset \Ker \varphi \Rightarrow \im N \subset Q$.
\item If $\varphi$ has type $3$ then there exist a $2$-dimensional linear subspace $P$ of $\K^2$
and a $1$-dimensional linear subspace $D$ of $\Ker \varphi$ such that
$W_\varphi$ is the set of all matrices $N \in \Mat_{n,r}(\K)$ that map $\K^r$ into $\Ker \varphi$
and $P$ into $D$.
\item If $\varphi$ has type $4$ then $n=r+1$ and there exist a linear subspace
$Q$ of $\K^r$ and a linear subspace $R$ of $\Ker \varphi$
such that $\dim Q=r-1$, $\dim R=n-3$, and $W_\varphi$ is the set of all matrices of $\Mat_{n,r}(\K)$
that map $Q$ into $R$ and $\K^r$ into $\Ker \varphi$.
\end{itemize}

\begin{claim}\label{claim10}
There do not exist bad linear forms $\varphi_1$ and $\varphi_2$ such that $\varphi_1$ has type $1$ and $\varphi_2$ has type $2$.
\end{claim}

\begin{proof}
Assume that the contrary holds, and take $\varphi_1$ of type $1$ and $\varphi_2$ of type $2$.
Then, $n-2 \leq r \leq n-1$.
If $r=n-2$ then $W_{\varphi_2}$ is equivalent to $\calR(r-1,0)$ because of its dimension. Then, by following the same line of reasoning
as in Claim \ref{claim3}, we obtain a contradiction.
Thus, $r=n-1$.

Obviously, $\varphi_1$ and $\varphi_2$ are non-collinear, whence $\Ker \varphi_1 \cap \Ker \varphi_2$ has codimension $2$ in $\K^n$.

There exist a non-zero vector $x \in \K^r$ and a linear subspace $Q \subset \K^n$ with codimension $2$ which is included in $\Ker \varphi_2$ and
such that, for all $N \in W$,
$$\im N \subset \Ker \varphi_1 \Rightarrow Nx=0$$
and
$$\im N \subset \Ker \varphi_2 \Rightarrow \im N \subset Q.$$

We split the discussion into two subcases, whether $Q$ is included in $\Ker \varphi_1$ or not.
\begin{itemize}
\item Assume that $Q\subset \Ker \varphi_1$. \\
Then, for all $N \in W$ such that $\im N \subset \Ker \varphi_2$
we obtain that $\im N \subset Q$ and then $Nx=0$. It follows from the rank theorem that
$$\codim W \geq (n-1)+(r-1) = 2n-3.$$
\item Assume that $Q \not\subset \Ker \varphi_1$. \\
Let us choose a complementary subspace $H$ of $\K x$ in $\K^r$.
For all $N \in W$ such that $\im N \subset \Ker \varphi_1$
and $NH \subset \Ker \varphi_2$, we find $Nx=0$ and then $\im N \subset Q \cap \Ker \varphi_1$.
Since $Q \cap \Ker \varphi_1$ has codimension $3$ in $\K^n$, the rank theorem yields
$$\codim W \geq (n-1)+(r-1) = 2n-3.$$
\end{itemize}
In any case, we have contradicted our assumption on the codimension of $W$.
\end{proof}

Now, we examine the bad linear forms of type $3$.

\begin{claim}\label{claim11}
There is no bad linear form of type $3$.
\end{claim}

\begin{proof}
Assume on the contrary that there exists a bad linear form $\varphi$ of type $3$.
Then, we attach linear subspaces $P \subset \K^r$ and $D \subset \Ker \varphi$ to the type $3$ form $\varphi$.
Note that $n \geq 5$ since $\# \K>2$.

Here is our first step:

\vskip 3mm
\noindent \textbf{Step 1: Every bad linear form that is non-collinear to $\varphi$ has type $1$.} \\
Let $\psi$ be a bad linear form that is non-collinear to $\varphi$ and that does not have type $1$.

\begin{itemize}
\item \textbf{Case 1: $\psi$ has type $2$.} \\
Let us consider a linear subspace $Q$ of codimension $2$ such that
$$\forall N \in W, \; \im N \subset \Ker \psi \Rightarrow \im N \subset Q.$$
Let $y \in \Ker \psi \cap \Ker \varphi$. Then, we know that some $N \in W_\varphi$ has range $\K y$,
whence $y \in Q$. As $\Ker \psi \cap \Ker \varphi$ has codimension $2$ in $\K^n$, this leads to
$Q=\Ker \psi \cap \Ker \varphi$, and hence $Q \subset \Ker \varphi$.
Thus, for all $N \in W$, if $\im N \subset \Ker \psi$, then $\im N \subset Q$ and $N$ maps $P$ into $D$.
This leads to
$$\codim W \geq 2(n-2)+(r-2)>2n-4,$$
contradicting our assumptions.

\item \textbf{Case 2: $\psi$ has type $4$.} \\
Then, $n=r+1$ and we can attach subspaces $Q \subset \K^r$ and $R \subset \Ker \psi$ to the type $4$ form $\psi$.
Note that $\dim Q=r-1=n-2$.
Choose $x \in Q \setminus P$ (such a vector exists because $\dim Q \geq 3$ and $\dim P=2$).
Let $y \in \Ker \varphi \cap \Ker \psi$.
There is a matrix $N \in W_\varphi$ such that $Nx=y$ and $\im N=\K y$. Then, $N \in W_\psi$ and hence
$Nx \in R$. Thus $\Ker \varphi \cap \Ker \psi \subset R$, which is absurd because
$R$ has codimension $3$ in $\K^n$.

\item \textbf{Case 3: $\psi$ has type $3$.} \\
Then, we attach spaces $P'$ and $D'$ to the type $3$ form $\psi$.
We claim that $P=P'$. Indeed, assume that such is not the case, and choose $x \in P' \setminus P$.
Let $y \in \Ker \varphi \cap \Ker \psi$ be a non-zero vector.
We have a matrix $N \in W_\varphi$ such that $\im N=\K y$ and $Nx=y$.
Then, $N \in W_\psi$ and hence $y \in D'$. This is absurd because $\dim(\Ker \varphi \cap \Ker \psi) \geq 2$, owing to the assumption that $n \geq 4$.

Now, $W_\varphi+W_\psi$ contains every matrix $M$ that maps $P$ into $D+D'$.
If $D \neq D'$, then the subspace of all such matrices has codimension $2n-4$ in $\Mat_{n,r}(\K)$,
whence $W$ equals that space, which shows that $V$ is equivalent to a subspace of $\calR(2,r-1)$, thereby contradicting our assumptions.

Hence, $D=D'$ and $D \subset \Ker \varphi \cap \Ker \psi$.
Since $(\K^n)^\star$ has a basis of bad linear forms, we can choose a bad linear form $\chi$ such that $D \not\subset \Ker \chi$.
It follows from the previous study that $\chi$ must have type 1. Hence, there is a non-zero vector $x \in \K^r$
such that $\forall N \in W, \; \im N \subset \Ker \chi \Rightarrow Nx=0$.
Set $\calX:=\{z \in \K^r \mapsto \chi(Nz)\mid N \in W\}$, which is a linear subspace of $(\K^r)^\star$.
By the factorization lemma for linear maps,
we obtain a linear map $g : \calX \rightarrow \K^n$ such that $Nx=g(z \mapsto \chi(Nz))$ for all $N \in W$.
As $W_\varphi$ contains every matrix of $\Mat_{n,r}(\K)$ with range $D$, we successively find that
$$\calX=(\K^r)^\star=\{z \in \K^r \mapsto \chi(Nz)\mid N \in W_\varphi\}$$
and that $\im g \subset D$, whence $N x \in D$ for all $N \in W$. Therefore, $V$ is equivalent to a subspace of $\calR(1,r)$, which, by lifting lemma 3.1, contradicts our assumptions.
\end{itemize}
This proves the claimed result.

\vskip 3mm
Next, let $\psi$ be an arbitrary type $1$ bad linear form.
Then, we have a non-zero vector $x \in \K^r$ such that every matrix of $W_\psi$ vanishes at $x$.
We claim that $x\in P$. If not, we choose a non-zero vector $y \in \Ker \varphi \cap \Ker \psi$, and then we find some
$N \in W_\varphi$ such that $Nx=y$ and $\im N=\K y$; then, $N \in W_\psi$, which contradicts the fact that $Nx \neq 0$.

Now, we can extend $\varphi$ into a basis $(\varphi,\varphi_2,\dots,\varphi_n)$ of bad linear forms on $\K^n$.
We obtain non-zero vectors $x_2,\dots,x_n$ of $P$ such that, for all $i \in \lcro 2,n\rcro$, every matrix of
$W_{\varphi_i}$ vanishes at $x_i$.
Let $i \in \lcro 2,n\rcro$. For all $N \in W$, we set
$$N^{(i)} : z \in \K^r \mapsto \varphi_i(Nz),$$
we consider the subspace
$$S_i:=\bigl\{N^{(i)} \mid N \in W\bigr\} \subset (\K^r)^\star,$$
and we obtain a linear map
$$L_i : S_i \rightarrow \K^n$$
such that
$$\forall N \in W, \; L_i(N^{(i)})=Nx_i.$$

\noindent \textbf{Step 2: $\rk L_i=2$ for all $i \in \lcro 2,n\rcro$.} \\
Let $i \in \lcro 2,n\rcro$. Note that $\rk L_i=\dim W x_i$.
Assume that $\dim W x_i \leq 1$. Then, we successively obtain that $W$ is equivalent to a subspace of $\calR(1,r-1)$, that
$V$ is equivalent to a subspace of $\calR(1,r)$, and that $V$ is $r$-decomposable (by lifting lemma 3.1).
Hence, $\dim W x_i \geq 2$.
Next, let $g^\star \in (\K^r)^\star$ be a linear form that vanishes everywhere on $P$.
Let us choose $y \in \Ker \varphi \setminus \Ker \varphi_i$. We can find
$N \in W_\varphi$ such that $\Ker N=\Ker g^\star$ and $\im N=\K y$.
It follows that $N^{(i)}$ is a non-zero linear form whose kernel includes that of $g^\star$, and hence
$g^\star=(\lambda N)^{(i)}$ for some $\lambda \in \K$. Then, as $Nx_i=0$ we obtain that $L_i(g^\star)=0$.
Using the rank theorem, we deduce that $\rk L_i \leq 2$, which completes the proof.
\vskip 3mm
In particular, we have just shown that $\dim W x_i=2$ for all $i \in \lcro 2,n\rcro$.

\vskip 3mm
\noindent \textbf{Step 3: The vectors $x_2,\dots,x_n$ are pairwise collinear.} \\
Assume that the contrary holds. Without loss of generality, we can assume that $x_2$ and $x_3$ are non-collinear.
Then, as $\dim W x_2=2$, $\dim W x_3=2$ and $\codim W \leq 2n-4$, we get that $W$ is exactly the space of all matrices
$N \in \Mat_{n,r}(\K)$ such that $Nx_2 \in W x_2$ and $N x_3 \in W x_3$.
We can choose a non-zero vector $y \in Wx_2 \cap \Ker \varphi_2$.
There exists a rank $1$ matrix $N \in W$ such that $N x_2=y$, but then $\im N=\K y$
and hence $N \in W_{\varphi_2}$, leading to $Nx_2=0$. This is a contradiction.

\vskip 3mm
Now, we are ready to complete the proof.
Since we can safely replace $x_i$ with any collinear non-zero vector, no generality is lost in assuming that all the $x_i$ vectors are equal to some non-zero vector $x$ of $P$. Set
$$S:=\bigl\{z \mapsto \varphi(Nz) \mid N \in W\bigr\}.$$
Then, with the canonical projection $\pi : \K^n \rightarrow \K^n/D$, we have an additional linear map
$L : S \rightarrow \K^n/D$ such that
$$\forall N \in W, \; \pi(Nx)=L\bigl(z \mapsto \varphi(Nz)\bigr).$$
Note that $1 \leq \rk L$ since $Wx$ has dimension $2$. Then, with the same line of reasoning as in
the proof of Claim \ref{claim6} from Section \ref{specialliftingsection2}, we obtain
$$\dim W-2 \leq (n-1)(r-2)+(r-1)=nr-2n+1,$$
which contradicts our assumptions.
This final contradiction proves the claimed result.
\end{proof}

\begin{claim}\label{claim12}
There is no bad linear form of type $4$.
\end{claim}

\begin{proof}
Assume on the contrary that there exists a bad linear form $\varphi$ of type $4$. Note that this implies that $\# \K>2$, and hence
$n \geq 5$.
Then, $n=r+1$, $\codim W=2n-4$ and we can fix subspaces $Q \subset \K^r$ and $R \subset \Ker \varphi$ that are
attached to the type $4$ form $\varphi$.

By Claim \ref{claim7}, we can choose a bad linear form $\psi$ such that $R \not\subset \Ker \psi$ (as $\dim R>0$).

\vskip 3mm
\noindent \textbf{Step 1: $\psi$ does not have type $1$.} \\
As $\dim (R \cap \Ker \psi) \geq \dim R-1 \geq n-4>0$, we can choose a non-zero vector $y$ in $R \cap \Ker \psi$.
Then, $W_\varphi$ contains all the matrices of $\Mat_{n,r}(\K)$ with range $\K y$, and hence they all belong to $W_\psi$.
Obviously, this bars $\psi$ from having type $1$.

\vskip 3mm
\noindent \textbf{Step 2: $\psi$ does not have type $2$.} \\
Assume on the contrary that $\psi$ has type $2$.
Then, we have a subspace $S$ of $\K^n$ with codimension $2$
such that $\im N \subset S$ for all $N \in W_\psi$.
For each non-zero vector $y \in \Ker \varphi \cap \Ker \psi$, we can find a rank $1$ matrix $N \in W_\varphi$ with range $\K y$,
whence $N \in W_\psi$ and $y \in S$. As $\dim (\Ker \varphi \cap \Ker \psi)=n-2$, we deduce that $S=\Ker \varphi \cap \Ker \psi$.
It follows that for all $N \in W$, if $\im N \subset \Ker \psi$ then we successively obtain that $\im N \subset S$,
that $N \in W_\varphi$ and finally that $N$ maps $Q$ into $R$.
We deduce that
$$\codim W \geq 2r-1=2n-3,$$
contradicting our basic assumptions.

\vskip 3mm
It follows that $\psi$ must have type $4$.
Let us consider subspaces $Q' \subset \K^r$ and $R' \subset \Ker \psi$ attached to the type $4$ form $\psi$.
As $\Ker \psi$ does not include $R$, we have $R \neq R'$, and hence $\dim(R+R') \geq n-2$.

On the other hand, $W_\varphi$ (respectively, $W_\psi$) contains every matrix of $\Mat_{n,r}(\K)$ with range included in $R$ (respectively, in $R'$).
It follows that $W$ contains every matrix with range included in $R+R'$.

If $\dim(R+R') \geq n-1$, it would follow that $\codim W \leq r=n-1<2n-4$, contradicting earlier results.
Hence, $\dim(R+R')=n-2$.

\vskip 3mm
\noindent \textbf{Step 3: $Q=Q'$.} \\
Assume that we can find $x \in Q' \setminus Q$.
Let $y \in \Ker \varphi \cap \Ker \psi$. We can find $N \in W_\varphi$ with range $\K y$ such that $Nx=y$.
It follows that $y \in R$. Hence, $\Ker \varphi \cap \Ker \psi \subset R$, which contradicts $\dim R=n-3$.
Thus $Q' \subset Q$, and hence $Q=Q'$.

\vskip 3mm
We are ready to reach a final contradiction.
As $\K^n=\Ker \varphi+\Ker \psi$, every vector $y$ of $\K^n$ spans the range of some matrix
$N \in W$ with kernel $Q$. It follows that $W$ contains every matrix of $\Mat_{n,r}(\K)$ that
maps $Q$ into $R+R'$. As the space of all such matrices has codimension $2n-4$ in $\Mat_{n,r}(\K)$, we deduce that both spaces are equal,
whence $W$ is equivalent to $\calR(r-1,1)$. It follows that $V$ is equivalent to a subspace of $\calR(r-1,2)$.
Since $n=r+1$, lifting lemma 3.2 applies to $V^T$, which shows that $V^T$ is $r$-decomposable.
This contradicts the assumption that $V$ is not $r$-decomposable, thereby completing the proof.
\end{proof}

\begin{claim}\label{claim13}
There is no bad linear form of type $2$.
\end{claim}

\begin{proof}
Assume on the contrary that we can find a bad linear form $\varphi_1$ of type $2$.
By the above claims every bad linear form must be of type $2$.
It follows that $r=n-1$ or $r=n-2$.

\vskip 3mm
\noindent \textbf{Step 1: $W$ is not equivalent to $\calR(r,0)$.} \\
Assume that $W$ is equivalent to $\calR(r,0)$.
If $r=n-1$, then $V$ is equivalent to a subspace of $\calR(r,1)$, and lifting lemma 3.1
applied to $V^T$ would show that $V$ is $r$-decomposable, contradicting our assumptions.
It follows that $r=n-2$.
Then, without loss of generality we can assume that $W=\calR(r,0)$. For every $N \in \Mat_r(\K)$, denote by
$\beta(N)$ and $\gamma(N)$ the last two entries of the image of $\begin{bmatrix}
N \\
[0]_{2 \times r}
\end{bmatrix}$ under $f$. Then, $\beta$ and $\gamma$ are linear forms on $\Mat_r(\K)$
that vanish at every non-singular matrix. Yet, as $r \geq 2$, Flanders's theorem shows that
no affine hyperplane of $\Mat_r(\K)$ consists solely of singular matrices. Thus, $\beta=\gamma=0$.
It follows that $V \subset \calR(r,0)$, contradicting our assumption that $V$ is not $r$-decomposable.

\vskip 3mm
\noindent \textbf{Step 2: $r=n-1$.} \\
Assume on the contrary that $r=n-2$. Then, $\#\K>2$ and $\codim W=2n-4$.
Let us choose a bad linear form $\varphi_1$. Then, we have a linear subspace $Q_1$ of codimension $3$ in $\K^n$
such that $W_{\varphi_1}$ is the set of all matrices $N$ in $\Mat_{n,r}(\K)$ such that $\im N \subset Q_1$.
As $\dim Q_1>0$, Claim \ref{claim7} shows that we can find a bad linear form $\varphi_2$ such that $Q_1 \not\subset \Ker \varphi_2$,
yielding a linear subspace $Q_2$ of codimension $3$ in $\K^n$
such that $W_{\varphi_2}$ is the set of all matrices $N$ in $\Mat_{n,r}(\K)$ such that $\im N \subset Q_2$.
Then, $W$ contains every matrix with range included in $Q_1+Q_2$. However, $\dim (Q_1+Q_2) \geq n-2$ since
$Q_1 \neq Q_2$ and $\dim Q_1=\dim Q_2=n-3$. As $\codim W =2n-4$, we deduce that $\dim(Q_1+Q_2)=n-2$ and that
$W$ is exactly the space of all matrices $N \in \Mat_{n,r}(\K)$ such that $\im N \subset Q_1+Q_2$.
Thus, $W$ is equivalent to $\calR(r,0)$, contradicting Step 1.

\vskip 3mm
\noindent \textbf{Step 3: An inductive construction.} \\
To every bad linear form $\varphi$, we can attach a linear subspace $Q \subset \Ker \varphi$ with codimension $2$ in
$\K^n$ such that every matrix of $W_\varphi$ has its range included in $Q$.
Let $i \in \lcro 1,n-2\rcro$, and assume that there exist bad linear forms $\varphi_1,\dots,\varphi_{i-1}$,
with attached linear subspaces $Q_1,\dots,Q_{i-1}$ such that, for all $j \in \lcro 2,i-1\rcro$,
the space $\underset{k=1}{\overset{j-1}{\bigcap}} Q_k$ has codimension $j$ in $\K^n$ and is not included in $\Ker \varphi_j$,
and $\underset{k=1}{\overset{i-1}{\bigcap}} Q_k$ has codimension $i$ in $\K^n$.
By Claim \ref{claim7}, we can choose a bad linear form $\varphi_i$ whose kernel does not include
$\underset{k=1}{\overset{i-1}{\bigcap}} Q_k$, and denote by $Q_i$ an attached space.
Then, $\underset{k=1}{\overset{i}{\bigcap}} Q_k$ is a proper subspace of $\underset{k=1}{\overset{i-1}{\bigcap}} Q_k$
with codimension $1$ or $2$ (since $Q_i$ is included in $\Ker \varphi_i$).
Assume that $\underset{k=1}{\overset{i}{\bigcap}} Q_k$ has codimension $2$ in $\underset{k=1}{\overset{i-1}{\bigcap}} Q_k$;
then, it has codimension $i+2$ in $\K^n$; yet,
$$\forall N \in W, \; \im N \subset \underset{k=1}{\overset{i}{\bigcap}} \Ker \varphi_k
\Rightarrow \im N \subset \underset{k=1}{\overset{i}{\bigcap}} Q_k$$
and hence
$$\codim W \geq 2r = 2n-2,$$
contradicting our assumptions.
Thus, $\underset{k=1}{\overset{i}{\bigcap}} Q_k$ has codimension $i+1$ in $\K^n$.

Therefore, by induction (the case $i=1$ being obvious) we obtain bad linear forms
$\varphi_1,\dots,\varphi_{n-1}$, with attached spaces $Q_1,\dots,Q_{n-1}$
such that, for all $i \in \lcro 2,n-1\rcro$, the hyperplane $\Ker \varphi_i$ does not include
$Q_1 \cap \cdots \cap Q_{i-1}$, which has codimension $i$ in $\K^n$.

\vskip 3mm
\noindent \textbf{Step 4: The space $Q_1+Q_2$ is a hyperplane of $\K^n$ that includes all the $Q_i$ spaces.} \\
Since $\codim Q_1=\codim Q_2=2$ and $\codim (Q_1\cap Q_2)=3$, we find that $\codim (Q_1+Q_2)=1$.
In other words, $Q_1+Q_2$ is a linear hyperplane of $\K^n$.
Next, let $i \in \lcro 3,n\rcro$. Since $Q_1 \cap \cdots \cap Q_{i-1}$ is not included in $\Ker \varphi_i$,
$Q_1 \cap Q_2$ is not included in $Q_i$ either.
With the line of reasoning from Step 3 (more precisely, the inductive step at $i=2$),
we find that both $Q_1 \cap Q_i$ and $Q_2 \cap Q_i$ have codimension $3$ in $\K^n$, and hence they have codimension $1$ in $Q_i$.
Their intersection is $(Q_1 \cap Q_2) \cap Q_i$. Again, with the line of reasoning from Step 3, we see that
$(Q_1 \cap Q_2) \cap Q_i$ has codimension $4$ in $\K^n$, and hence it has codimension $2$ in $Q_i$.
Therefore, $Q_1 \cap Q_i$ and $Q_2 \cap Q_i$ are distinct linear
hyperplanes of $Q_i$, which yields $Q_i=(Q_1 \cap Q_i)+(Q_2 \cap Q_i)$.
Hence, $Q_i \subset Q_1+Q_2$, as claimed.

\vskip 3mm
Next, we choose a (non-zero) linear form $\chi$ on $\K^n$ such that
$$\Ker \chi=Q_1+Q_2.$$

\vskip 3mm
\noindent \textbf{Step 5: $Q_2\not\subset \Ker \varphi_1$.} \\
Assume on the contrary that $Q_2 \subset \Ker \varphi_1$. Then, for every $N \in W$,
if $\im N \subset \Ker \varphi_2$ we successively find that $\im N \subset Q_2$, and then $\im N \subset Q_1 \cap Q_2$.
This leads to
$$\codim W \geq 2r>2n-4,$$
contradicting our assumptions.

\vskip 3mm
\noindent \textbf{Step 6: $Q_1=\Ker \chi \cap \Ker \varphi_1$.} \\
As $Q_2$ is not included in $\Ker \varphi_1$, we have $\Ker \chi \neq \Ker \varphi_1$, and hence
$\Ker \chi \cap \Ker \varphi_1$ has dimension $n-2$. Since $\dim Q_1=n-2$ and $Q_1 \subset \Ker \chi \cap \Ker \varphi_1$,
the claimed result follows.

\vskip 3mm
\noindent \textbf{Step 7: The linear forms $\varphi_1,\dots,\varphi_{n-1},\chi$ constitute a basis of the dual space of $\K^n$.} \\
It suffices to prove that the intersection of their kernels equals zero.
Let $i \in \lcro 2,n-1\rcro$.
As $Q_1 \cap \cdots \cap Q_{i-1}$ is not included in $\Ker \varphi_i$,
the space $Q_1 \cap \cdots \cap Q_{i-1} \cap \Ker \varphi_i$ is a hyperplane of $Q_1 \cap \cdots \cap Q_{i-1}$, and hence it has codimension
$i+1$ in $\K^n$. On the other hand $Q_1 \cap \cdots \cap Q_{i-1} \cap \Ker \varphi_i$ includes
$Q_1 \cap \cdots \cap Q_i$, and hence those spaces are equal.
By induction, we deduce that
$$\{0\}=Q_1 \cap \cdots \cap Q_{n-1}=Q_1 \cap \Ker \varphi_2 \cap \cdots \cap \Ker \varphi_{n-1.}$$
Using Step 6, we conclude that $\{0\}=\Ker \chi \cap \Ker \varphi_1 \cap \Ker \varphi_2 \cap \cdots \cap \Ker \varphi_{n-1}$,
which yields the claimed result.

\vskip 3mm
Now, we are ready to conclude. Operating on rows, we lose no generality in assuming that
$\varphi_1,\dots,\varphi_{n-1},\chi$ are the canonical linear forms on $\K^n$.
For $M \in W$, we denote its rows by $R_1(M),\dots,R_n(M)$, so that
$$M=\begin{bmatrix}
R_1(M) \\
\vdots \\
R_n(M)
\end{bmatrix}.$$
As $\Ker \chi$ includes all the $Q_i$'s, we obtain, for all $i \in \lcro 1,n-1\rcro$, a linear mapping
$$L_i : R_i(W) \rightarrow R_n(W)$$
such that
$$\forall i \in \lcro 1,n-1\rcro, \; \forall M \in W, \quad R_n(M)=L_i(R_i(M)).$$
Finally, set $t:=\dim R_n(W)$, which is the common rank of the $L_i$ maps.
Then, we obtain a linear injection of $\Ker R_n$ into $\Ker L_1 \times \cdots \times \Ker L_{n-1}$,
so that
$$\dim W-t \leq (n-1)(r-t).$$
This yields
$$\dim W \leq nr-r-tn+2t=nr-2n+4+\bigl((n-3)+(2-n)t\bigr)$$
As $\dim W \geq nr-2n+4$, this yields $t=0$.
Hence, $W \subset \calR(r,0)$. Just like in the proof of Step 1, this leads to a contradiction. This completes the proof.
\end{proof}

It follows from the previous results that every bad linear form has type $1$.
Then, one checks that the proofs of Claims 5 and 6 from Section \ref{specialliftingsection2}
still hold under the assumption $\dim W \geq nr-2n+4$, which yields that
$V$ is equivalent to a subspace of $\calR(1,r)$.
A final contradiction then follows from lifting lemma 3.1.
This completes the proof of special lifting lemma 3.

\subsection{Wrapping the proof up}\label{wrapupsecond}

We are finally ready to complete the proof of the second classification theorem.
As in the proof of the other two classification theorems, we proceed by induction over $n,p,r$.
The case $r\leq 1$ is already known for all integers $n$ and $p$.
If $r=2$, it is known (see \cite{AtkLloydPrim} for fields with more than $2$ elements, and \cite{dSPprimitiveF2} for the field
with two elements) that a linear subspace $V$ of $\Mat_{n,p}(\K)$ with upper-rank $2$ is $2$-decomposable
whenever $\dim V\geq 4$ if $\# \K>2$, and $\dim V \geq 6$ if $\# \K=2$. This yields the second classification theorem in the case
when $r=2$.

In the rest of the proof, we let $n,p,r$ be non-negative integers such that $n \geq p \geq r \geq 3$.
Let $S$ be a rank-$\overline{r}$ linear subspace of $\Mat_{n,p}(\K)$ such that
$$\dim S \geq nr-2(n-p+r)+2+\epsilon(\K).$$
If $r=p$ we simply have $S\subset \calR(0,r)$.
Thus, in the rest of the proof, we assume that
$3 \leq r \leq p-1$. It follows that $4 \leq p \leq n$.

If $S$ were equivalent to $\calR(s,s)$ for some non-negative integer $s$ such that $2s \leq r$,
then $S$ would be $r$-decomposable. In the rest of the proof, we assume that this situation does not hold.
Thus, according to Lemma \ref{3rdkey}, we can find a $1$-dimensional linear subspace $D$ of $\K^n$
such that $\dim S^D \leq \lfloor \frac{r-1}{2}\rfloor$, or we can find a linear hyperplane
$H$ of $\K^p$ such that $\dim S_H \leq \lfloor \frac{r-1}{2}\rfloor$.
Note that $\lfloor \frac{r-1}{2}\rfloor \leq r-2$.

From there, we split the discussion into four main cases.
Remember from Remark \ref{liftingremark3} that if the conclusion of the second classification theorem holds
for $S^T$ then it also holds for $S$.

\subsubsection{Case 1: There exists a $1$-dimensional linear subspace $D$ of $\K^n$ such that $1 \leq \dim S^D \leq \lfloor \frac{r-1}{2}\rfloor$.}

We apply the ERC method.

Without loss of generality, we can assume that $D$ is spanned by the first vector of the standard basis of $\K^n$
and that $S^D$ contains $E_{1,1}$. Then, by the extraction lemma, every matrix $M \in S$
splits as
$$M=\begin{bmatrix}
? & [?]_{1 \times (p-1)} \\
[?]_{(n-1) \times 1} & P(M)
\end{bmatrix},$$
where $P(S)$ is a rank-$\overline{r-1}$ linear subspace of $\Mat_{n-1,p-1}(\K)$.
Then,
$$\dim P(S) \geq \dim S-(n-1)-\dim S^D.$$
One checks that
\begin{multline*}
\bigl(nr-2(n-p+r)+2+\epsilon(\K)\bigr)-\bigl((n-1)(r-1)-2((n-1)-(p-1)+(r-1)\bigr)+2+\epsilon(\K)\bigr) \\
=(n-1)+(r-2).
\end{multline*}
As $\dim S^D \leq r-2$,
it follows that
$$\dim P(S) \geq (n-1)(r-1)-2\bigl((n-1)-(p-1)+(r-1))+2+\epsilon(\K).$$
Hence, the induction hypothesis applies of $P(S)$.
Moreover, since $\dim S^D \leq  \lfloor \frac{r-1}{2}\rfloor$, the above inequality on $\dim P(S)$ can be
sharpened so as to discard the following ``exceptional" cases from Remark \ref{2ndremark}:
\begin{enumerate}[(i)]
\item $n \geq p+1$, $r=6$ and $P(S)$ is equivalent to a linear subspace of $\calR(3,2)$ with codimension at most $1$;
\item $n=p+1$, $r=7$ and $P(S)$ is equivalent to $\calR(3,3)$ or to $\calR(4,2)$;
\item $n=p$, $r=7$ and $P(S)$ is equivalent to a subspace of $\calR(3,3)$ with codimension at most $1$;
\item $n=p$, $r=8$ and $P(S)$ is equivalent to $\calR(4,3)$ or $\calR(3,4)$;
\item $(n,p,r)=(5,5,4)$, $\# \K=3$ and $P(S)$ is equivalent to $\calU_4(\K)$.
\end{enumerate}

Thus, according to Remark \ref{2ndremark}, only four cases remain to be considered:
\begin{itemize}
\item \textbf{Subcase 1.1: $P(S)$ is equivalent to a subspace of $\calR(i,r-1-i)$ for some $i \in \{0,1\}$.} \\
Then, $S$ is equivalent to a subspace of $\calR(i+1,r-i)$ for some $i \in \{0,1\}$, and the conclusion
follows from lifting lemma 3.1 or from lifting lemma 3.2.
\item \textbf{Subcase 1.2: $P(S)$ is equivalent to a subspace of  $\calR(r-1-i,i)$ for some $i \in \{0,1\}$, and $r \geq 4$ or $n=p$.} \\
Then, $S$ is equivalent to a subspace of $\calR(r-i,i+1)$ for some $i \in \{0,1\}$, and the
conclusion follows from lifting lemma 3.1 or from lifting lemma 3.2 applied to $S^T$.
\item \textbf{Subcase 1.3: $P(S)$ is equivalent to a subspace of $\calR(2,r-3)$, and $r=3$ or $n \geq 6$.} \\
Then, $S$ is equivalent to a subspace of $\calR(3,r-2)$ and the conclusion follows from lifting lemma 3.3.
\item \textbf{Subcase 1.4: $P(S)$ is equivalent to a subspace of $\calR(r-3,2)$, $n=p$, and $n \geq 6$.} \\
Then, $S$ is equivalent to a subspace of $\calR(r-2,3)$ and the conclusion follows from lifting lemma 3.3 applied to $S^T$.
\end{itemize}

\subsubsection{Case 2: There exists a linear hyperplane $H$ such that $1 \leq \dim S_H \leq \lfloor \frac{r-1}{2}\rfloor$.}

Again, we apply the ERC method.
This time around, we obtain
$$\dim P(S) \geq \dim S-(p-1)-\dim S_H,$$
and we can follow the same line of reasoning as in Case 1 because $p \leq n$.

\subsubsection{Case 3: There exists a linear hyperplane $H$ of $\K^p$ such that $S_H=\{0\}$.}

We apply the EC method.
Without loss of generality, we can assume that $H=\{0\} \times \K^{p-1}$.
Then, we split every matrix $M \in S$ as
$$M=\begin{bmatrix}
[?]_{n \times 1} & J(M)
\end{bmatrix},$$
and $J(S)$ is a rank-$\overline{r}$ linear subspace of $\Mat_{n,p-1}(\K)$.
Now,
$$\dim J(S)=\dim S \geq nr-2\bigl(n-(p-1)+r\bigr)+4+\epsilon(\K).$$
Thus, by induction we know that $J(S)$ is $r$-decomposable.
Note that $n>p-1$. Moreover, it is easily seen that the following ``exceptional" cases
can be discarded thanks to the improved lower bound on $\dim J(S)$:
\begin{enumerate}[(i)]
\item $n \geq p$, $r=5$ and $J(S)$ is equivalent to a linear subspace of $\calR(3,2)$ with codimension at most $1$;
\item $n=p$, $r=6$ and $J(S)$ is equivalent to $\calR(3,3)$ or to $\calR(4,2)$;
\end{enumerate}

This leaves us with only four subcases to consider.

\vskip 3mm
\noindent \textbf{Subcase 3.1: $J(S)$ is equivalent to a subspace of $\calR(1,r-1)$ or $\calR(2,r-2)$.} \\
Then, $S$ is equivalent to a subspace of $\calR(1,r)$ or $\calR(2,r-1)$, and we conclude
by applying lifting lemma 3.1 or lifting lemma 3.2.

\vskip 3mm
\noindent \textbf{Subcase 3.2: $J(S)$ is equivalent to a subspace of $\calR(r,0)$ or of $\calR(r-1,1)$.} \\
Then, $S$ is equivalent to a subspace of $\calR(r,1)$ or $\calR(r-1,2)$.
If $r \geq 4$, then one of lifting lemmas 3.1 and 3.2 applies to $S^T$, otherwise
one of lifting lemmas 3.2 or 3.3 directly applies to $S$.
In any case, we obtain the expected conclusion.

\vskip 3mm
\noindent \textbf{Subcase 3.3: $J(S)$ is equivalent to a subspace of $\calR(0,r)$, and $r=p-1$.} \\
Then, with $W:=J(S)$, we find a linear map $f : W \rightarrow \K^n$
such that
$$S=\Bigl\{\begin{bmatrix}
f(N) & N
\end{bmatrix} \mid N \in W \Bigr\}.$$
As $r=p-1$ we see that
$$\codim W \leq 2n-4-\epsilon(\K).$$
If $n \geq 5$ or $\# \K=2$, the conclusion follows from special lifting lemma 3.
Otherwise, we must have $n=p=4$, $r=3$ and $\# \K>2$, and the conclusion follows directly from Proposition \ref{4by4r=3}.

\vskip 3mm
\noindent \textbf{Subcase 3.4: $J(S)$ is equivalent to a subspace of $\calR(0,r)$, and $r<p-1$.} \\
Then, $S$ is equivalent to a subspace of $\calR(0,p-1)$.
Thus, there exists a rank-$\overline{r}$ linear subspace $V$ of $\Mat_{n,p-1}(\K)$
such that $S$ is equivalent to the space $\widetilde{V}$ of all matrices of the form
$\begin{bmatrix}
N & [0]_{n \times 1}
\end{bmatrix}$ with $N \in V$. As $\dim V=\dim S$ and $n>p-1$, we see by induction that $V$
is $r$-decomposable, and it follows that $\widetilde{V}$ is $r$-decomposable. Hence, $S$ is $r$-decomposable.

\subsubsection{Case 4: There exists a $1$-dimensional linear subspace $D$ of $\K^n$ such that $S^D=\{0\}$.}

If $n=p$, then Case 3 applies to $S^T$, and we obtain the expected conclusion.
In the remainder of the proof, we assume that $n>p$.

Then, we apply the ER method.
Without loss of generality, we can assume that $D$ is spanned by the first vector of the standard basis of $\K^n$.
Then, we split every matrix $M \in S$ up as
$$M=\begin{bmatrix}
[?]_{1 \times p} \\
A(M)
\end{bmatrix},$$
and $A(S)$ is a rank-$\overline{r}$ linear subspace of $\Mat_{n-1,p}(\K)$.
We still have $n-1 \geq p$, but now
$$\dim A(S)=\dim S \geq (n-1)r-2\bigl((n-1)-p+r\bigr)+(2+\epsilon(\K))+(r-2).$$
Thus, the induction hypothesis can be applied to $A(S)$.
Again, thanks to the improved lower-bound on $A(S)$, we can discard the following exceptional cases from Remark \ref{2ndremark}:
\begin{enumerate}[(i)]
\item $n \geq p+2$, $r=5$ and $A(S)$ is equivalent to a linear subspace of $\calR(3,2)$ with codimension at most $1$;
\item $n=p+2$, $r=6$ and $A(S)$ is equivalent to $\calR(3,3)$ or to $\calR(4,2)$;
\item $n=p+1$, $r=6$ and $A(S)$ is equivalent to a subspace of $\calR(3,3)$ with codimension at most $1$;
\item $n=p+1$, $r=7$ and $A(S)$ is equivalent to $\calR(4,3)$ or $\calR(3,4)$;
\item $(n,p,r)=(5,4,4)$, $\# \K=3$ and $A(S)$ is equivalent to $\calU_4(\K)$.
\end{enumerate}

Hence, only the following remaining cases need to be considered.
\vskip 3mm
\noindent \textbf{Subcase 4.1: $A(S)$ is equivalent to a subspace of $\calR(i,r-i)$ for some $i \in \{0,1,2\}$.} \\
Then, $S$ is equivalent to a subspace of $\calR(i+1,r-i)$.
Noting that $n \geq p+1 \geq r+2$, we obtain that $n \geq 6$ whenever $r \geq 4$. Hence,
one of lifting lemmas 3.1, 3.2 or 3.3 applies to $S$, which yields the excepted conclusion.

\vskip 3mm
\noindent \textbf{Subcase 4.2: $r \geq 4$ and $A(S)$ is equivalent to a subspace of $\calR(r-i,i)$ for some $i \in \{1,2\}$.} \\
Then, $S$ is equivalent to a subspace of $\calR(r+1-i,i)$, and, since $r \geq 4$, one of lifting lemmas 3.1 or 3.2
applies to $S^T$, yielding that $S$ is $r$-decomposable.

\vskip 3mm
\noindent \textbf{Subcase 4.3: $A(S)$ is equivalent to a subspace of $\calR(r,0)$.} \\
Then, as $r<p$, we see that $A(S)$ is equivalent to a subspace of $\calR(n-1,0)$.
This yields a linear subspace $V$ of $\Mat_{n-1,p}(\K)$ such that
$S$ is equivalent to the space $\widetilde{V}$ of all matrices of the form
$\begin{bmatrix}
N \\
[0]_{1 \times p}
\end{bmatrix}$ with $N \in V$.
Then, as $\dim V=\dim S$, we obtain by induction that
$V$ is $r$-decomposable: indeed, the only exceptional solution would be the case when $n=5$, $p=4$, $r=3$, $\# \K=3$ and $V$ is equivalent to
$\calU_4(\K)$, but this case has been previously discarded.

\vskip 3mm
Our case-by-case study is now completed. Hence, the second classification theorem is finally established.

\section{Conclusion, or where to go next}

As we have demonstrated through the examples in our introduction, the second classification theorem is optimal in the full-rank case, i.e.\ in the case when $r=p-1$: right under the lower bound from that theorem, we can find a rank-$\overline{r}$ space that is not $r$-decomposable. Under this bound, we expect that the maximal subspaces should become increasingly complicated, even for algebraically closed fields. Anyway,
the methods that we have used in this article have clearly reached their limit.

In this short section, we wish to point to a possible direction for further research on the topic.
Let us start with a simple observation: let $s$ be a positive integer such that $s<p-1$.
Let $n \geq m \geq s+1$ and $W$ be a linear subspace of $\Mat_{m,s+1}(\K)$ with upper-rank $s$ which is not $s$-decomposable.
Then, $V:=W \vee \Mat_{n-m,p-s-1}(\K)$ has upper-rank $p-1$ but it is not $(p-1)$-decomposable.
Note that
$$\dim V \geq n(p-s-1) \quad \text{and} \quad n(p-s-1)=n(p-1)-sn.$$
Thus, for small values of $s$ we obtain a dimension that is substantially larger
than the lowest dimension among the rank-$\overline{p-1}$ compression spaces.
Moreover, one checks that the equivalence class of $V$ determines that of $W$. Thus,
finding whole new ranges of dimensions for which the classification of rank-$\overline{p-1}$ subspaces is known
would require a rather precise understanding of the structure of all rank-$\overline{s}$ subspaces for \emph{small} values of $s$.
Incidentally, our proof of the second classification theorem did involve some results from the classification of rank-$\overline{2}$ spaces.

Let us reframe the problem by using the notion of a \textbf{primitive} bounded rank space.
Let us recall the definition from \cite{AtkLloydPrim}:

\begin{Def}
Let $n,p,r$ be non-negative integers with $r \leq \min(n,p)$, and
$\calV$ be a linear subspace of $\Mat_{n,p}(\K)$ with upper-rank $r$.
We say that $\calV$ is \textbf{primitive} when
it satisfies the following four conditions:
\begin{enumerate}[(i)]
\item It is not equivalent to a subspace of $\calR(0,p-1)$.
\item It is not equivalent to a subspace of $\calR(n-1,0)$.
\item There does not exist a space $\calT$ that is equivalent to $\calV$ and in which
we can write every matrix $N$ as $N=\begin{bmatrix}
H(N) & [?]_{n \times 1}
\end{bmatrix}$ with $H(\calT)$ a rank-$\overline{r-1}$ subspace of $\Mat_{n,p-1}(\K)$.
\item There does not exist a space $\calT$ that is equivalent to $\calV$ and in which
we can write every matrix $N$ as $N=\begin{bmatrix}
H(N) \\
[?]_{1 \times p}
\end{bmatrix}$ with $H(\calT)$ a rank-$\overline{r-1}$ subspace of $\Mat_{n-1,p}(\K)$.
\end{enumerate}
\end{Def}

Roughly, a primitive space with upper-rank $r$ is one which cannot be obtained from
another bounded rank space with fewer rows or columns by an ``obvious" completion method.
Note that any space with upper-rank $r$ that includes a primitive one is primitive.
Moreover, it is obvious that no compression space is primitive (except the one of all $0$ by $0$ matrices!).

Primitive spaces are connected to general spaces through the following result.

\begin{prop}[See Theorem 1 of \cite{AtkLloydPrim}]
Let $V$ be a linear subspace of $\Mat_{n,p}(\K)$ with upper-rank $r$.
Then, there are integers $s,t,s',t'$ such that $s+t \leq \min(p-t',n-s')$
together with a primitive linear subspace
$W$ of $\Mat_{s',t'}(\K)$ such that:
\begin{enumerate}[(i)]
\item The space $V$ is equivalent to a subspace of the space of all matrices of the form
$$\begin{bmatrix}
[?]_{s \times t} & [?]_{s \times t'} & [?]_{s \times (p-t-t')} \\
[?]_{s' \times t} & N & [0]_{s' \times (p-t-t')} \\
[?]_{(n-s-s') \times t} & [0]_{(n-s-s') \times t'} & [0]_{(n-s-s') \times (p-t-t')}
\end{bmatrix}$$
with $N \in W$.
\item The space $W$ has upper-rank $r-(s+t)$.
\end{enumerate}
\end{prop}

Moreover, if $V$ is a \emph{maximal} linear subspace of $\Mat_{n,p}(\K)$ with upper-rank $r$,
then it is easily shown that $W$ is a maximal primitive linear subspace of $\Mat_{s',t'}(\K)$ with upper-rank $r-s-t$,
and its equivalence class is uniquely determined by that of $V$.

Thus, determining maximal bounded rank spaces amounts to classifying the maximal primitive ones.
On the other hand, the classification theorems of the present article can be roughly restated (for linear subspaces) as saying that
a space of matrices with upper-rank $r$ is non-primitive as long as its dimension is large enough.

Thus, in our view, the ultimate challenge consists in tackling the following issue:
\begin{center}
What is the maximal dimension for a (maximal) primitive subspace of $\Mat_{n,p}(\K)$
with upper-rank $r$ (if such subspaces exist)?
\end{center}

When $p$ is very large with respect to $n$, or vice versa,
this dimension is known to equal $r+1$ provided that $\K$ has more than $r$ elements and such subspaces exist (see \cite{AtkinsonPrim}). On the other hand, when $n$ is close to $p$, very large primitive subspaces of bounded rank matrices exist: when $n$ is odd, a key example is
the one of the space $\Mata_n(\K)$ of all $n$ by $n$ alternating matrices with entries in $\K$;
its upper rank equals $n-1$, it is easy to prove that it is primitive, and on the other hand it is maximal
(see Proposition 5 of \cite{FillmoreLaurieRadjavi} for fields with more than $2$ elements, and Proposition 3.7 of \cite{dSPprimitiveF2}
for the field with $2$ elements).
A reasonable conjecture would be that, when $n$ is odd, $\dbinom{n}{2}$ is the maximal dimension for a
primitive subspace of $\Mat_n(\K)$ with upper-rank $\overline{n-1}$, provided that $\# \K \geq n$.
We believe that profound new insights are needed to prove such a result if it happens to be true.

\end{document}